\documentclass[12pt]{amsart}
\usepackage{a4wide}

\usepackage{times}
\usepackage{amssymb}
\usepackage{graphicx,xspace}
\usepackage{epsfig}
\usepackage[usenames,dvipsnames]{xcolor}
\usepackage{tikz}
\usepackage[T1]{fontenc}
\usepackage[utf8]{inputenc}
\usepackage{bm}
\usepackage[boxsize=1em]{ytableau}
\usepackage{cancel}
\usepackage{hyperref,pifont}
\usepackage{todonotes}
\makeatletter
\def\@tocline#1#2#3#4#5#6#7{\relax
  \ifnum #1>\c@tocdepth 
  \else
    \par \addpenalty\@secpenalty\addvspace{#2}%
    \begingroup \hyphenpenalty\@M
    \@ifempty{#4}{%
      \@tempdima\csname r@tocindent\number#1\endcsname\relax
    }{%
      \@tempdima#4\relax
    }%
    \parindent\z@ \leftskip#3\relax \advance\leftskip\@tempdima\relax
    \rightskip\@pnumwidth plus4em \parfillskip-\@pnumwidth
    #5\leavevmode\hskip-\@tempdima #6\nobreak\relax
    \dotfill\hbox to\@pnumwidth{\@tocpagenum{#7}}\par 
    \nobreak
    \endgroup
  \fi}
\makeatother

\usepackage[capitalize]{cleveref}

\theoremstyle{plain}

\newtheorem{lemma}{Lemma}[section]
\newtheorem{theorem}[lemma]{Theorem}
\newtheorem{corollary}[lemma]{Corollary}
\newtheorem{proposition}[lemma]{Proposition}
\newtheorem{definition}[lemma]{Definition}

\theoremstyle{remark}
\newtheorem{remark}[lemma]{Remark}
\newtheorem{example}[lemma]{Example}
\newtheorem{observation}[lemma]{Observation}

\def\eps{\varepsilon}

\def\Si{\bm{\sigma}}
\def\Mu{\bm{\mu}}

\def\si{\sigma}

\def\Sn{\mathfrak{S}}

\def\xx{\bm{x}}
\def\yy{\bm{y}}
\def\XX{\vec{\mathbf{x}}}
\def\YY{\vec{\mathbf{y}}}

\def\TT{\bm{T}}
\def\ll{\vec{\ell}}
\def\R{\mathbb{R}}

\def\CCC{\mathcal{C}}
\def\TTT{\mathcal{T}}
\def\SSS{\mathcal{S}}
\def\proba{\mathbb{P}}
\def\esper{\mathbb{E}}

\DeclareMathOperator{\db}{db}

\def\patterntree{t_0}
\def\One{\bm{1}}
\def\nonp{{\scriptscriptstyle{\mathrm{not}} {\oplus}}}
\def\nonm{{\scriptscriptstyle{\mathrm{not}} {\ominus}}}
\def\W{W}
\def\O{\mathcal{O}}
\def\cst{\text{cst}}
\def\racine{\sqrt{1-\tfrac z \rho}}
\def\lineaire{(1-\tfrac z \rho)}
\def\Mk{{\vec{\mathbf{m}}_k}}
\def\Mn{{\vec{\mathbf{m}}_n}}
\def\Cbase{\gamma}
\def\Const{B}
\newcommand{\Internal}[1]{\mathrm{Int}(#1)}

\DeclareMathOperator{\id}{id}
\DeclareMathOperator{\occ}{\widetilde{occ}}
\DeclareMathOperator{\perm}{perm}
\DeclareMathOperator{\Perm}{Perm}

\DeclareMathOperator{\pat}{pat}

\DeclareMathOperator{\Cat}{Cat}

\DeclareMathOperator{\Occ}{Occ}

\DeclareMathOperator{\Arg}{Arg}

\newcommand{\set}[1]{\left\{#1\right\}}
\newcommand\restr[2]{{%
		\left.\kern-\nulldelimiterspace %
		#1 %
		\right|_{#2} %
	}}

\def\Av{\mathrm{Av}}
\def\FCA{\mathrm{FCA}(\ll)}

\title{Universal limits of substitution-closed permutation classes}

\author[F. Bassino]{Frédérique Bassino}
       \address[FB]{Université Paris 13, Sorbonne Paris Cité, LIPN, CNRS UMR 7030, F-93430 Villetaneuse, France}
       \email{bassino@lipn.univ-paris13.fr}

 \author[M. Bouvel]{Mathilde Bouvel}
 \author[V. Féray]{Valentin Féray}
       \address[MB,VF]{Institut für Mathematik, Universität Zürich, Winterthurerstr. 190, CH-8057 Zürich, Switzerland}
       \email{mathilde.bouvel@math.uzh.ch}
       \email{valentin.feray@math.uzh.ch}

 \author[L. Gerin]{Lucas Gerin}
       \address[LG]{CMAP, \'Ecole Polytechnique, CNRS, Route de Saclay, F-91128 Palaiseau Cedex, France}
       \email{gerin@cmap.polytechnique.fr}

  \author[M. Maazoun]{Mickaël Maazoun}
   \address[MM]{École Normale Supérieure de Lyon, UMPA UMR 5669 CNRS, 46 allée d’Italie, F-69364 Lyon Cedex 07, France}
   \email{mickael.maazoun@ens-lyon.fr}

 \author[A. Pierrot]{Adeline Pierrot}
 \address[AP]{LRI, Université Paris-Sud, Bat. 650 Ada Lovelace, F-91405 Orsay Cedex, France}
       \email{adeline.pierrot@lri.fr}



\begin{document}

\begin{abstract}
We consider uniform random permutations in proper substitution-closed classes
and study their limiting behavior in the sense of permutons.

The limit depends on the generating series of the simple permutations in the class.
  Under a mild sufficient condition, the limit is an elementary one-parameter deformation
    of the limit of uniform separable permutations, previously identified as
    the Brownian separable permuton. This limiting object is therefore in some sense universal.
    We identify two other regimes with different limiting objects.
    The first one is degenerate; the second one is nontrivial and related to stable trees.

These results are obtained thanks to a characterization of the convergence
of random permutons
through the convergence of their expected pattern densities.
The limit of expected pattern densities is then computed by using the substitution tree encoding of permutations and performing singularity analysis on the tree series.
\end{abstract}

\maketitle

\tableofcontents
\section{Introduction}

The aim of this paper is to study the asymptotic behavior of a permutation of large size, 
picked uniformly at random in a \emph{substitution-closed permutation class} generated by a given (finite or infinite) family of \emph{simple permutations}
satisfying additional conditions. 
We first give a few definitions necessary to present the recent literature on related problems, and to state our results. 

\subsection{Permutation classes and their limit}\label{sec:PermutationClass}

For any positive integer $n$, the set of permutations of $[n]:= \{1,2,\ldots, n\}$ is denoted by $\Sn_n$. 
We write permutations of $\Sn_n$ in one-line notation as $\sigma = \sigma(1) \sigma(2) \dots \sigma(n)$. 
For a permutation $\sigma$ in $\Sn_n$, the \emph{size} $n$ of $\sigma$ is denoted by $|\sigma|$. 

For $\sigma\in \Sn_n$, and $I\subset [n]$ of cardinality $k$, let $\pat_I(\sigma)$ be the permutation of $\Sn_k$ induced by $\{\sigma(i) : i\in I\}$. 
For example for $\sigma=65831247$ and $I=\{2,5,7\}$ we have
$$
\pat_{\{2,5,7\}}\left(6\mathbf{5}83\mathbf{1}2\mathbf{4}7\right)=312
$$
since the values in the subsequence $\sigma(2) \sigma(5) \sigma(7)=514$ are in the same relative order as in the permutation $312$. 
A permutation $\pi = \pat_I(\sigma)$ is a \emph{pattern} involved (or contained) in $\sigma$, 
and the subsequence $(\sigma(i))_{i \in I}$ is an \emph{occurrence} of $\pi$ in $\sigma$. 
When a pattern $\pi$ has no occurrence in $\sigma$, we say that $\sigma$ \emph{avoids} $\pi$. 
The pattern containment relation defines a partial order on $\mathfrak{S} = \cup_n \mathfrak{S}_n$: 
we write $\pi \preccurlyeq \sigma$ if $\pi$ is a pattern of $\sigma$.

A \emph{permutation class} is a family $\mathcal{C}$ of permutations that is downward closed for $\preccurlyeq$, \emph{i.e.} 
for any $\sigma \in \mathcal{C}$ and any pattern $\pi \preccurlyeq \sigma$, it holds that $\pi \in \mathcal{C}$. 
For every set $B$ of patterns, we denote by $\Av(B)$ the set of all permutations that avoid every pattern in $B$. 
Clearly, for all $B$, $\Av(B)$ is a permutation class. 
Conversely (see for instance~\cite[Paragraph 5.1.2]{BonaBook}), every class $\mathcal{C}$ of permutations can be defined by a set $B$ of excluded patterns. 
Moreover, for any given $\mathcal{C}$, we can define uniquely a set $B$ such that $\mathcal{C}=\Av(B)$: 
it is enough to impose that $B$ is chosen minimal (for set inclusion) among all $B'$ such that $\mathcal{C}=\Av(B')$. 
This $B$ (which happens to be an antichain) is called the \emph{basis} of the class $\mathcal{C}$. 
The basis of a permutation class may be finite or infinite (see~\cite[Paragraph 7.2.3]{BonaBook}). 

One in many ways permutation classes can be studied is by looking at the features of a typical large permutation $\sigma$ in the class. 
A particularly interesting characteristic is the frequency of occurrence of a pattern $\pi$, especially when it is considered for all $\pi$ simultaneously.
Denote by $\mathrm{occ}(\pi,\sigma)$ the number of occurrences of a pattern $\pi\in\Sn_k$ in $\sigma \in \Sn_n$
and by $\occ(\pi,\sigma)$ the {\em pattern density} of $\pi$ in $\sigma$.
More formally
\begin{align}
\mathrm{occ}(\pi,\sigma) &= \mathrm{card}\{I \subset [n] \text{ of cardinality }k \text{ such that } \pat_I(\sigma)=\pi\} \notag\\ %
\occ(\pi,\sigma) &= \frac{\mathrm{occ}(\pi,\sigma)}{\binom{n}{k}} \, 
=\mathbb{P} \left(\pat_{\bm I}(\sigma)=\pi \right),\label{eq:EsperanceOccurrence}
\end{align}
where $\bm I$ is randomly and uniformly chosen among the $\binom{n}{k}$ subsets of $[n]$ with $k$ elements.
The study of the asymptotics of $\occ(\pi,\bm \sigma_n)$, where $\bm{\sigma}_n$ is a uniform random permutation of size $n$ in a permutation class $\mathcal C$ and $\pi\in \Sn$ is a fixed pattern, has been carried out in several cases. 
\begin{itemize}
	\item  The behavior of $\mathbb{E}[\occ(\pi,\bm\sigma_n)]$ for various classes $\mathcal C$ and  fixed $\pi$ was investigated by B\'ona~\cite{Bona1,Bona2}, Homberger \cite{Homberger},
	Chang, Eu and Fu \cite{ChenEuFu} and Rudolf \cite{Rudolph}.
	\item Janson, Nakamura and Zeilberger \cite{JansonNakamuraZeilberger} considered higher moments and joint moments, 
	rigourously when $\mathcal{C} = \Sn$, and also empirically for various classes $\mathcal{C}$.
	A bit later,
	Janson \cite{JansonPermutations} has given for $\mathcal C = \Av(132)$ the joint limit in distribution of the random variables $\occ(\pi,\bm\sigma_n)$, properly rescaled as to yield a nontrivial limit.
\end{itemize}

A parallel line of work consists in studying the asymptotic shape of the diagram of $\bm \sigma_n$. The \emph{diagram} of a permutation $\sigma \in \Sn_n$ is the set of points $\{(i,\sigma(i)),\, 1\leq i \leq n\}$ in the Cartesian plane.
To gain understanding of typical large permutations in $\mathcal{C}$, 
one can investigate the geometric properties of the diagram of $\bm \sigma_n$ as $n\to\infty$, possibly after rescaling this diagram so that it fits into a unit square.
\begin{itemize}
	\item Madras and Liu
	\cite{MadrasLiu}, Atapour and Madras \cite{AtapourMadras}
	and Madras and Pehlivan \cite{MadrasPehlivan} considered the asymptotic shape of $\bm{\sigma_n}$ when $\mathcal C = \mathrm{Av}(\tau)$
	for small patterns $\tau$.
	\item In parallel, Miner and Pak \cite{MinerPak}
	described very precisely the asymptotic shape of $\bm\sigma_n$, when $\mathcal C =\mathrm{Av}(\tau)$ for the $6$ patterns $\tau$ in $\Sn_3$. 
	These shapes are related to  Brownian excursion, as explained by Hoffman, Rizzolo and Slivken \cite{HoffmanBrownian1,HoffmanBrownian2}.
	\item Bevan describes the limit shape
	of permutations in so-called {\em connected monotone grid classes}
	\cite[Chapter 6]{BevanPhD}.%
\end{itemize}

These two points of view may seem different, but they are in fact tightly bound together. 
Indeed, as we shall see in Section~\ref{Sec:Permutons}, 
it follows from results of~\cite{Permutons} that the convergence of pattern densities 
characterizes the convergence of the diagrams, seen as permutons. 
This important property was actually the main motivation for the introduction of permutons in~\cite{Permutons}.

\subsection{The permuton viewpoint}\label{ssec:permutons_intro}

A {\em permuton} is a probability measure on the unit square $[0,1]^2$ with uniform marginals,
{\em i.e.} its pushforwards by the projections on the axes are both the Lebesgue measure on $[0,1]$. 
Permutons generalize permutation diagrams in the following sense:
to every permutation $\sigma\in \Sn_n$, we associate the permuton $\mu_\sigma$ with density 
\[
\mu_\sigma(dxdy) = n\One_{\sigma(\lceil xn \rceil) = \lceil yn \rceil} dxdy.
\]
Note that it amounts to replacing every point $(i,\sigma(i))$ in the diagram of $\sigma$ (normalized to the unit square) 
by a square of the form $[(i-1)/n,i/n]\times [(\sigma(i)-1)/n,\sigma(i)/n]$, which has mass $1/n$ uniformly distributed. 
\smallskip

Permutons were first considered by Hoppen, Kohayakawa, Moreira, Rath and Sampaio in \cite{Permutons},
with the point of view of characterizing
permutation sequences with convergent pattern densities.
The name {\em permuton} and the measure point of view
were given afterwards by Glebov, Grzesik, Klimo\v{s}ová and Kr{\'a}l \cite{FinitelyForcible}.
This recently introduced concept has already been the subject of many articles, including:
\begin{itemize}
  \item results on the set of possible pattern densities of permutons \cite{FinitelyForcible,ParameterTesting, LargeDeviations};
  \item a large deviation principle in the space of permutons,
    giving access to the analysis of random permutations with fixed pattern densities \cite{LargeDeviations};
  \item the description of the limiting distribution of 
    the number of fixed points (and more generally of cycles of a given length) 
    for ``equi-continuous'' sequences of permutations with a limiting permuton \cite{FixedPointsPermuton};
  \item central limit theorems and refinements for pattern occurrences in random permutation models associated to permutons
    \cite{FMN3};
  \item the permuton convergence of some exponentially tilded models of random permutations \cite{ExpTiltedPermuton};
  \item a study of permuton-valued processes, in the context of random sorting networks \cite{GeometryPermuton}.
\end{itemize}\smallskip

In the context of this article, the theory of permutons is a nice framework
to state scaling limit results for sequences of (random) permutations.
Indeed, the space $\mathcal M$ of permutons is equipped with the topology of weak convergence of measures, 
which makes it a compact metric space. 
This allows one to define convergent sequences of permutations: 
we say that $(\sigma_n)_n$ converges to $\mu$ when $(\mu_{\sigma_n}) \to \mu$ weakly. 
Accordingly, for a sequence $(\bm{\sigma_n})_n$ of {\em random} permutations,
we will consider the convergence in distribution of the associated random measures $(\mu_{\bm{\sigma_n}})$
in the weak topology. The limiting object is then a random permuton.

By definition, convergence to a permuton encodes the first-order asymptotics of the shape of a sequence of permutations. 
As we shall see in Section~\ref{Sec:Permutons}, it also encodes the first-order asymptotics of pattern densities: 
a sequence $(\bm{\sigma}_n)_n$ of random permutations converges in distribution to a random permuton if and only if the sequences of random variables $(\occ(\pi,\bm\sigma_n))_n$ converge in distribution, 
jointly for all $\pi\in\Sn$. 
Moreover, for any pattern $\pi$, the limit distribution of the density of $\pi$ can be expressed as a function of the limit permuton. 

Our previous article \cite{BrownianPermutation} studies the limit of the class $\mathcal C = \Av(2413,3142)$ of separable permutations, 
in terms of pattern densities and permutons. 
\begin{theorem}
\label{thm:thm_separable}
	Let $\bm \sigma_n$ be a uniform random separable permutation of size $n$. 
	There exists a random permuton $\bm \mu$, \textit{called the Brownian separable permuton},
    such that $(\mu_{\bm \sigma_n})_n$ converges in distribution to $\bm \mu$. 
\end{theorem}
The result of \cite{BrownianPermutation} is more precise, and describes the asymptotic joint distribution of the random variables $\occ(\pi,\Si_n)$ as a measurable functional of a \emph{signed Brownian excursion}. 
This object is a normalized Brownian excursion whose strict local minima are decorated with an i.i.d. sequence of balanced signs in $\{+,-\}$. 
In another paper \cite{MickaelConstruction}, the fifth author gives a direct construction of $\bm \mu$
from this signed Brownian excursion.
In particular, $\bm \mu$ is not equal almost surely to a given permuton;
in this regard, separable permutations behave differently from all other classes 
analysed so far in the literature (which converge to a deterministic permuton).

The class of separable permutations is the smallest nontrivial substitution-closed class, as defined in the next subsection. 
The present paper aims at showing a convergence result similar to Theorem~\ref{thm:thm_separable} for other substitution-closed classes. 
We will see that in many cases the limit belongs to a one-parameter family of deformations of the Brownian separable permuton: 
the \emph{biased Brownian separable permuton} $\bm \mu^{(p)}$ of parameter $p\in (0,1)$ is obtained from a biased signed Brownian excursion 
(defined similarly to the signed Brownian excursion but with each sign having probability $p$ of being a $+$). 
Simulations of the biased Brownian separable permuton are given in \cref{Fig:SimusPermuton}.
A precise definition will be given in \cref{Sec:Standard} (\cref{Def:PermutonBiaise}).

\begin{figure} [htbp]
  \vspace{-1.3cm}
\begin{center}
  \includegraphics[width=0.32\textwidth]{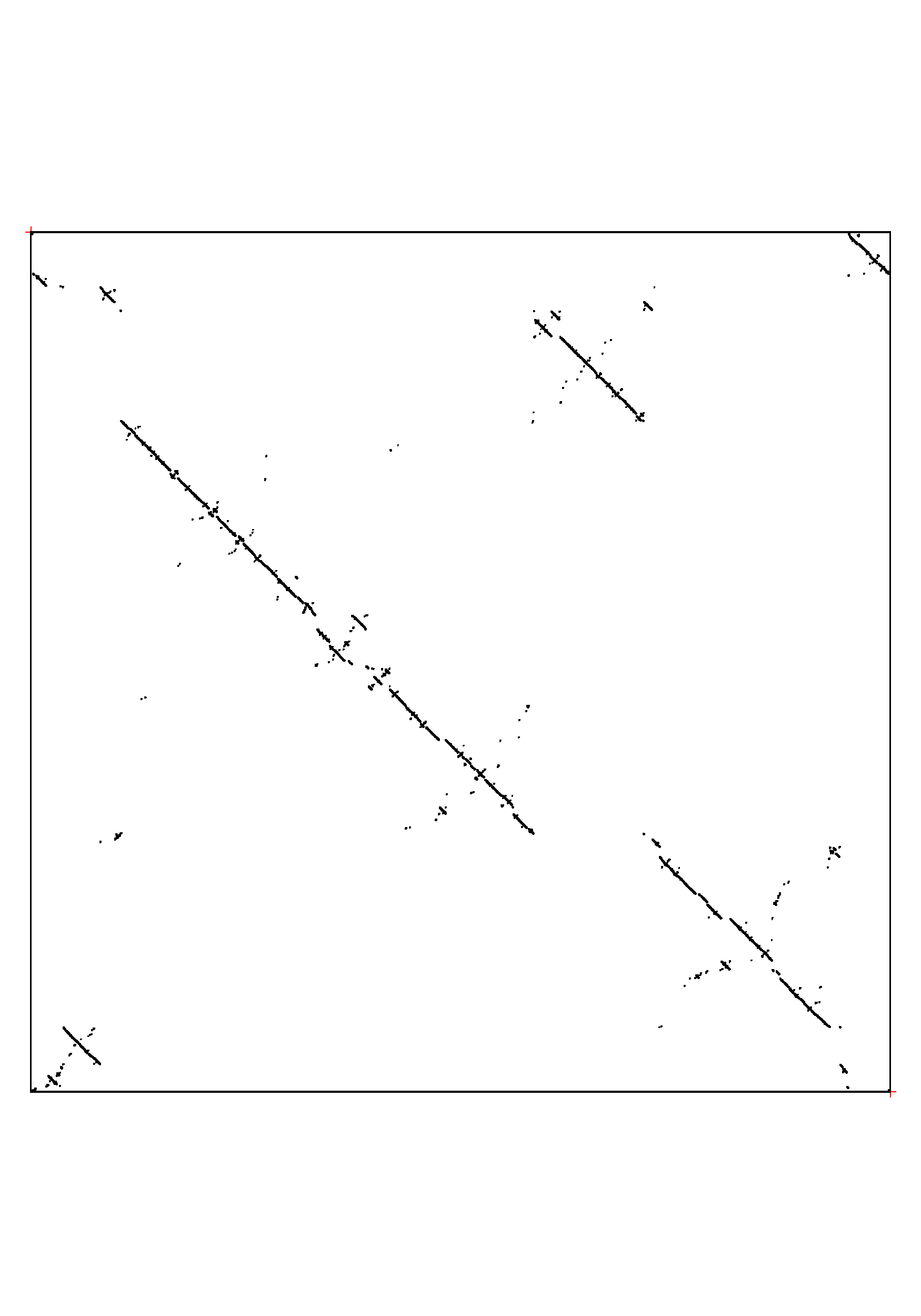}\ \
  \includegraphics[width=0.32\textwidth]{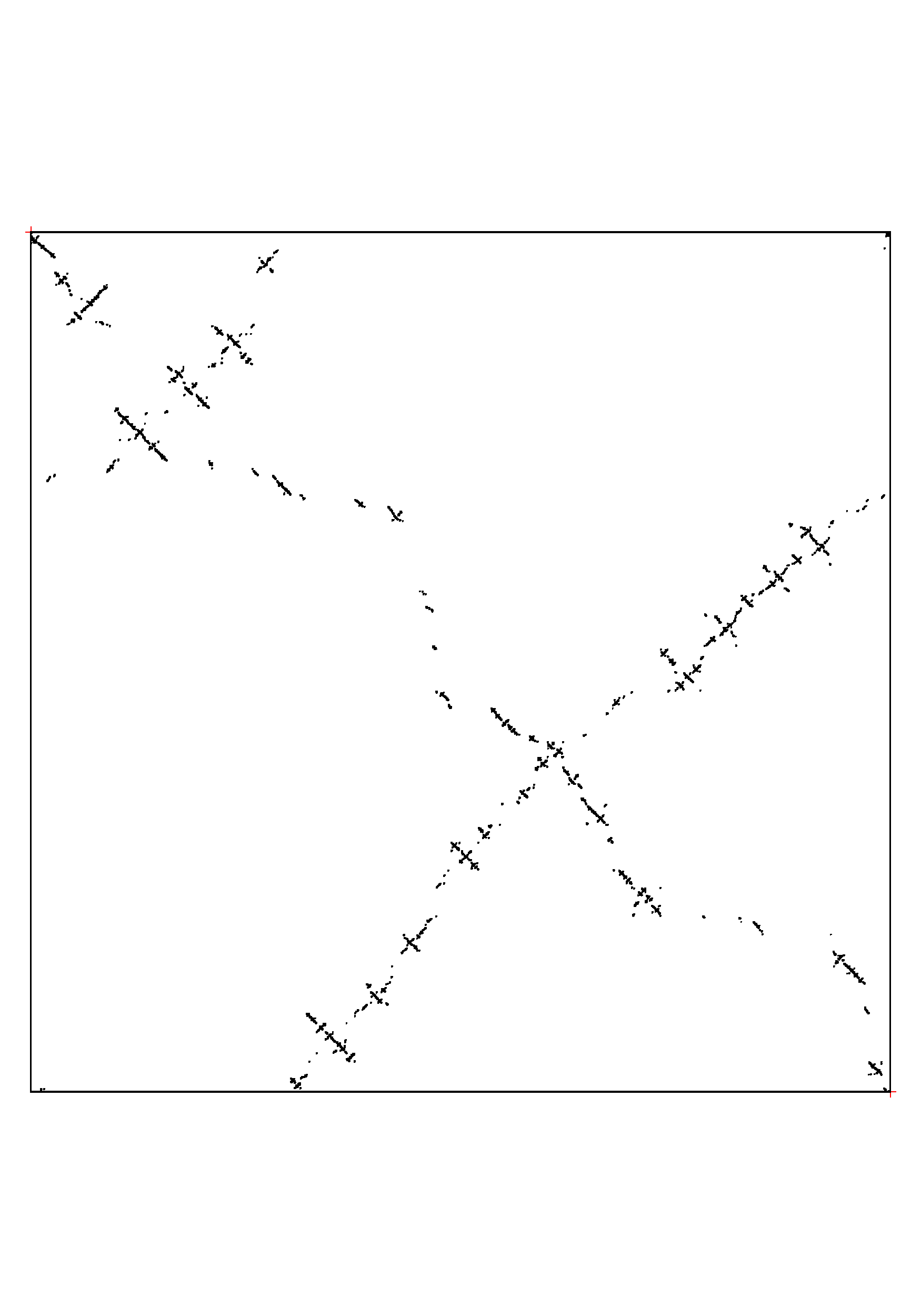}\ \ 
  \includegraphics[width=0.32\textwidth]{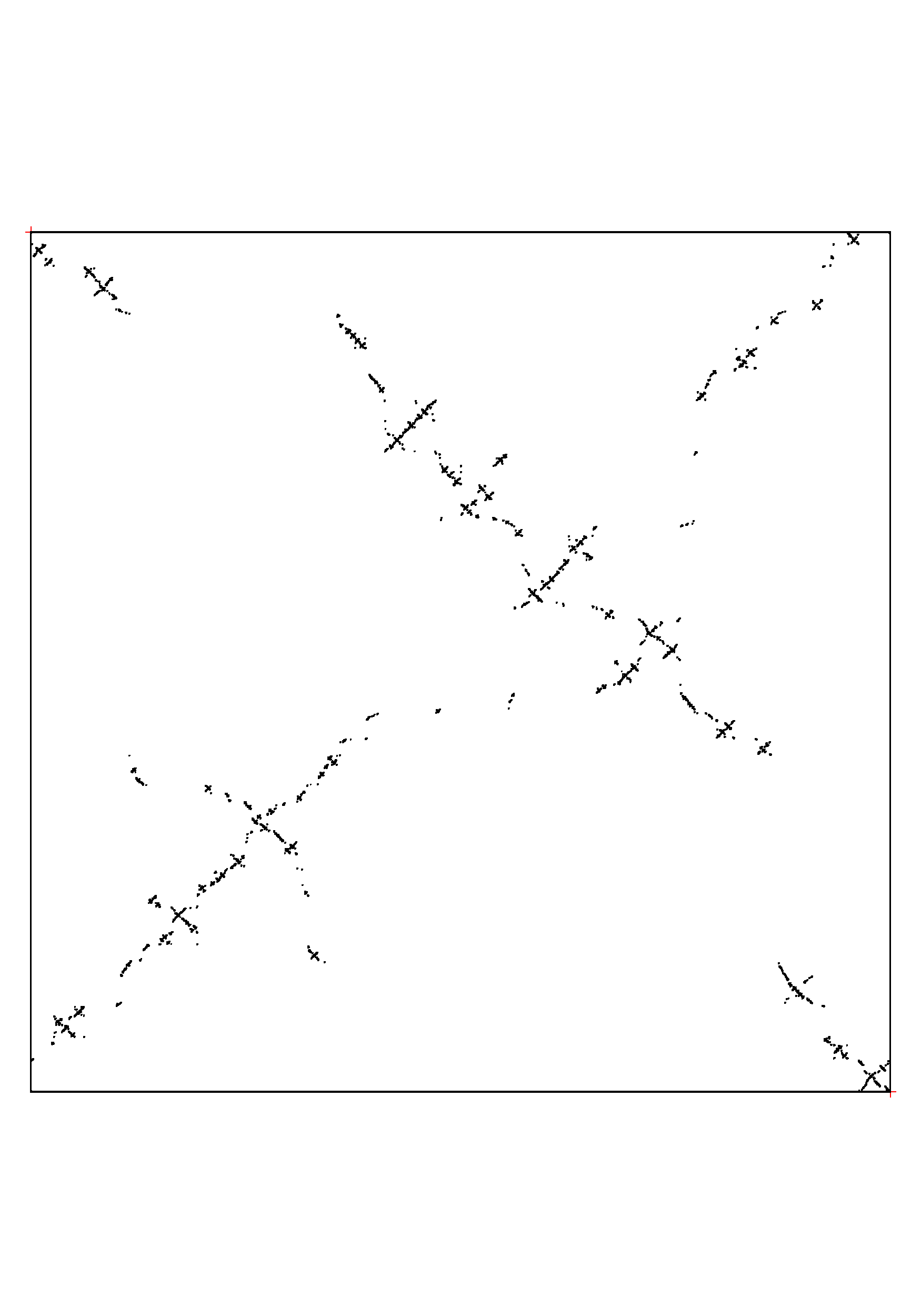}
\end{center}\vspace{-1.3cm}
\caption{Three simulations of the biased Brownian separable permuton $\bm\mu^{(p)}$ (left: $p=0.2$; middle: $p=0.45$; right: $p=0.5$.).
How these simulations were obtained is discussed in \cref{sec:AppendicePermuton}.
}\label{Fig:SimusPermuton}
\end{figure}

Finally, we mention that although permutons are a very nice, natural, and powerful way of studying ``limits of permutation classes'' (in particular because they unify many earlier results, as explained above), 
this approach has its weaknesses. Most importantly, it gives no information beyond the first order. 
For instance, the results of \cite{MinerPak,HoffmanBrownian1} describe the \textit{canoe} shape of large permutations in classes avoiding one pattern of length three. 
As the width of the canoe is a $o(n)$, one only sees the diagonal (or antidiagonal) in the permuton limit, but has no information about the fluctuations around this limit.

\subsection{Substitution-closed classes}\label{sec:OperatorsPermutations} 

\begin{definition}%
Let $\theta=\theta(1)\cdots \theta(d)$ be a permutation of size $d$, and let $\pi^{(1)},\dots,\pi^{(d)}$ be $d$ other permutations. 
The \emph{substitution} of $\pi^{(1)},\dots,\pi^{(d)}$ in $\theta$ is the permutation of size $|\pi^{(1)}|+ \dots +|\pi^{(d)}|$ 
obtained by replacing each $\theta(i)$ by a sequence of integers isomorphic to $\pi^{(i)}$ while keeping the relative order induced by $\theta$ between these subsequences.\\
This permutation is denoted by $\theta[\pi^{(1)},\dots,\pi^{(d)}]$. 
We sometimes refer to $\theta$ as the \emph{skeleton} of the substitution. 
\end{definition}

When $\theta$ is $12 \ldots k$ (resp. $k\ldots 21$), for any value of $k \geq 2$, we rather write $\oplus$ (resp. $\ominus$) instead of $\theta$. 
Note that the specific value of $k$ does not appear in this notation, but can be recovered counting the number of permutations $\pi^{(i)}$ which are substituted in $\oplus$ (resp. $\ominus$). 

Examples of substitution (see \cref{fig:sum_and_skew} below) are conveniently presented representing permutations by their diagrams: 
the diagram of $\theta[\pi^{(1)},\dots,\pi^{(d)}]$ is obtained by 
blowing up each point $\theta_i$ of $\theta$ onto a square containing the diagram of $\pi^{(i)}$. 

\begin{figure}[htbp]
    \begin{center}
      \includegraphics[width=7cm]{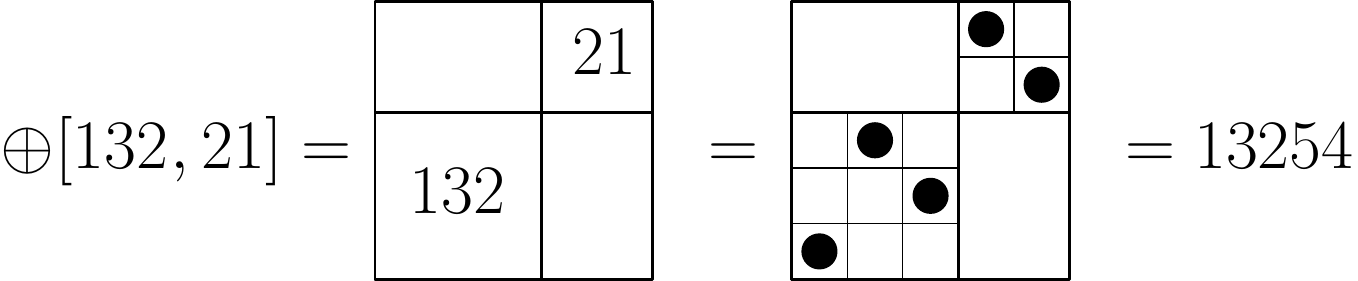} \qquad 
      \includegraphics[width=7cm]{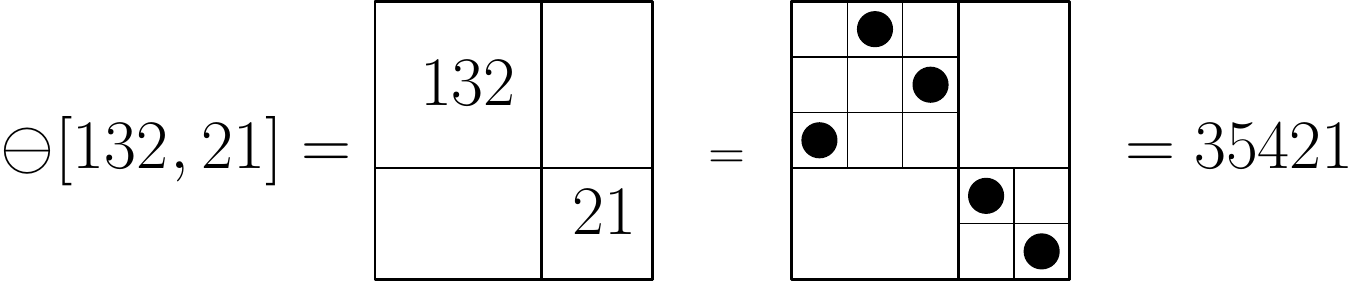}\\

\hspace{1cm}

	\includegraphics[width=95mm]{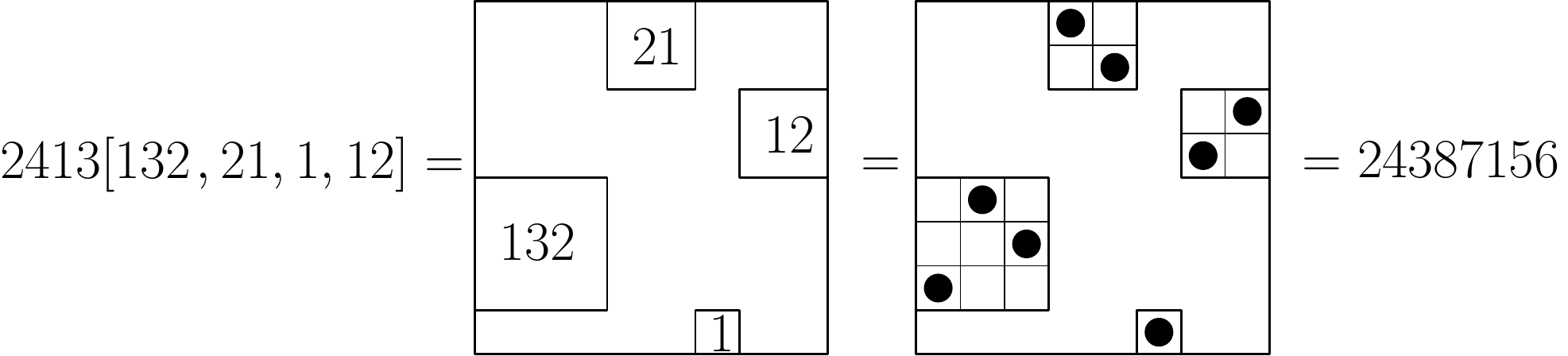}
    \end{center}
\caption{Substitution of permutations.\label{fig:sum_and_skew}}
\end{figure}

By definition of permutation classes, if $\theta[\pi^{(1)},\dots,\pi^{(d)}] \in \mathcal{C}$ for some permutation class $\mathcal{C}$,
then $\theta, \pi^{(1)},\dots,\pi^{(d)} \in \mathcal{C}$. The converse is not always true.
\begin{definition}%
A permutation class $\mathcal{C}$ is \emph{substitution-closed} if, for every $\theta, \pi^{(1)},\dots,\pi^{(d)}$ in $\mathcal{C}$,
$\theta[\pi^{(1)},\dots,\pi^{(d)}] \in \mathcal{C}$.
\end{definition}

The focus of this paper will be substitution-closed classes.
To study such classes it is essential to observe that any permutation has a canonical decomposition using substitutions, which can be encoded in a tree. This decomposition is canonical in the same sense as the decomposition of integers into products of primes. 
In this analogy, \emph{simple permutations} play the role of prime numbers and the substitution plays the role of the product.
We first give a simple definition: 
a permutation $\sigma$ is $\oplus$-indecomposable (resp. $\ominus$-indecomposable) 
if it cannot be written as $\oplus[\pi^{(1)},\pi^{(2)}]$ (resp. $\ominus[\pi^{(1)},\pi^{(2)}]$), 
(or equivalently, if there is no $d$ such that $\sigma$ can be written as $\oplus[\pi^{(1)},\dots,\pi^{(d)}]$ (resp. $\ominus[\pi^{(1)},\dots,\pi^{(d)}]$)).

\begin{definition}%
A \emph{simple permutation} is a permutation of size $n > 2$ that does not map any nontrivial interval 
(\emph{i.e.} a range in $[n]$ containing at least two and at most $n-1$ elements) onto an interval.
\end{definition}
For instance, $451326$ is not simple as it maps $[3;5]$ onto $[1;3]$. 
The smallest simple permutations are $2413$ and $3142$ (there is no simple permutation of size $3$). 

\medskip

\noindent \textbf{Remark:}
Usually in the literature, the definition of a simple permutation requires only $n \ge 2$ and not $n> 2$,
so that $12$ and $21$ are considered to be simple.
However in our work, $12$ and $21$ do not play the same role as the other simple permutations, that is why we do not consider them to be simple.

\begin{theorem}[Decomposition of permutations, Proposition 2 in \cite{AA05}]\label{Th:AlbertAtkinson}
Every permutation $\sigma$ of size $n\geq 2$ can be uniquely decomposed as either:
\begin{itemize}
\item $\alpha[\pi^{(1)},\dots,\pi^{(d)}]$, where $\alpha$ is simple (of size $d\geq 4$),
\item $\oplus[\pi^{(1)},\dots,\pi^{(d)}]$, where $d\geq 2$ and $\pi^{(1)},\dots,\pi^{(d)}$ are $\oplus$-indecomposable,
\item $\ominus[\pi^{(1)},\dots,\pi^{(d)}]$, where $d\geq 2$ and $\pi^{(1)},\dots,\pi^{(d)}$ are $\ominus$-indecomposable.
\end{itemize}
\end{theorem}
This decomposition theorem can be applied recursively inside the permutations $\pi^{(i)}$ appearing in the items above, 
until we reach permutations of size $1$. 
Doing so, a permutation $\sigma$ can be naturally encoded by a rooted planar tree,
whose internal nodes are labeled by the skeletons of the substitutions that are considered along the recursive decomposition process, 
and whose leaves correspond to the elements of $\sigma$.
This construction provides a one-to-one correspondence between permutations and {\em canonical trees} (defined below) that maps the size to the number of leaves.
\begin{definition}\label{defintro:CanonicalTree}
A \emph{canonical tree}  is a rooted planar tree whose internal nodes carry labels satisfying the following constraints.
\begin{itemize}
\item Internal nodes are labeled by $\oplus,\ominus$, or by a simple permutation.
\item A node labeled by $\alpha$ has degree\footnote{Throughout the paper, by \emph{degree} of a node in a tree, we mean the number of its children (which is sometimes called arity in other works). 
Note that it is different from the graph-degree: for us, the edge to the parent (if it exists) is not counted in the degree.} 
$|\alpha|$, nodes labeled by $\oplus$ and $\ominus$ have degree at least~$2$.  
\item A child of a node labeled by  $\oplus$ (resp. $\ominus$) cannot be labeled by  $\oplus$ (resp. $\ominus$).
\end{itemize}
\end{definition}
Canonical trees are known in the literature under several names: decomposition trees, substitution trees,\ldots
We choose the term canonical because we consider many variants of substitution trees in this paper, but only these canonical ones provide a one-to-one correspondence with permutations.

The representation of permutations by their canonical trees is essential in the study of subs\-ti\-tu\-tion-closed classes. 
The reason is that, for any such class $\mathcal{C}$, the set of canonical trees of permutations in $\mathcal{C}$ can be easily described. 

\begin{proposition}
\label{prop:treesOfSubsClosedClasses}
Let $\mathcal{C}$ be a substitution-closed permutation class, and assume\footnote{Otherwise, 
$\mathcal{C} = \{12\ldots k : k \geq 1 \}$ or $\mathcal{C} = \{k\ldots 21 : k \geq 1 \}$ or  $\mathcal{C} = \{1\}$ and these cases are trivial.} that $12,21 \in \mathcal{C}$. 
Denote by $\mathcal{S}$ the set of simple permutations in $\mathcal{C}$. 
The set of canonical trees encoding permutations of $\mathcal{C}$ is the set of all canonical trees built on the set of nodes $\{\oplus, \ominus\} \cup \{\alpha : \alpha \in \mathcal{S}\}$. 
\end{proposition}

\begin{proof}
First, if a canonical tree contains a node labeled by a simple permutation $\alpha \notin \mathcal{S}$, then the corresponding permutation $\sigma$ contains the pattern $\alpha \notin \mathcal{C}$, and hence $\sigma \notin \mathcal{C}$. 
Second, by induction, all canonical trees built on $\{\oplus, \ominus\} \cup \{\alpha : \alpha \in \mathcal{S}\}$ encode permutations of $\mathcal{C}$, because $\mathcal{C}$ is substitution-closed. 
If necessary, details can be found in~\cite[Lemma 11]{AA05}. 
\end{proof}

For instance, the class $\Av(2413,3142)$ of separable permutations studied in~\cite{BrownianPermutation} 
corresponds to the set of all canonical trees built on $\{\oplus, \ominus\}$, \emph{i.e.}, to $\mathcal{S}=\emptyset$. 
It is therefore the smallest nontrivial substitution-closed class.

\begin{observation}
\label{obs:Sdownwardclosed}
Let $\mathcal{C}$ be any substitution-closed permutation class, and let $\mathcal{S}$ be the set of simple permutations in $\mathcal{C}$. 
Because $\mathcal{C}$ is a class, it holds that for all $\alpha \in \mathcal{S}$, if $\alpha'$ is a simple permutation such that $\alpha' \preccurlyeq \alpha$, then $\alpha' \in \mathcal{S}$. 
Whenever a set $\mathcal{S}$ of simple permutations satisfies this property, we say that $\mathcal{S}$ is {\em downward-closed} (implicitly: for $\preccurlyeq$ and among the set of simple permutations). 
\end{observation}

Thanks to their encoding by families of trees, it can be proved that substitution-closed permutation classes (possibly, satisfying additional constraints) share a common behavior. 
For example, the canonical tree representation of their elements
imply that all substitution-closed classes with finitely many simple permutations    
have an algebraic generating function~\cite[Corollary 14]{AA05}.
(This is actually easy, the main contribution of~\cite{AA05} being to generalize this algebraicity result to all classes containing a finite number of simple permutations, 
again using canonical trees as a key tool.)
Our work illustrates this universality paradigm in probability theory:
we prove that the biased Brownian separable permuton is the limiting permuton 
of many substitution-closed classes (see Theorem~\ref{Th:MainIntro} and to a lesser extent Theorem~\ref{Th:Main3}). 

\subsection{Our results: Universality}%

Let $\mathcal{S}$ be a (finite or infinite) set of simple permutations. 
We denote by $\langle \mathcal{S}\rangle_n$ the set of permutations of size $n$ 
whose canonical trees use only nodes $\oplus$, $\ominus$ and $\alpha \in \mathcal{S}$, 
and we define $\langle \mathcal{S}\rangle = \cup_n \langle \mathcal{S}\rangle_n$. 
From Proposition~\ref{prop:treesOfSubsClosedClasses} and Observation~\ref{obs:Sdownwardclosed}, 
every substitution-closed permutation class $\mathcal{C}$ containing $12$ and $21$ can be written as $\mathcal{C} = \langle \mathcal{S}\rangle$ 
for a downward-closed set $\mathcal{S}$ of simple permutations (which is just the set of simple permutations in $\mathcal{C}$). 

\begin{remark}
For a generic (not necessarily downward-closed) set $\mathcal{S}$ of simple permutations, 
$\langle \mathcal{S}\rangle$ is a family of permutations more general than a substitution-closed permutation class. 
The results that we obtain apply not only to permutation classes but also to such sets of permutations. 

Note however that our work does not consider substitution-closed sets of permutations 
not containing either $12$ or $21$ (as mentioned above, a permutation {\em class} not containing one of these two permutations
is necessary trivial, but there might be interesting such substitution-closed {\em sets}).
In principle, such sets of permutations could also be studied by the approach developed in this paper, 
but we prefer to leave such cases outside of our study. 
Indeed, to cover them, it would require to re-do all computations,
modifying the combinatorial equations that we start from (see Proposition~\ref{Prop:systeme1} p.~\pageref{Prop:systeme1}) 
and all equations that follow, so as not to allow the nodes labeled $\oplus$ and/or $\ominus$. 
\end{remark}

We are interested in the asymptotic behavior of a uniform permutation $\Si_n$ in $\langle \mathcal{S}\rangle_n$ which we describe in terms of permutons. 
Let 
$$
S(z)=\sum_{\alpha \in \mathcal{S}} z^{|\alpha|}
$$
be the generating function of $\mathcal{S}$ and let $R_S\in [0,+\infty]$ be the radius of convergence of $S$.
\begin{theorem}[Main Theorem: the standard case]\label{Th:MainIntro}
Let $\mathcal{S}$ be a set of simple permutations such that
\begin{equation}\label{eq:h1intro}
  R_S > 0 \quad \text{and} \quad \lim_{r\rightarrow R_S \atop r < R_S} S'(r) > \frac{2}{(1+R_S)^2} -1.
  \tag{H1}
\end{equation}
For every $n\geq 1$, let $\Si_n$ be a uniform permutation in $\langle \mathcal{S}\rangle_n$, and let $\mu_{\Si_n}$ be the random permuton
associated with $\Si_n$.
The sequence $(\mu_{\Si_n})_n$ tends in distribution in the weak
convergence topology to the biased Brownian separable permuton $\bm \mu^{(p)}$
whose parameter $p$ is given in \eqref{eq:p_plus_p_moins} p.~\pageref{eq:p_plus_p_moins}.
\end{theorem}
An important point in \cref{Th:MainIntro} is that the limiting object
depends on $\mathcal{S}$ only through the parameter $p$. 
It turns out that $p$ only depends on the number of occurrences 
of the patterns $12$ and $21$ in the elements of $\mathcal{S}$.
We illustrate this universality of the limiting object on \cref{Fig:PermInClasses},
by showing large uniform random permutations in two different substitution-closed 
classes:
the first one has a finite set of simple permutations
$\SSS=\{2413,3142,24153,42513\}$,
while the second is the substitution closure of $\Av(321)$,
which contains infinitely many simple permutations and satisfies \eqref{eq:h1intro} (as we will explain below).
Although this is hard to see on the picture,
the corresponding values of the biaised parameter
are different, namely .5 and around .6 respectively 
(see \cref{Sec:Standard}, \cref{ex:calculp1,ex:calculp3}).
\begin{figure} [htbp]
	\begin{center}
      \includegraphics[width=.4\textwidth]{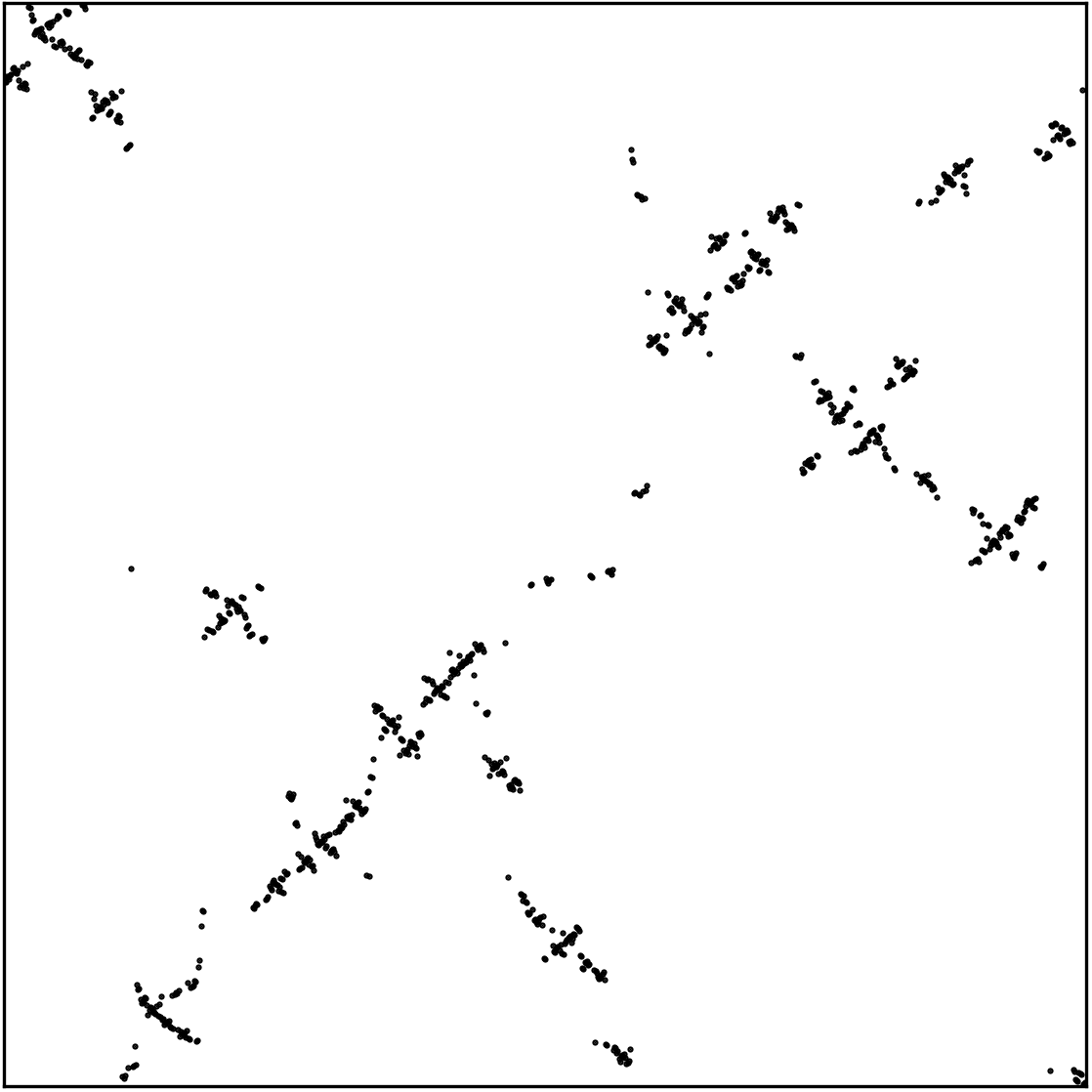} \qquad
      \includegraphics[width=0.4\textwidth]{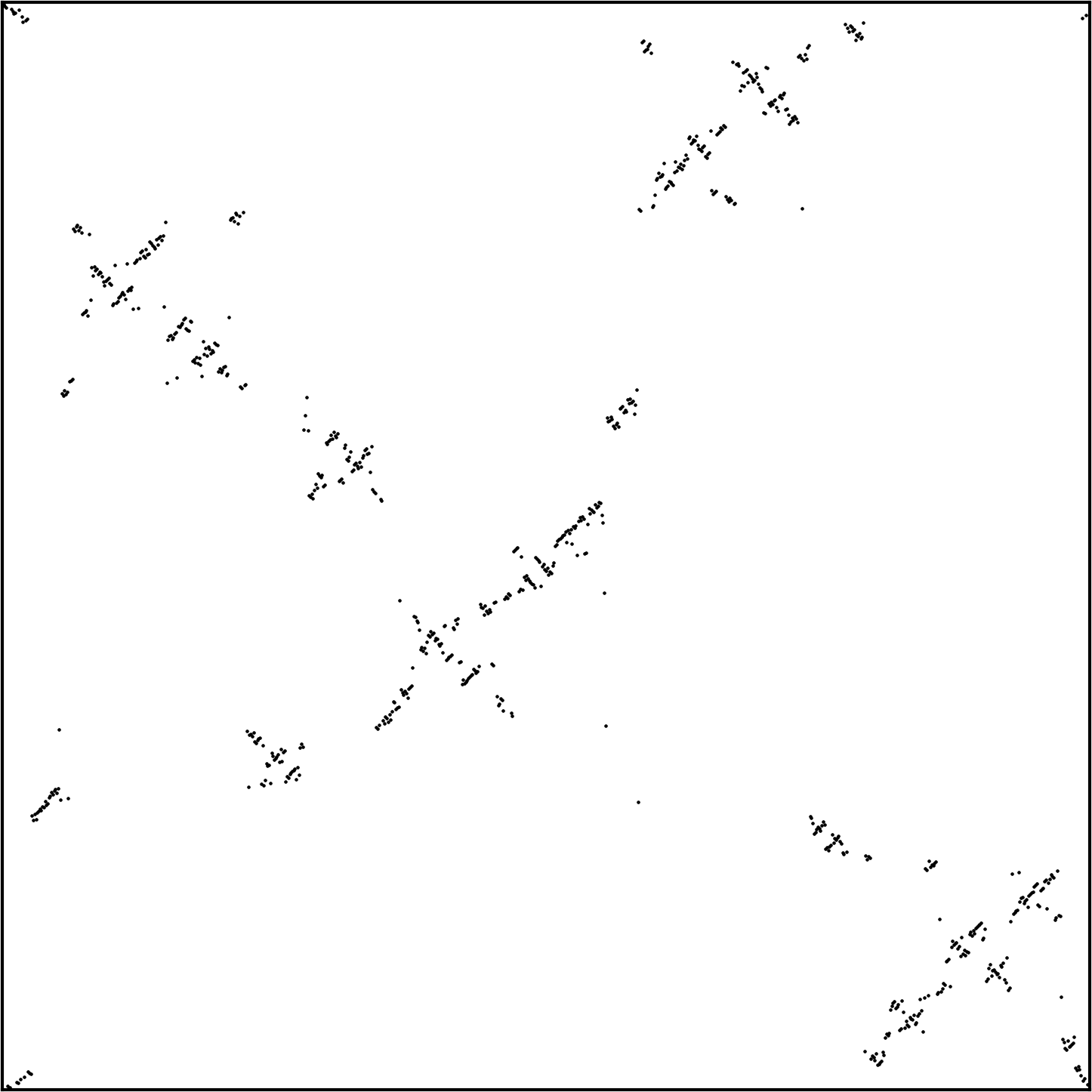}
	\end{center}
	\caption{On the left: A uniform random permutation of size 981
    in the substitution-closed class with $\SSS=\{2413,3142,24153,42513\}$.
    On the right: A uniform random permutation of size 840 in the substitution-closure $\CCC$ of $\Av(321)$, \emph{i.e.}, when $\SSS$ is the set of simple $321$-avoiding permutations.}
    \label{Fig:PermInClasses}
\end{figure}

In the following, to lighten the notation, we write $S'(R_S) := \lim_{r\rightarrow R_S \atop r < R_S} S'(r)$. Note that $S'(R_S)$ may be $\infty$. 

The case when Condition \eqref{eq:h1intro} of \cref{Th:MainIntro} is not satisfied is discussed in the next section.
When Condition \eqref{eq:h1intro} is satisfied the case is called \emph{standard} 
because there are natural and easy sufficient conditions to ensure this case (that are given below).
Moreover, this case includes most sets $\mathcal{S}$ studied so far in the literature on permutation classes, to our knowledge.
This gives a fairly precise (and positive) answer to an important question raised in our previous article \cite{BrownianPermutation}: is the Brownian separable permuton universal 
(in the sense that it describes the limit of a large family of substitution-closed classes)? 

We now give several cases in which Condition \eqref{eq:h1intro} of \cref{Th:MainIntro} is satisfied.
\begin{itemize}
\item If $S$ is a generating function with radius of convergence $R_S > \sqrt{2}-1$,
  \eqref{eq:h1intro} is satisfied. 
  Indeed, the condition $R_S > \sqrt{2}-1$ implies $\frac{2}{(1+R_S)^2} -1 <0$, 
  and $S'(R_S)$ is nonnegative since $S'$ (like $S$) is a series with nonnegative coefficients. 
  In particular, the situation where $R_S > \sqrt{2}-1$ covers the cases where
  there are finitely many simple permutations in the class (then $S$ is a polynomial and $R_S=\infty$),
  and more generally where $R_S=1$ ({\em i.e.}
  the number of simple permutations of size $n$ grows subexponentially).
\item If $S'$ is divergent at $R_S$, \eqref{eq:h1intro} is trivially verified. In particular, this happens
  when $S$ is a rational generating function, or when $S$ has a square root singularity at $R_S$.
\end{itemize}

In the literature, there are quite a few examples of permutations classes whose set $\mathcal S$ of simple permutations
has been enumerated.
We can therefore ask whether Condition~\eqref{eq:h1intro} applies to them.
In most examples we could
find, it is indeed satisfied, and this follows from the discussion above. We record these examples here. 
\begin{itemize}
\item %
  Classes with finitely many simple permutations have
attracted a fair amount of attention,  see~\cite{AA05} and subsequently~\cite{BBPR, MathildeMarniCyril, BRV}.
\item Several families of simple permutations with a bounded number of elements of each size have appeared in the literature:
  the family of exceptional simple permutations (also called simple parallel alternations
in~\cite{BrignallSurvey}),
 the family of wedge simple permutations (see also~\cite{BrignallSurvey}), the families
  of oscillations and quasi-oscillations (see~\cite{PinPerm}), and the families of simple permutations contained
  in the following three classes: $\Av(4213,3142)$, $\Av(4213,1342)$ and $\Av(4213,3124)$ -- see \cite{AlbertAtkinsonVatter3}.
\item The family of simple pin-permutations has a rational generating function -- see \cite{PinPerm}. 
\item The generating function $S$ is also rational when $\mathcal S$ is the set of simple permutations contained in several
permutation classes defined by the avoidance of two patterns of size~$4$, namely $\Av(3124,4312)$ -- see~\cite{Pantone},
$\Av(2143,4312)$ and $\Av(1324,4312)$ -- see~\cite{AlbertAtkinsonBrignall3}, $\Av(2143,4231)$ -- see\cite{AlbertAtkinsonBrignall},
$\Av(1324,4231)$ -- see~\cite{AlbertAtkinsonVatter}, $\Av(4312,3142)$ and $\Av(4231,3124)$ -- see~\cite{AlbertAtkinsonVatter3}.
\item The set $\mathcal S$ of simple permutations of  the class 
$\Av(4231, 35142, 42513, 351624)$ enumerated in~\cite{AlbertBrignall}  is also rational.
\item We come back to the above example, where $\CCC$ is the substitution
  of $Av(321)$. This class has been studied in \cite{Basis_Subst_Closure},
  where an explicit basis of avoided patterns is given.
  In this case, $\mathcal S$ is the set of simple permutations avoiding $321$,
  whose generating function $S$ is computed in~\cite{AlbertVatter}: 
it has a square-root singularity at $R_S = \tfrac{1}{3}$,
which proves that \eqref{eq:h1intro} is fulfilled.
\end{itemize}

In addition to verifying Condition~\eqref{eq:h1intro},
we have computed the numerical value of the parameter $p$
for some of the above-mentioned sets $\mathcal S$ of simple permutations;
see \cref{ex:calculp1,ex:calculp2,ex:calculp3} (p.~\pageref{ex:calculp1}).
%

Notably absent from the above list is the class $\Av(2413)$, enumerated in \cite{Stankova3142,Bona0}.
Since the avoided pattern, $2413$, is simple, this class is substitution-closed.
Its generating series
behaves as $C (\rho-z)^{\mathbf{3/2}}$ around its dominant singularity $\rho=1/8$.
This prevents the set of simple permutations in this class to satisfy Condition \eqref{eq:h1intro};
compare with \cref{Prop:asymp_T}.
\subsection{Our results: Beyond universality}
When $R_S > 0$, for the two remaining cases $S'(R_S)<2/(1+R_S)^2-1$ and $S'(R_S)=2/(1+R_S)^2-1$, the asymptotic behavior of $\mu_{\Si_n}$ is qualitatively different,
and the results require slight additional hypotheses and notation.
As a consequence, for the moment we only briefly describe these behaviors, the results being stated with full rigor later.

\begin{itemize}
\item {\bf Case }$S'(R_S)<2/(1+R_S)^2-1$. This is a \emph{degenerate} case.\\
We first show in \cref{Th:MainH2}  that, with a small additional assumption which will be called $(CS)$,
the sequence $(\mu_{\Si_n})$  of random permutons converges.
If uniform simple permutations in $\mathcal{S}\cap \Sn_n$ have a limit (in the sense of permutons),
we show that the limit of permutations in $\langle \mathcal{S} \rangle$ is the same 
(see \cref{Prop:H2degenerate} and the subsequent comment).
This explains the terminology ``degenerate'': all permutations in the class (or set) $\langle \mathcal{S} \rangle$ are close to the simple ones,
and the ``composite'' structure of permutations does not appear in the limit.

\item {\bf Case }$S'(R_S)=2/(1+R_S)^2-1$. This \emph{critical} case is more subtle.\\
  We again need to assume the above mentioned hypothesis $(CS)$. According to the behavior of $S$ near $R_S$, the limiting permuton of $(\mu_{\Si_n})$ can  either be the (biased) Brownian separable permuton (\cref{Th:Main3}) or belong to a new family of \emph{stable permutons} (\cref{Th:Main4}). 
Finite substructures of stable permutons are 
connected to those of the random \emph{stable tree} (see \cite{DuquesneLeGall}),
which explains the terminology. Two simulations are presented in \cref{Fig:SimusPermutonStable}.
\end{itemize}
\begin{figure} [htbp]
	\begin{center}
		\includegraphics[width=0.40\textwidth]{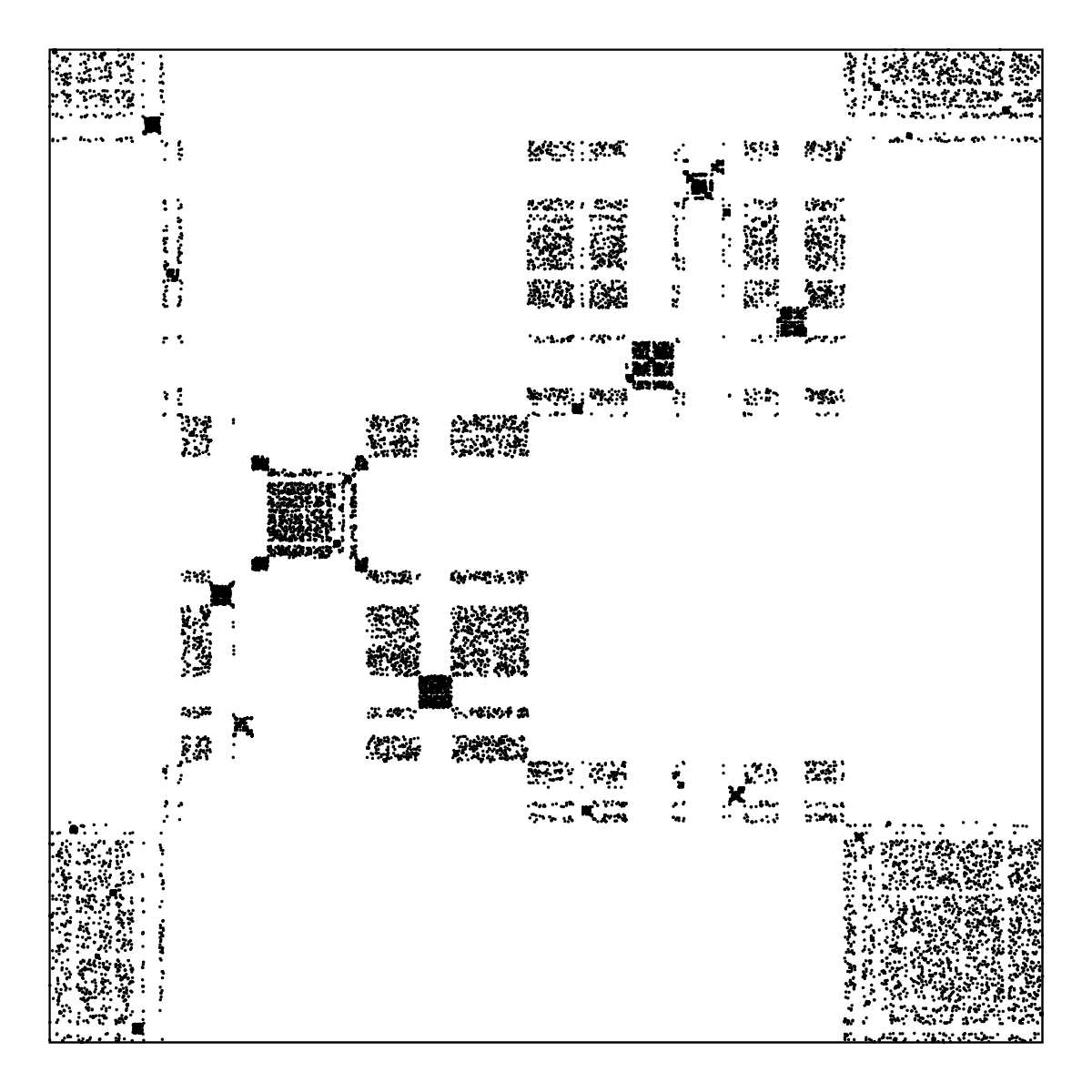}\hspace{2cm}
		\includegraphics[width=0.40\textwidth]{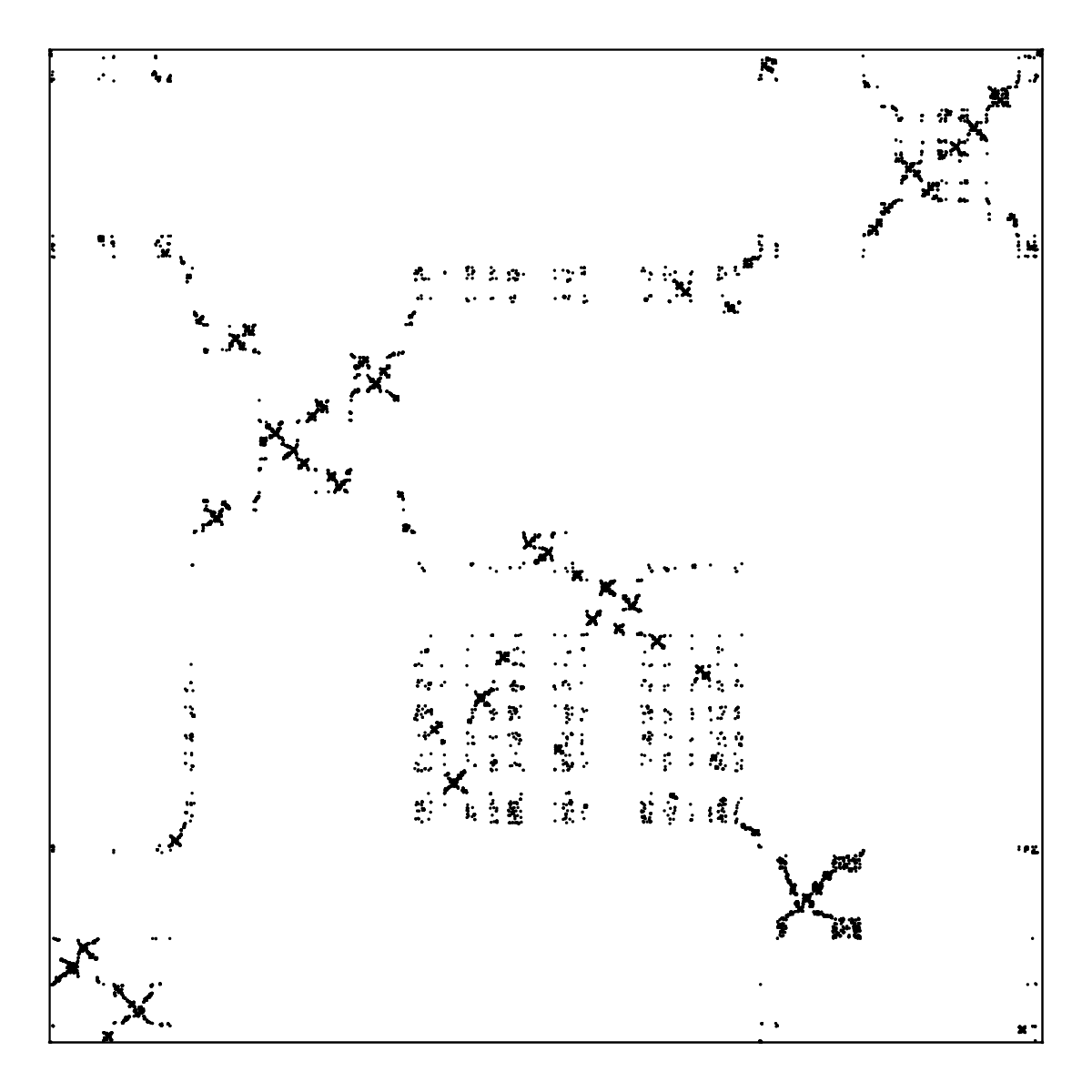}
	\end{center}
	\caption{Simulations of a $1.1$-stable and $1.5$-stable permuton, driven by the uniform measure.
	}\label{Fig:SimusPermutonStable}
\end{figure}
We believe that the above-mentioned class $\Av(2413)$ belongs to the degenerate regime.
Indeed, in the critical regime, the singularity exponent of the class should be smaller than $1$,
and cannot be $3/2$, as for $\Av(2413)$.
Since there is no direct description of the simple permutations in $\Av(2413)$,
it seems however out of reach to prove our hypothesis $(CS)$ for this specific class.
We are therefore unable to describe its limiting permuton and let this open for further research.
We refer to \cite[Fig. 7]{Madras2413} for simulations of uniform random permutations in this class.

\begin{remark}\label{Rem:GW}
  The variety of behaviors that we observe can be informally understood in terms of trees.
  We have seen in \cref{sec:OperatorsPermutations} that permutations in $\langle \mathcal{S}\rangle$
  can be encoded by trees.
  Taking a uniform element in $\langle \mathcal{S}\rangle_n$,
  we can prove that the corresponding tree is a multi-type Galton-Watson tree
  conditioned on having $n$ leaves.
  This link with conditioned Galton-Watson trees is not used in this paper, but may give intuition on our results.

  In the standard and critical case, these Galton-Watson trees are critical.
  It is therefore not surprising to see two different limiting behaviors.
  When the law of reproduction has finite variance, we get one behavior related to the Brownian excursion and the Brownian continuum random tree (that we describe as the universal one).
  When the law of reproduction has infinite variance, we get the behavior related to stable trees. 
  In the standard case, the law of reproduction %
  always has finite variance.

  On the contrary, in the degenerate case, the underlying Galton-Watson tree model is subcritical.
  At the limit, such trees conditioned to being large have one internal node of very high degree
  (\cite[Theorem 7.1, case (ii)]{JansonSurveysGWTrees}).
  This node corresponds to a large simple permutation in the tree encoding a uniform random permutation $\Si_n$ in $\langle \mathcal{S}\rangle$.
  It is therefore not surprising that $\Si_n$ is asymptotically close to a uniform \emph{simple} permutation in $\langle \mathcal{S}\rangle$.
\end{remark}

\begin{remark}
The reader may have noticed that all cases where we describe the asymptotic behavior of $\mu_{\Si_n}$ are such that $R_S >0$. 

Observe that it is always the case for proper permutation classes (\emph{i.e.}, permutation classes different from $\Sn$). 
Indeed, from the Marcus-Tardos Theorem~\cite{MaTa04}, the number of permutations of size $n$ in a proper class is at most $c^n$, for some constant $c$. 
For the class $\Sn$, we however do have $R_S = 0$, since 
there are asymptotically $e^{-2}n!(1+\mathcal{O}(1/n))$ simple permutations of size $n$ \cite[Theorem 5]{AlbertAtkinsonKlazar}. 
In this case, the sequence $(\mu_{\Si_n})_n$ of permutons  associated with a uniform permutation $\Si_n$ in $\Sn$
converges in distribution to the uniform measure on $[0,1]^2$.
The situation where $R_S = 0$ may happen as well for sets $\langle \mathcal{S} \rangle$ where $\mathcal{S}$ is not downward-closed, 
but we leave these cases open. 
\end{remark}

\subsection{Limits of proportions of pattern occurrences}
Let us change our approach and discuss in this section
the asymptotic behavior (as $n\to \infty$) of the proportion
$\occ(\pi,\Si_n)$ of occurrences of a fixed pattern $\pi$ in $\Si_n$
(as done in \cite{Bona1,Bona2,ChenEuFu,Homberger,JansonPermutations,JansonNakamuraZeilberger,Rudolph}
for uniform random permutations in various classes).
Since most examples fit in that regime, we focus here
on the standard case (when \eqref{eq:h1intro} is satisfied).

As mentioned in \cref{ssec:permutons_intro} and explained in more details
in \cref{Sec:Permutons}, the convergence of $\mu_{\Si_n}$ 
towards $\bm \mu^{(p)}$ implies the (joint) convergence in distribution
\begin{equation}
  \big(\occ(\pi,\Si_n) \big)_\pi \to \occ(\pi,\bm \mu^{(p)}).
  \label{eq:LimitOcc}
\end{equation}
The limiting random variables $\occ(\pi,\bm \mu^{(p)})$ have been
studied in \cite[Section~9]{BrownianPermutation} (for $p=.5$):
in particular, $\occ(\pi,\Si_n)$ is non-deterministic if and only if $\pi$ is separable of size at least 2
and it is possible to compute their moment algorithmically.
These results are easily extended to the general case $p \in (0,1)$.
Therefore, for separable patterns $\pi$, \eqref{eq:LimitOcc} establishes
the convergence of $\occ(\pi,\Si_n)$ to a non-deterministic limit.
Since these are bounded variables, their (joint) moments also converge
to the (joint) moments of the limiting vector, which can be computed algorithmically
(even if in practice only low order moments can be effectively computed; see the
discussion in \cite[Section 9]{BrownianPermutation}).
Note that all these limiting moments are trivially nonzero,
since these are moments of nondeterministic nonnegative random variables.

For nonseparable patterns however, the situation is different:
\eqref{eq:LimitOcc} only entails the convergence of $\occ(\pi,\Si_n)$ to $0$.
Indeed, if $\pi$ is nonseparable, the limiting quantity $\occ(\pi,\bm \mu^{(p)})$ is identically $0$
(this is a consequence of \cite[Proposition 9.1]{BrownianPermutation} when $p=.5$, the result being easily extended to $p \in (0,1)$).

We can go further and ask whether $\big(\occ(\pi,\Si_n)\big)_n$ has a limit in distribution
with some appropriate normalization. 
We therefore investigate the moments of $\occ(\pi,\Si_n)$.
In \cref{Sec:OccNonSeparables}, 
we define some permutation statistics $\db(\pi)$ (see \cref{eq:defi_db}) and 
show that under the hypothesis (H1) we have the following asymptotic behavior\footnote{We say that the sequence $(a_n)$ behaves as $\Theta(b_n)$ if there are $c,C>0$ such that $c |b_n|\leq |a_n|\leq C |b_n|$ for every $n\geq 1$.}.

\smallskip

\noindent{\bf Proposition \ref{prop:momentsnonseparables}.}
\emph{
For each $\pi \in \CCC$ and $m\geq 1$, we have $\esper[(\occ(\pi,\Si_n))^m] = \Theta (n^{-\db(\pi)/2})$. 
}

\smallskip

Proposition \ref{prop:momentsnonseparables} also holds for separable patterns $\pi$:
in that case $\db(\pi)$=0 and we have \hbox{$\esper[(\occ(\pi,\Si_n))^m]= \Theta (1)$}.
No news here, since the moments of $\occ(\pi,\Si_n)$ have nonzero limits,
as previously explained.
For nonseparable patterns, $\db(\pi)$ is positive and  measures in some sense how nonseparable $\pi$ is.
Note that the order of magnitude of $\esper[(\occ(\pi,\Si_n))^m]$ is independent of $m$,
which implies that there is a set of probability $\Theta(n^{-\db(\pi)/2})$
on which the variables $\occ(\pi,\Si_n)$ stays bounded away from $0$ (see \cref{corol:momentsnonseparables}).
This event of small probability contributes to the asymptotic behavior of moments,
and thus the method of moments is inappropriate to find a limiting distribution 
for some appropriate normalization of $\occ(\pi,\Si_n)$.
Finding such a limiting distribution is therefore left as an open question.
\subsection{Outline of the proof}
\label{Subsec:outline_proof}
In our previous paper \cite{BrownianPermutation} (\emph{i.e.} when the family of simple permutations is $\mathcal{S}=\emptyset$), 
the proof of the convergence to the Brownian separable permuton strongly relied on 
a connection to Galton-Watson trees conditioned on having a given number of leaves.
This allowed us to use fine results by Kortchemski \cite{Igor} or Pitman and Rizzolo \cite{PitmanRizzolo}
on such conditioned random tree models. 

For a general family $\mathcal{S}$, generalizing this approach would require delicate
results on the asymptotic behavior of conditioned \emph{multitype} Galton-Watson trees.
Moreover, there are several other steps in the main proofs of \cite{BrownianPermutation},
in particular the subtree exchangeability argument, that are not easily adapted.

The strategy developed in the present paper is different. We strongly use the framework of permutons. Indeed, we first show that to establish the convergence  in distribution of $(\mu_{\bm \sigma_n})_n$ %
to some random permuton $\bm \mu$,
it is enough to prove the convergence of $\Big(\esper\left[\occ(\pi,\bm \sigma_n)\right]\Big)_n$ for every pattern $\pi$ (see \cref{thm:randompermutonthm}). By definition, if $\pi \in \Sn_k$ and $n\geq k$,
\begin{equation}
	\esper[\occ(\pi,\bm \sigma_n)] = \frac {\# \{\sigma \in \langle \SSS \rangle_n, I \subset [n] : \pat_I(\sigma) = \pi\}} {\binom n k\ \#\langle \SSS \rangle_n}
    \label{eq:MainQuotient}
  \end{equation}

The asymptotic behavior of the numerator and denominator is then obtained with analytic combinatorics, which allows us to transfer from the behavior of a generating series near its singularity to the asymptotic behavior of its coefficients.
This goes in three steps. 
\medskip

\noindent \textbf{Step 1: Enumeration.} We compute (or characterize by an implicit equation) some generating series.
For instance to estimate the denominator of \eqref{eq:MainQuotient} we consider $\sum_{n\geq 1} \#\langle \SSS \rangle_n\ z^n$.
We readily use the size-preserving bijection between $\langle \SSS \rangle$ and the class $\TTT$ of $\SSS$-canonical trees, counted by the number of leaves.
Hence the generating function we want to compute is the same as that of $\TTT$, denoted $T$.

Using again the encoding of permutations by trees, the numerator
can be described as a number of trees with marked leaves and some conditions
on the tree induced by these marked leaves.
Obtaining generating functions for such combinatorial classes is possible,
and needs the introduction of several intermediary functions which count trees with various contraints, and possibly one marked leaf. This is detailed in \cref{Sec:Enumeration}.
\medskip

\noindent \textbf{Step 2: Singularity analysis.} Then we want to know the singular behavior of the generating functions we computed so far.
As it turns out, the singular behavior of some intermediate function $T_{\nonp}$ drives the singular behavior of all the other series.
The function $T_{\nonp}$ is characterized by the implicit equation
\begin{equation}
\label{Eq:Tnonp_intro}
	T_\nonp(z) = z + \Lambda(T_\nonp(z)),
\end{equation}
where $\Lambda$ is a known analytic function with radius of convergence $R_\Lambda$ that involves $S$ and some rational functions (see \cref{eq:DefLambda} p.~\pageref{eq:DefLambda}).
Hence the behavior of $T_\nonp$ depends on whether there is a point inside the disk of convergence $D(0,R_\Lambda)$ where $\Lambda'=1$, because around such a \textit{critical} point, the equation \eqref{Eq:Tnonp_intro} is not invertible.
Since $\Lambda$ is a series with positive integer coefficients, it suffices to check the sign of $\Lambda'(R_\Lambda) -1$, which can easily be translated in terms of the function $S$.
This is where the sign of $S'(R_S)- 2/(1+R_S)^2 + 1$ appears, leading to the three different cases.
More precisely\footnote{In this informal description, we left out some conditions on the singularity of $S$ that appear in the critical and degenerate cases.}
 
\begin{itemize}
\item The standard case $S'(R_S)>2/(1+R_S)^2-1$ is equivalent to $\Lambda'(R_\Lambda) >1$.
In this case there is a unique critical point $\tau \in (0,R_\Lambda)$.
As a result, the radius of convergence of $T_\nonp$ is $\rho = \tau - \Lambda(\tau)$,
and the analyticity of $\Lambda$ around $\tau = T_\nonp(\rho)$ implies that $T_\nonp$ has a singularity of exponent $1/2$ 
(\cref{lem:MMCversionSerie}).
Such a behavior is sometimes called \textit{branch point} in the literature:
$\Lambda$ is analytic at $\tau$ but the equation \eqref{Eq:Tnonp_intro} has two solutions (called branches)
near $\rho$ and one cannot find an analytic solution in a neighbourhood of $\rho$.
The solution $T_\nonp$ is therefore singular at $\rho$.

\item The degenerate case $S'(R_S)<2/(1+R_S)^2-1$ is equivalent to $\Lambda'(R_\Lambda) <1$.
In this case there is no critical point in the disk $D(0,R_\Lambda)$ nor at its boundary.
As a result, the unique dominant singularity of $T_\nonp$ is the point $\rho = R_\Lambda - \Lambda(R_\Lambda)$
where $T_\nonp(\rho)$ reaches the singularity $R_\Lambda$ of $\Lambda$. Moreover, $T_\nonp$ has a bounded derivative at its singularity, and so has exponent $\delta > 1$, which is the same as the exponent of $S$
(\cref{lem:Analyse_Tnonp2}).

\item The critical case $S'(R_S)=2/(1+R_S)^2-1$ is equivalent to $\Lambda'(R_\Lambda) =1$.
In this case there is no critical point inside the disk $D(0,R_\Lambda)$, but the singularity $R_\Lambda$ of $\Lambda$ is a critical point.
Once again the radius of convergence of $T_\nonp$ is $\rho = R_\Lambda - \Lambda(R_\Lambda)$, but $T_\nonp$ has no first derivative at its singularity.
Here the exponent of the singularity of $T_\nonp$ depends on that of the singularity of $S$, and belongs to $[1/2,1)$
(see \cref{lem:Inversion3}).

\end{itemize}
Once we have found the asymptotic behavior of $T_{\nonp}$, we should analyze
the tree series found in Step 1.
It is purely routine from an analytic point of view,
but involves some combinatorial arguments, regarding the encoding of permutations
by substitution trees.\medskip

\noindent \textbf{Step 3: Transfer.} Finally we use a transfer theorem of analytic combinatorics (\cref{thm:transfert}) to translate the singularity exponents we found in Step 2 into a limiting behavior for \eqref{eq:MainQuotient}.
Informally, a square-root singularity, which is the same as in "usual" families of trees, will lead to the Brownian separable permuton. A singularity of exponent in $(1/2,1)$ will lead to the $\delta$-stable tree, where $\delta \in (1,2)$ is the inverse of the exponent. A singularity of exponent $\delta > 1$ will invariably lead to the degenerate case.

\subsection{Organization of the paper}

\begin{figure}[htbp]
\begin{center}
\includegraphics[width=10cm]{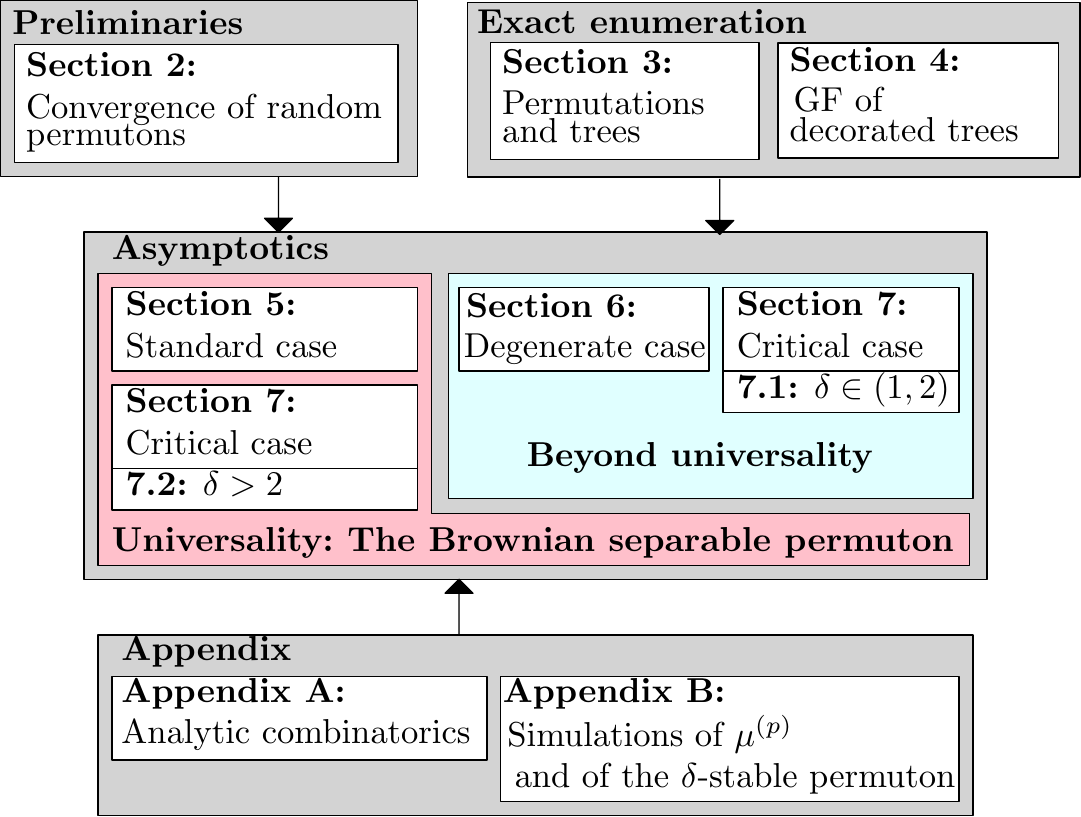}
\caption{Reader's guide of the paper.}\label{fig:orga-papier}
\end{center}
\end{figure}
The paper is organized as follows (see also \cref{fig:orga-papier}).
\begin{itemize}
\item Section \ref{Sec:Permutons} is devoted to proving useful results on the convergence of random permutons.
The proofs heavily rely on previous estimates for deterministic permutons \cite{Permutons}.
We believe that these general results regarding random permutons are interesting on their own, therefore these are presented in a self-contained way.
\item In Sections \ref{Sec:CodingByTrees} and \ref{Sec:Enumeration}, we prove nonasymptotic enumeration results for the number of permutations encoded by some given families of (decorated) trees. The main result is \cref{prop:Dec_TtS}, which is the first step towards the estimation of $\mathbb{E}\left[\occ(\pi,\bm \sigma_n)\right]$.
\item In Sections \ref{Sec:Standard}, \ref{sec:Asymp2}, \ref{sec:critical} we prove our main results: the convergence of the sequence $(\mu_{\Si_n})_n$ of permutons. As already mentioned, the quantitative behavior depends on the family $\mathcal{S}$, more precisely on the sign of $S'(R_S)-2/(1+R_S)^2+1$:
\begin{itemize}
\item Section \ref{Sec:Standard} is devoted to the standard case $S'(R_S)>2/(1+R_S)^2-1$. We show in \cref{Th:Main} the convergence to the biased Brownian separable permuton.
\item In Section \ref{sec:Asymp2}, we consider the degenerate case $S'(R_S)<2/(1+R_S)^2-1$.
\item In Section \ref{sec:critical}, we consider the critical case $S'(R_S)=2/(1+R_S)^2-1$. This case itself is divided into two subcases, according to whether the exponent $\delta$ (defined in \cref{Def:HypoCS}) is smaller (\cref{ssec:deltaPlusPetit}) or greater (\cref{ssec:deltaPlusGrand}) than $2$.
\end{itemize} 
\item We postpone to \cref{sec:complex_analysis} many useful results of complex analysis.
\item Finally, \cref{sec:AppendicePermuton} 
  discusses how  \cref{Fig:SimusPermuton,Fig:SimusPermutonStable,Fig:PermInClasses}
  have been obtained.
\end{itemize}
\section{Convergence of random permutons}\label{Sec:Permutons}

In this section, we first recall the terminology of (deterministic) permutons, introduced in \cite{Permutons}.
We also adapt their results to obtain criteria for the convergence in distribution of random permutons.

\noindent{\bf Notation.}
\emph{
Since this section involves many different probability spaces,
we use a superscript on $\proba$ (and similarly on expectation symbols $\esper$) to record the source of randomness.
In the case where the event $A(\bm u, \bm v)$ (or the function $H(\bm u, \bm v)$)
depends on two random variables $\bm u$ and $\bm v$, we interpret $\proba^{\bm u}(A(\bm u, \bm v))$ (or $\esper^{\bm u}[H(\bm u, \bm v)]$)
as the conditional probability (expectation) with respect to $\bm v$.
}
\medskip

\subsection{Deterministic permutons and extracted permutations}\label{Sec:ExtractedPermutations} ~
Recall from \cref{ssec:permutons_intro} that a permuton is 
a probability measure on the unit square with uniform marginals. 
To a permutation $\sigma$ of size $n$, we can associate the permuton 
$\mu_\sigma$ which is essentially the (normalized) diagram of $\sigma$,
where each dot has been replaced with a small square of dimension $1/n \times 1/n$ carrying a mass $1/n$.

Let $\mathcal M$ be the set of permutons.
We need to equip $\mathcal M$ with a topology.
We say that a sequence of (deterministic) permutons $(\mu_n)_n$ converges \emph{weakly} to $\mu$ (simply denoted $\mu_n \to \mu$) if 
$$
\int_{[0,1]^2} f d\mu_n \stackrel{n\to +\infty}{\to} \int_{[0,1]^2} f d\mu,
$$
for every bounded and continuous function $f: [0,1]^2 \to \mathbb{R}$.
With this topology, $\mathcal M$ is compact and metrizable by a metric $d_\square$ which has been introduced in \cite{Permutons} (see Lemmas 2.5 and 5.3 in \cite{Permutons}):
$$
\mu_{n}\stackrel{n\to +\infty}{\to} \mu \qquad  \Leftrightarrow \qquad d_\square(\mu_n,\mu)\stackrel{n\to +\infty}{\to} 0.
$$
Since $\mathcal{M}$ is compact, Prokhorov's theorem ensures that the space of probability distributions on $\mathcal M$ is compact (for convergences of measure, we refer to \cite{Billingsley}). 
\medskip

Recall from Section \ref{sec:PermutationClass} that for $\sigma \in \Sn_n$ and $\pi \in \Sn_k$, we have
$$
\occ(\pi,\sigma)=\mathbb{P}^{{\bm I}_{n,k}} \left(\pat_{{\bm I}_{n,k}}(\sigma)=\pi \right),
$$
where ${\bm I}_{n,k}$ is randomly and uniformly chosen among the $\binom{n}{k}$ subsets of $[n]$ with $k$ elements.
The random permutation $\pat_{{\bm I}_{n,k}}(\sigma)$ is called the {\em induced subpermutation} (of size $k$) in $\sigma$.
We will define the {\em pattern density} $\occ(\pi,\mu)$ of a pattern $\pi \in \Sn_k$ in a {\em permuton} $\mu$
by analogy with this formula.

Take a sequence of $k$ random points $(\XX,\YY) = ((\xx_1,\yy_1),\dots, (\xx_k,\yy_k))$ in $[0,1]^2$, 
independently with common distribution $\mu$. 
Because $\mu$ has uniform marginals and the $\xx_i$'s (resp. $\yy_i$'s) are independent, 
it holds that the $\xx_i$'s (resp. $\yy_i$'s) are almost surely distinct.
We denote by $(\xx_{(1)},\yy_{(1)}),\dots, (\xx_{(k)},\yy_{(k)})$ the $x$-ordered sample of $(\XX,\YY)$,
 \emph{i.e.} the unique reordering of the sequence $((\xx_1,\yy_1),\dots, (\xx_k,\yy_k))$ such that
$\xx_{(1)}<\cdots<\xx_{(k)}$.
Then the values $(\yy_{(1)},\cdots,\yy_{(k)})$ are in the same
relative order as the values of a unique permutation, that we denote $\Perm(\XX,\YY)$.
Since the points are taken at random, $\Perm(\XX,\YY)$ is a random permutation of size $k$.
We call it the {\em induced subpermutation} (of size $k$) in $\mu$.
Then we set
\[\occ(\pi,\mu) = \proba^{\XX,\YY} \, \big(\, \Perm(\XX,\YY)=\pi \,\big).
\]
Rewriting this probability in an integral form, we get immediately:
\begin{equation}
\label{eq:defIntegraleOcc}
\occ(\pi,\mu) = \int_{([0,1]^2)^k} \One_{\Perm(\vec x, \vec y) = \pi}\;\mu(dx_1dy_1)\cdots\mu(dx_kdy_k)
\end{equation}
which identifies $\occ(\pi,\cdot)$ as a measurable function on the space of permutons.

In the following, as we consider a random permuton $\bm \mu$, we need to construct a finite sequence of points $(\xx_1,\yy_1), \dots, (\xx_k,\yy_k)$, which are independent with common distribution $\bm \mu$ \textit{conditionally on $\bm \mu$}. This is possible up to considering a new probability space where the joint distribution of $(\bm \mu, (\xx_1, \yy_1),\ldots,(\xx_k, \yy_k))$ is characterized as follows: for every positive measurable functional $H : \mathcal M \times ([0,1]^2)^k \to \R$, 
\begin{multline}\label{eq:propUniverselle}
\esper^{\bm \mu, \XX,\YY }[H(\bm \mu,(\xx_1,\yy_1), \ldots, (\xx_k,\yy_k))]\\
= \esper^{\bm \mu} \left[ \int_{([0,1]^2)^k} \bm \mu(dx_1\,dy_1) \cdots \bm \mu(dx_k\,dy_k) H(\bm \mu, (x_1,y_1), \ldots, (x_k,y_k))
\right].
\end{multline}
In this new probability space, we call $\Mk$ the vector $(\XX,\YY) = (\xx_i, \yy_i)_{1\leq i \leq k}$, and we use the notation \hbox{$\Perm(\Mk, \bm\mu)=\Perm(\XX,\YY)$},
to highlight the two levels of randomness.
\medskip

We end this section by the following two estimates, proved in \cite{Permutons}.
\begin{lemma}[Occurrences in a permutation and its associated permuton {\cite[Lemma 3.5]{Permutons}}]\label{lem:densityofassociatedpermuton}\ 
If $\pi \in \Sn_k$ and $\sigma \in \Sn_n$, then 
$$
|\occ(\pi,\sigma)-\occ(\pi,\mu_\sigma)| \leq \frac 1 n \binom k 2.
$$
\end{lemma}

\begin{lemma}[Approximation of a permuton by a permutation {\cite[Lemma 4.2]{Permutons}}]
		\label{lem:subpermapproxpermuton}
		There is a $k_0$ such that if $k>k_0$, for any permuton $\nu$,
		\[\proba^\Mk\left[
		d_{\square}(\mu_{\Perm(\Mk,\nu)},\nu)
		\geq 16k^{-1/4}\right]
		\leq \frac12 e^{-\sqrt{k}}.
		\]
	\end{lemma}

\subsection{Random permutons and convergence in distribution}
We now consider a sequence of random permutations $(\bm \sigma_n)$ (with $\bm \sigma_n$ of size $n$).
An example of interest for the present paper is when, for each $n \ge 1$,
$\bm \sigma_n$ is a uniform random permutation of size $n$ in a given class $\CCC$.
Another example are the random permutations $(\bm \sigma_n)_{n\geq 1} = (\Perm(\Mn,\bm \mu))_{n\geq 1}$
constructed above from a given random permuton $\bm \mu$.
In the case where $\bm \mu$ is deterministic, these correspond to the $Z$-random permutations from \cite{Permutons},
used to prove that each permuton is the limit of some permutation sequence.
\medskip

Taking ${\bm I}_{n,k}$ independently from $(\bm \sigma_n)$, we have for every $\pi$ of size $k$:
\begin{equation}\label{eq:E(occ)=P(pat)}
\esper^{\bm \sigma_n}[\occ(\pi,\bm \sigma_n)] 
= \esper^{\bm \sigma_n}\left[\mathbb{P}^{{\bm I}_{n,k}} \left(\pat_{{\bm I}_{n,k}}(\bm \sigma_n)=\pi \right)\right]
=\proba^{\bm\sigma_n,{\bm I}_{n,k}} (\pat_{{\bm I}_{n,k}}(\bm \sigma_n) = \pi).
\end{equation}
Similar, for a random permuton $\Mu$, we have 
\begin{equation}
\esper^{\bm \mu}[\occ(\pi,\bm \mu)]
=\proba^{\bm\mu,\Mk} (\Perm(\Mk,\bm \mu)=\pi).
\label{eq:E(occ)=P(perm)}
\end{equation}
This is a consequence of \eqref{eq:propUniverselle} above, 
applied to $H(\mu,(x_1,y_1),\ldots,(x_k,y_k)) = \One_{\Perm(\vec x,\vec y) = \pi}$
and combined with \eqref{eq:defIntegraleOcc}.
The same argument may be applied to 
$$H(\nu,(x_1,y_1),\ldots,(x_k,y_k)) = \One_{d_{\square}(\mu_{\Perm(\vec x,\vec y)},\nu)
	\geq 16k^{-1/4}},$$ yielding a randomized version of  \cref{lem:subpermapproxpermuton}.

\begin{lemma}[Approximation of a random permuton by a random permutation]
	\label{lem:subpermapproxrandompermuton}
	There is a $k_0$ such that if $k>k_0$, for any random permuton $\bm \nu$,
	\[\proba^{\bm \nu,\Mk}
	\left[
	d_{\square}(\mu_{\Perm(\Mk,\bm \nu)},\bm \nu)
	\geq 16k^{-1/4}\right]
	\leq \frac12 e^{-\sqrt{k}}.
	\]
\end{lemma}

\medskip

This result has an important consequence for the distribution of random permutons. 
\begin{proposition}[Subpermutations characterize the distribution of $\bm \mu$]\label{Prop:CaracterisationLoiPermuton}
Let $\bm \mu$, $\bm {\mu}'$ be two random permutons. If there exists $k_1$ such that for $k\geq k_1$ and every $\pi$ of size $k$ we have 
$$
\proba^{\bm \mu,\Mk}(\Perm(\Mk,\bm \mu)=\pi)=\proba^{\bm \mu',\Mk}(\Perm(\Mk,\bm \mu')=\pi),
$$
then $\bm \mu \stackrel{d}{=}\bm \mu'$.
\end{proposition}
\begin{proof}
We need to prove that $\mathbb{E}^{\bm \mu}[\phi(\bm \mu)]=\mathbb{E}^{\bm \mu'}[\phi(\bm \mu')]$ for every bounded and continuous function $\phi:\mathcal{M}\to \mathbb{R}$. Fix $k\geq k_1$. It holds that 
\begin{align*}
\mathbb{E}^{\bm\mu}[\phi(\bm\mu)]-\mathbb{E}^{\bm\mu'}[\phi(\bm\mu')]
& = \mathbb{E}^{\bm\mu,\Mk}[\phi(\bm\mu)-\phi(\mu_{\Perm(\Mk,\bm\mu)})] \\
& + \left (\mathbb{E}^{\bm\mu,\Mk}[\phi(\mu_{\Perm(\Mk,\bm\mu)})]
-\mathbb{E}^{\bm\mu',\Mk'}[\phi(\mu_{\Perm(\Mk',\bm\mu')})]\right) \\
& + \mathbb{E}^{\bm\mu',\Mk'}[\phi(\mu_{\Perm(\Mk',\bm\mu')})-\phi(\bm\mu')],
\end{align*}
where $\Mk'$ denotes a sequence of $k$ independent points
with common distribution $\bm \mu'$,
conditionally on $\bm \mu'$.
The second term in the above display is zero by assumption. Moreover, from \cref{lem:subpermapproxrandompermuton} the first and third terms go to zero when $k\to +\infty$.
\end{proof}

Our main theorem in this section deals with the convergence of sequences of random permutations to a random permuton. 
It generalizes the result of \cite{Permutons} which states that \emph{deterministic} permuton convergence is characterized by convergence of pattern densities. 
We extend their proof to the case of \emph{random} sequences, where permuton convergence \emph{in distribution} is characterized by convergence of \emph{average} pattern densities,
or equivalently of the induced subpermutations of any (fixed) size.

\begin{theorem}\label{thm:randompermutonthm}
	For any $n$, let $\bm\sigma_n$ be a random permutation of size $n$. 
	Moreover, for any fixed $k$, let ${\bm I}_{n,k}$ be a uniform random subset of $[n]$ with $k$ elements, independent of $\bm \sigma_n$.
	The following assertions are equivalent.
	\begin{enumerate}%
		\item [(a)] $(\mu_{\bm \sigma_n})_n$ converges in distribution for the weak topology to some random permuton $\bm \mu$. 
		\item [(b)] The random infinite vector $\big(\occ(\pi,\bm\sigma_n)\big)_{\pi \in \mathfrak S}$ converges in distribution in the product topology to some random infinite vector $(\bm \Lambda_\pi)_{\pi \in \mathfrak S}$. 
		\item [(c)]For every $\pi$ in $\mathfrak S$, there is a $\Delta_\pi \geq 0$ such that \[\esper[\occ(\pi,\bm \sigma_n)] \xrightarrow{n\to\infty} \Delta_\pi.\]
		\item [(d)]For every $k$, the sequence  $\big(\pat_{{\bm I}_{n,k}}(\bm\sigma_n)\big)_n$ of random permutations converges in distribution to some random permutation $\bm \rho_k$.
	\end{enumerate}
	Whenever these assertions are verified, we have $(\bm \Lambda_\pi)_\pi \stackrel d = (\occ(\pi,\bm \mu))_\pi$ and for every $\pi\in\Sn_k$,
	\[ \proba(\bm \rho_{k} = \pi) = \Delta_\pi = \esper[\bm \Lambda_\pi] = \esper[\occ(\pi,\bm\mu)] = \proba(\Perm(\Mk,\bm\mu) = \pi). \] 

\end{theorem}

\begin{observation}
\label{obs:StartAt2}
In item (c) above, it is enough to consider all $\pi$ of size at least $2$. 
Indeed, for $\pi=1$, the statement is trivial, since $\occ(\pi, \cdot)$ is identically $1$. 
\end{observation}

\noindent{\bf Proof of (a)$\Rightarrow$(b). } Let $\pi_1,\ldots,\pi_r$ be a finite sequence of patterns. By \cite[Lemma 5.3]{Permutons}, the map  $\mu \mapsto (\occ(\pi_i,\mu))_{1\leq i \leq r}$ is continuous. 
Therefore, $\mu_{\bm \sigma_n}\stackrel{d}{\to} \bm \mu$ implies
$$
\big(\occ(\pi_i,\mu_{\bm \sigma_n})\big)_{1\leq i \leq r} \stackrel{d}{\to} \big(\occ(\pi_i ,\bm\mu)\big)_{1\leq i \leq r}.
$$
Using \cref{lem:densityofassociatedpermuton},
one can replace each $\occ(\pi_i,\mu_{\bm\sigma_n})$ by $\occ(\pi_i,\bm\sigma_n)$ in the above convergence. 
This proves the convergence in distribution 
of all induced permutations $\big(\occ(\pi_i,{\bm \sigma_n})\big)_{1\leq i \leq k}$, 
and hence of $\big(\occ(\pi,\bm\sigma_n)\big)_{\pi \in \mathfrak S}$ in the product topology (see for instance \cite[ex. 2.4 p.~19]{Billingsley}).
\noindent{\bf Proof of (b)$\Rightarrow$(c). } If $\occ(\pi,\bm\sigma_n) \stackrel{d}{\to} \bm \Lambda_\pi$, 
as $\occ$ takes values in $[0,1]$, we have
$$
\esper[\occ(\pi,\bm\sigma_n)] \stackrel{n\to\infty}{\to}  \esper[\bm \Lambda_\pi].
$$

\noindent{\bf Proof of (c)$\Rightarrow$(d). } Fix $\pi\in \Sn_k$ and consider the sequence 
$$
\proba^{\bm\sigma_n,{\bm I}_{n,k}}(\pat_{{\bm I}_{n,k}}(\bm \sigma_n) = \pi)= \esper^{\bm\sigma_n}[\occ(\pi,\bm \sigma_n)],
$$
which converges if (c) holds (the equality comes from \cref{eq:E(occ)=P(pat)}). Since $\pat_{{\bm I}_{n,k}}(\bm \sigma_n)$ is a random variable taking its values in the finite set $\Sn_k$, this says exactly that the sequence $\big(\pat_{{\bm I}_{n,k}}(\bm \sigma_n)\big)_n$ converges in distribution.
\medskip

\noindent{\bf Proof of (d)$\Rightarrow$(a).}
Consider a sequence of random permutations $(\bm \sigma_n)$ satisfying (d),
{\em i.e.} for every $k$, there is a random permutation $\bm \rho_k$ such that $\pat_{{\bm I}_{n,k}}(\bm \sigma_n) \stackrel{d}{\to} \bm \rho_k$.
Put differently, for every pattern $\pi$ of size $k$, we have
$$\mathbb{P}^{\bm\sigma_n, {\bm I}_{n,k}} \left(\pat_{{\bm I}_{n,k}}(\bm \sigma_n)=\pi \right)\to  \proba(\bm \rho_k = \pi).$$

From \cref{lem:densityofassociatedpermuton} and \cref{eq:E(occ)=P(pat)}, we get
\begin{equation*}
\esper^{\bm\sigma_n} \left[ \occ(\pi,\mu_{\bm\sigma_n})\right] = \esper^{\bm\sigma_n} \left[ \occ(\pi,\bm\sigma_n)\right] + \mathcal{O}(1/n)
= \proba^{{\bm I}_{n,k},\bm\sigma_n}(\pat_{{\bm I}_{n,k}}(\bm\sigma_n) = \pi) + \mathcal{O}(1/n).
\end{equation*}
Set $\bm\theta_{k,n} = \Perm(\Mk,\mu_{\bm\sigma_n})$. Then, using \cref{eq:E(occ)=P(perm)}, for every $\pi \in \Sn_k$, we have 
\begin{align*}
\proba^{\bm{\theta_{k,n}}}(\bm \theta_{k,n} = \pi) & = 
\proba^{\Mk,\bm\sigma_n}(\Perm(\Mk, \mu_{\bm\sigma_n}) = \pi)
=  \esper^{\bm\sigma_n} \left[ \occ(\pi,\mu_{\bm\sigma_n})\right] \\
& = \proba^{{\bm I}_{n,k},\bm\sigma_n}(\pat_{{\bm I}_{n,k}}(\bm\sigma_n) = \pi) + \mathcal{O}(1/n) \rightarrow \proba(\bm \rho_k = \pi).
\end{align*}

In other words, $\bm \theta_{k,n} \stackrel{d}{\to} \bm\rho_k$. Since $\mu_{\bm\rho_k}$ takes its values in a finite set of permutons, this implies
\begin{equation}\label{Eq:mumuSubPermConverge}
\mu_{\bm\theta_{k,n}} \stackrel{d}{\to} \mu_{\bm\rho_k}.
\end{equation}

Let $H : (\mathcal M,d_{\square}) \to \R$ be a bounded continuous functional. 
It holds that 
\begin{align*}
\left| \esper\left[H(\mu_{\bm\sigma_n})\right] - \esper\left[H(\mu_{\bm\theta_{k,n}})\right]\right| \leq & \ \esper\left[\,\left|H(\mu_{\bm\sigma_n})- H(\mu_{\bm\theta_{k,n}})\right|\,\right] \\
\leq & \ \esper\left[\,\left|H(\mu_{\bm\sigma_n})- H(\mu_{\bm\theta_{k,n}})\right|\,\mathbf{1}_{d_{\square}\left(\mu_{\bm\sigma_n}, \mu_{\bm\theta_{k,n}}\right)\leq 16k^{-1/4} }\right] \\
& + \esper\left[\,\left|H(\mu_{\bm\sigma_n})- H(\mu_{\bm\theta_{k,n}})\right|\,\mathbf{1}_{d_{\square}\left(\mu_{\bm\sigma_n}, \mu_{\bm\theta_{k,n}}\right)> 16k^{-1/4} }\right]. 
\end{align*}
The first term can be bounded by introducing the modulus of continuity of $H$, which is defined as \hbox{$\omega(\eps) = \sup_{d_{\square}(\xi,\zeta) \leq \eps} |H(\xi)-H(\zeta)|$}. 
Since $\mathcal M$ is compact, it goes to $0$ when $\eps$ goes to $0$. 
Hence,
\begin{align*}
& \esper\left[\,\left|H(\mu_{\bm\sigma_n})- H(\mu_{\bm\theta_{k,n}})\right|\,\mathbf{1}_{d_{\square}\left(\mu_{\bm\sigma_n}, \mu_{\bm\theta_{k,n}}\right)\leq 16k^{-1/4} }\right] \\
& \qquad \leq \esper\left[ \omega\left( d_{\square}\left(\mu_{\bm\sigma_n}, \mu_{\bm\theta_{k,n}}\right)\right)\mathbf{1}_{d_{\square}\left(\mu_{\bm\sigma_n}, \mu_{\bm\theta_{k,n}}\right)\leq 16k^{-1/4} }\right] 
\leq \omega\left(16k^{-1/4}\right).
\end{align*}
As for the second term, for $k$ large enough, \cref{lem:subpermapproxrandompermuton} yields 
\begin{align*}
&\esper\left[\,\left|H(\mu_{\bm\sigma_n})- H(\mu_{\bm\theta_{k,n}})\right|\,\mathbf{1}_{d_{\square}\left(\mu_{\bm\sigma_n}, \mu_{\bm\theta_{k,n}}\right)> 16k^{-1/4} }\right] \\ 
& \qquad \leq \esper\left[ 2\,\sup |H|\ \mathbf{1}_{d_{\square}\left(\mu_{\bm\sigma_n}, \mu_{\bm\theta_{k,n}}\right)> 16k^{-1/4} }\right] 
\leq \frac12 e^{-\sqrt{k}}\, 2\sup |H|.
\end{align*}
Putting things together, we obtain 
\begin{equation}
\left| \esper\left[H(\mu_{\bm\sigma_n})\right] - \esper\left[H(\mu_{\bm\theta_{k,n}})\right]\right| \leq \omega\left(16k^{-1/4}\right) +  \frac12 e^{-\sqrt{k}}\, 2\sup |H|.
\label{eq:Tech9}
\end{equation}

Assume that $(\mu_{\bm\sigma_n})_n$ has a subsequence converging in distribution to a random permuton $\bm \mu'$.
Taking the limit when $n \to \infty$ of \eqref{eq:Tech9} along this subsequence, we get
\[\big| \esper[H(\bm \mu')] - \esper[H(\mu_{\bm\rho_k})]\big| \leq \omega\left(16k^{-1/4}\right) +  e^{-\sqrt{k}} \sup |H|.\]
(Recall indeed that $({\bm\theta_{k,n}})_n$ converges to ${\bm\rho_k}$ in distribution.)
The right-hand side tends to $0$ when $k$ tends to infinity,
which proves that $(\mu_{\bm\rho_k})_k$ converges to $\bm \mu'$ in distribution as well.

Therefore, all converging subsequences of $(\mu_{\bm\sigma_n})_n$ converge to the same limit $\bm \mu'$,
which is the limit of $(\mu_{\bm\rho_k})_{k \ge 1}$. 
Thanks to the compactness of the space of probability distributions on $\mathcal{M}$,
this is enough to conclude that $(\mu_{\bm\sigma_n})$ has indeed a limit.
Item (a) is proved.

\noindent{\bf Proof of additional statements. }
Assume that (a)--(d) hold. 
That $(\bm \Lambda_\pi)_\pi \stackrel d = (\occ(\pi,\bm \mu))_\pi$ follows from the proof of (a)$\Rightarrow$(b). 
Fix any integer $k$, and any permutation $\pi$ of size $k$. 
The above equality in distribution implies $\esper[\bm \Lambda_\pi] = \esper[\occ(\pi,\bm\mu)]$.
That $\Delta_\pi = \esper[\bm \Lambda_\pi]$ is clear from the proof of (b)$\Rightarrow$(c).
The equality $\proba(\bm \rho_{k} = \pi) = \Delta_{\pi}$ follows from the proof of (c)$\Rightarrow$(d).
Finally, $\esper[\occ(\pi,\bm\mu)] = \proba(\Perm(\Mk,\bm\mu) = \pi)$ comes from \cref{eq:E(occ)=P(perm)}. \qed
\begin{remark}
  In some sense, \cref{thm:randompermutonthm} can be seen as an analogue of a theorem of Aldous for random trees
  \cite[Theorem 18]{AldousCRT3}.
  Both in permutations and trees, there is a natural way to construct a smaller structure
  from $k$ elements of a big structure (induced subpermutations or subtrees).
  The goal is then to reduce the convergence of the big structure to the convergence, for each $k$,
  of the induced substructures. 
  For trees, we need an extra tightness assumption (that the family of trees is ``leaf-tight'' in Aldous' terminology).
  In our case, since the space of permutons is compact,
  we do not need such an assumption.
\end{remark}

We finish this section by a comment on the existence of random permutons
with prescribed induced subpermutations. 
\begin{definition}
	A family of random permutations $(\bm \rho_n)_n$ is \emph{consistent} if
	\begin{enumerate}
		\item for every $n\geq 1$, $\bm \rho_n \in \Sn_n$,
	    \item for every $n\geq k \geq 1$, if $\bm I_{n,k}$ is a uniform subset of $[n]$ of size $k$, independent of $\bm \rho_n$, then $\pat_{{\bm I}_{n,k}}(\bm \rho_n) \stackrel d = \bm \rho_k$.
    \end{enumerate} 
	\label{Def:consistency}
\end{definition}
It turns out that consistent family of random permutations and random permutons are essentially equivalent:
\begin{proposition}
	If $\Mu$ is a random permuton, then the family defined by $\bm \rho_k \stackrel d = \Perm(\Mk,\Mu)$ is consistent. Conversely, for every consistent family of random permutations $(\bm \rho_k)_{k\geq 1}$, there exists a random permuton $\bm \mu$ whose distribution is uniquely determined, such that $\Perm(\Mk,\Mu) \stackrel d = \bm\rho_k$. In that case, $\mu_{\bm \rho_n} \xrightarrow[n\to\infty]{d} \bm \mu$.
	\label{Prop:existence_permuton}
\end{proposition}
\begin{proof}
	Set $n\geq k \geq 1$. The first assertion follows from the following coupled construction of ${\vec{\mathbf{m}}_n}$ and $\Mk$ : $\Mk$ is a uniform random subset of ${\vec{\mathbf{m}}_n}$, chosen independently of it. It follows that $\Perm(\Mk, \Mu) = \pat_{{\bm I}_{n,k}}(\Perm({\vec{\mathbf{m}}_n},\Mu))$, for some random subset ${\bm I}_{n,k}$ of $[n]$.
By construction, the distribution of ${\bm I}_{n,k}$ is uniform and independent of $\Perm({\vec{\mathbf{m}}_n},\Mu)$. Hence the consistency follows.
	
	The converse is immediate, by applying the implication (d)$\rightarrow$(a) 
    and the last assertion of \cref{thm:randompermutonthm} to the sequence $(\bm \rho_k)_{k\geq 1}$.
    Consistency ensures that we get the prescribed induced subpermutations,
    and uniqueness in distribution follows by \cref{Prop:CaracterisationLoiPermuton}.
\end{proof}
\section{Coding permutations by trees}\label{Sec:CodingByTrees}

\subsection{Substitution trees}
\label{sec:SubsTrees}
As seen in \cref{sec:OperatorsPermutations} (\cref{Th:AlbertAtkinson}), any permutation $\sigma$ can be recursively decomposed using substitutions in a canonical way
and this decomposition can be encoded in a canonical tree.
However, if we do not impose conditions on $\theta$ and the $\pi^{(i)}$'s (as done in \cref{Th:AlbertAtkinson}),
a permutation $\sigma$ may be represented in many ways as a substitution $\sigma=\theta[\pi^{(1)},\dots,\pi^{(d)}]$,
where the $\pi^{(i)}$'s themselves may be further decomposed using substitutions.
Such decompositions can be recorded in {\em substitution trees}.

\begin{definition}\label{dfn:trees}
A \emph{rooted planar tree} is either a leaf, or consists of a root node $\varnothing$ 
with an ordered $k$-tuple of subtrees attached to the root,
which are themselves rooted planar trees.\\
In our context, the \emph{size} of a tree $t$ is its number of leaves. 
It is denoted $|t|$, 
whereas $\#t$ denotes the number of nodes of $t$ (including both leaves and internal nodes). 
\end{definition}

Internal vertices of all trees considered in this paper have degree at least $2$.
It is natural (and also convenient for counting purposes in \cref{Sec:Enumeration}) to consider that 
the single leaf of the tree of size $1$ is also its root (and is therefore also denoted $\varnothing$). 

Since we work with \emph{planar} trees, we can label their leaves canonically with the integers from $1$ to $|t|$: the leaf labeled by $i$ is the $i$th leaf met in the depth-first traversal of $t$ which choses left before right.
A subset of the set of leaves of a tree $t$ is therefore canonically represented by a subset $I$ of $[|t|]$.

\begin{definition}\label{def:SubstitutionTree}
A \emph{substitution tree} of size $n$ is a labeled rooted planar tree with $n$ leaves,
where any internal node with $k \ge 2$ children is labeled by a permutation of size $k$.
Internal nodes with only one child are forbidden.
\end{definition}

Internal nodes labeled by the ascending permutation $1 2 \cdots r$ or the descending permutation $r \cdots 2 1$ (for some $r \ge 2$)
will play a particular role.
Therefore we replace {\em every} such label with a $\oplus$ (for ascending permutations) or a $\ominus$ (for descending permutations).
Since the size of a label corresponds to the number of children and since there is exactly one ascending (resp. descending) permutation
of each size, there is no loss of information in this replacement.
Internal nodes labeled $\oplus$ or  $\ominus$ are called {\em linear nodes}, the other nodes being called {\em nonlinear}.
Among nonlinear nodes, the ones labeled by simple permutations are called {\em simple nodes}.

An example of substitution tree is shown in \cref{fig:ExemplePermutationAssocieeAUnArbre}, left.

\begin{definition}\label{Def:PermTree}
Let $t$ be a substitution tree. We define inductively the permutation $\perm(t)$ associated with $t$:
\begin{itemize}
	\item if $t$ is just a leaf, then $\perm(t)=1$;
	\item if the root of $t$ has $r\geq 2$ children with corresponding subtrees $t_1,\ldots,t_r$ (from left to right), and is labeled with the permutation $\theta$, then $\perm(t)$ is the permutation obtained as the substitution of $\perm(t_1),\dots,\perm(t_r)$ in $\theta$:
	\[\perm(t) = \theta[\perm(t_1),\ldots,\perm(t_r)].\]
\end{itemize}
\end{definition}

\begin{figure}[htbp]
    \begin{center}
      \includegraphics[width=10cm]{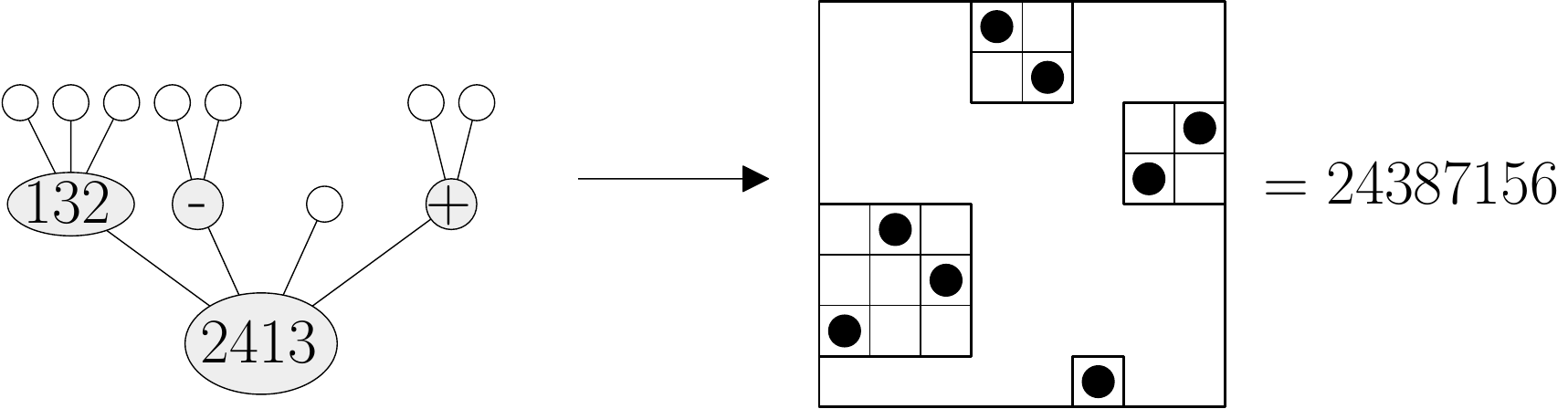}
    \end{center}
    \caption{A substitution tree encoding a permutation.}
    \label{fig:ExemplePermutationAssocieeAUnArbre}
\end{figure}

\cref{fig:ExemplePermutationAssocieeAUnArbre} illustrates this construction. 
When $\perm(t)=\sigma$, when say that $t$ is a tree that \emph{encodes} $\sigma$, or a tree \emph{associated with} $\sigma$. 
Nonsimple permutations $\sigma$ are encoded by several trees $t$.
However, if we restrict ourselves to canonical trees (which are particular cases of substitution trees; see \cref{defintro:CanonicalTree}),
we have uniqueness.
Indeed, from \cref{Th:AlbertAtkinson}, to any permutation $\sigma$ we can associate uniquely a canonical tree $t$ such that $\perm(t)=\sigma$.
\bigskip

The remaining of \cref{sec:SubsTrees} is devoted to the proof of simple combinatorial lemmas
on the structure of the set of substitution trees associated with a given permutation $\si$.
These lemmas are useful in \cref{Sec:Standard}.
\medskip

We first make the following observation.
Take a substitution tree $\tau$ of some permutation $\pi$ with a marked node $v$ labeled by $\theta$.
Consider also a substitution tree $\tau'$ of $\theta$.
Then replacing $v$ by the tree $\tau'$ yields a new substitution tree $\tau''$ of {\em the same permutation} $\pi$.
(When doing this replacement the $|\theta|$ subtrees attached to $v$ are glued on the leaves of $\tau'$,
respecting their order, see \cref{fig:inflation}.)
This operation will be referred to as {\em the inflation of $v$ with $\tau'$}.
\begin{figure}[hthb]
  \begin{center}
    \includegraphics[scale=.8]{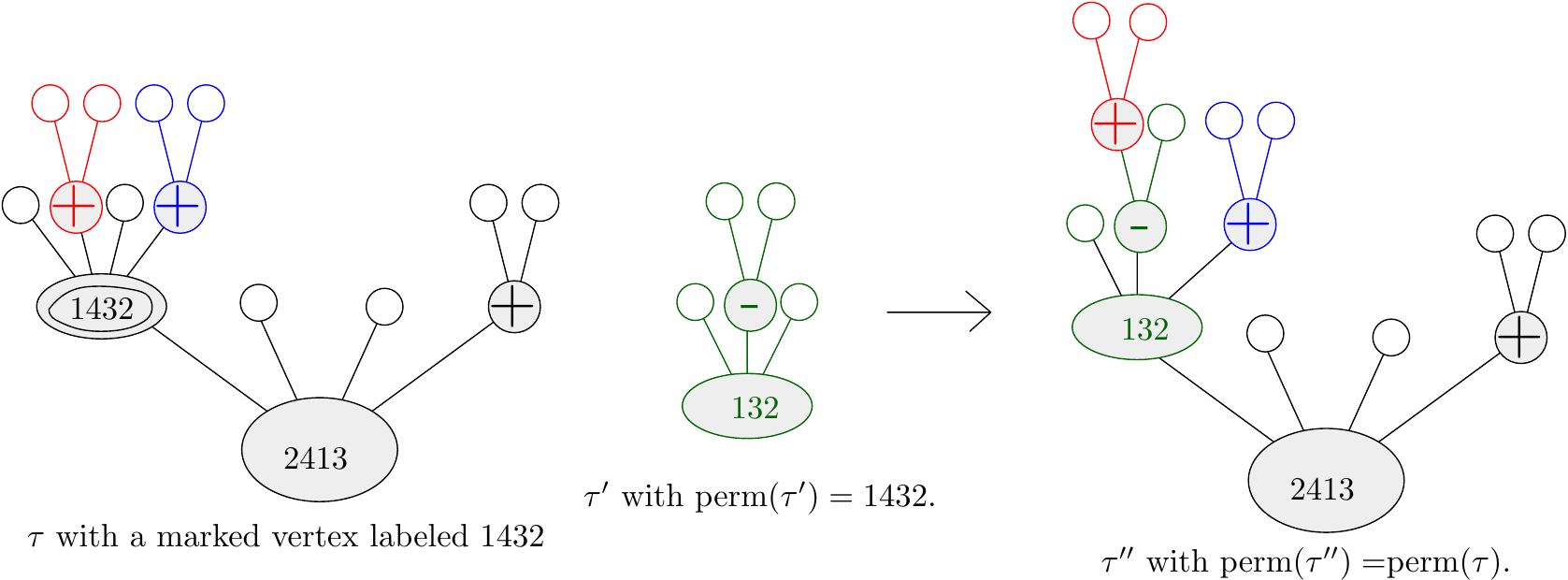}
  \end{center}
  \caption{Illustration of the inflation procedure (best seen with colors).}
  \label{fig:inflation}
\end{figure}

Conversely, consider a connected set $A$ of internal nodes in a substitution tree $\tau''$ of $\pi$.
From this set we build a substitution tree $\tau'$ whose set of internal nodes is $A$,
the ancestor-descendant relation in $\tau'$ is inherited from the one in $\tau''$,
and we add leaves so that the degree of each node of $A$ is the same in $\tau'$ than in $\tau''$.
We denote $\theta=\perm(\tau')$.
Then merging all nodes in $A$ into a single node labeled by $\theta$
turns $\tau''$ into a new substitution tree $\tau$ of {\em the same permutation} $\pi$.
We call this a {\em merge operation}.
For example, the tree $\tau$ of \cref{fig:inflation} can be obtained from the tree $\tau''$ of the same figure
by merging the nodes labeled $132$ and $\ominus$.
\medskip

We now consider a last family of substitution trees.
An \textit{expanded tree} is a substitution tree where nonlinear nodes are labeled by simple permutations,
while linear nodes are required to be binary.

\begin{lemma}
  \label{lem:fromCanonicalToExpanded}
Any expanded tree of $\pi$ is obtained from its canonical tree by inflating all nodes 
labeled by $\oplus$ (resp. $\ominus$)
with binary trees whose internal nodes are all labeled by $\oplus$ (resp. $\ominus$)
\end{lemma}
\begin{proof}
  Let $\tau$ be an expanded tree of $\pi$. 
  Consider, if any, two adjacent linear nodes of $\tau$ with the same label (either both $\oplus$ or both $\ominus$)
  and merge them. Note that the resulting node will still have label $\oplus$ or $\ominus$.
  We repeat this operation until there is no adjacent linear nodes with the same label.
  Nonlinear nodes in the resulting tree $\tau'$ are all labeled by simple permutations:
  it is the case in $\tau$ (by definition of expanded trees) and we did not create any new nonlinear nodes.
  Therefore $\tau'$ satisfy all conditions of canonical trees (see \cref{defintro:CanonicalTree}).
  By uniqueness, $\tau'$ is {\em the} canonical tree of $\pi$.
  Reversing the merge operations, $\tau$ can be obtained from $\tau'$ by inflating its linear nodes,
  which proves the proposition.
\end{proof}

We recall a fact well-known to combinatorialists:
 the number of complete binary trees ({\em i.e.} plane rooted trees, whose internal vertices have all degree $2$)
 with $d$ leaves is $\Cat_{d-1}$. Therefore each linear node of degree $d$ of the canonical tree,
 can be inflated with a binary tree in $\Cat_{d-1}$ ways.
We therefore get the following interesting corollary, regarding the number and properties of expanded trees.
\begin{corollary}
  \label{cor:OnExpandedTrees}
 Let $\pi$ be a permutation and $d_1,\cdots,d_r$ (resp. $e_1,\cdots,e_s$) be the degrees of the nodes
 labeled $\oplus$ (resp. $\ominus$)
 in the canonical tree of $\pi$. Then
 \begin{itemize}
   \item the number $\widetilde{N_\pi}$ of expanded trees of $\pi$ is $\prod_{i=1}^r \Cat_{d_i-1}\, \prod_{j=1}^s \Cat_{e_j-1}$,
where we denote by $\Cat_{k} := \frac{1}{k+1}\binom{2k}{k}$ the $k$-th Catalan number, which counts complete binary trees with $k$ leaves.
   \item each expanded tree of $\pi$ has $\sum_{i=1}^r (d_i-1)$ nodes labeled $\oplus$
     and $\sum_{j=1}^s (e_j-1)$ nodes labeled $\ominus$.
   \item the labels of the nonlinear nodes in any expanded tree of $\tau$
     are the same as in its canonical tree.
 \end{itemize}
\end{corollary}

\begin{lemma}
  \label{lem:fromExpandedToAll}
  Any substitution tree of $\pi$ can be obtained from some expanded tree of $\pi$
  by merge operations.
\end{lemma}
\begin{proof}
  The proof is similar to that of \cref{lem:fromCanonicalToExpanded}.
  Starting from any substitution tree of $\pi$ and %
  inflating every node that is neither simple nor binary by an expanded tree encoding its label,
  we get an expanded tree. Reversing these inflation operations,
  we can obtain any substitution tree from some expanded tree  of $\pi$,
  using only merge operations.
\end{proof}

\subsection{Induced trees}

Since permutations are encoded by trees and since we are interested in patterns in permutations,
we consider an analogue of patterns in trees: this leads to the notion of {\em induced trees}.

\begin{definition}[First common ancestor]\label{dfn:common_ancestor}
Let $t$ be a tree, and $u$ and $v$ be two nodes (internal nodes or leaves) of $t$. 
The \emph{first common ancestor} of $u$ and $v$ is the node furthest away from the root $\varnothing$ that appears 
on both paths from $\varnothing$ to $u$ and from $\varnothing$ to $v$ in $t$. 
\end{definition}

The following simple observation  allows to read the relative order of $\sigma_i$ and $\sigma_j$ in any substitution tree encoding $\sigma$.
\begin{observation}\label{obs:caract_perm_tau}
Let $i \neq j$ be two leaves of a substitution tree $t$ and $\sigma=\perm(t)$. Let $v$ be the first common ancestor of $i,j$ in $t$ and $\theta$ be the permutation labeling $v$.
We define $k$ (resp. $\ell$) such that the $k$-th (resp. $\ell$-th) child of $v$ is an ancestor of $i$ (resp. $j$).

Then $\sigma_i>\sigma_j$ if and only if $\theta_k>\theta_\ell$.
\end{observation}

\begin{definition}[Induced tree]\label{dfn:induced_subtree}
Let $t$ be a substitution tree, and let $I$ be a subset of the leaves of $t$.
The tree $t_I$ induced by $I$ is the substitution tree of size $|I|$ defined as follows.
The tree structure of $t_I$ is given by:
\begin{itemize}
 \item the leaves of $t_I$ are the leaves of $t$ labeled by elements of $I$; 
 \item the internal nodes of $t_I$ are the nodes of $t$ that are first common ancestors of two (or more) leaves in $I$; 
 \item the ancestor-descendant relation in $t_I$ is inherited from the one in $t$; 
 \item the order between the children of an internal node of $t_I$ is inherited from $t$. 
\end{itemize}
The label of an internal node $v$ of $t_I$ is defined as follows:
\begin{itemize}
 \item if $v$ is labeled by a permutation $\theta$ in $t$, the label of $v$ in $t_I$ is given by the pattern of $\theta$ 
 induced by the children of $v$ having a descendant that belongs to $t_I$ (or equivalently, to $I$).
\end{itemize}
\end{definition}
A detailed example of the induced tree construction is given in \cref{fig:ExampleCanonicalTree}.
\begin{figure}[htbp]
    \begin{center}
      \includegraphics[width=12cm]{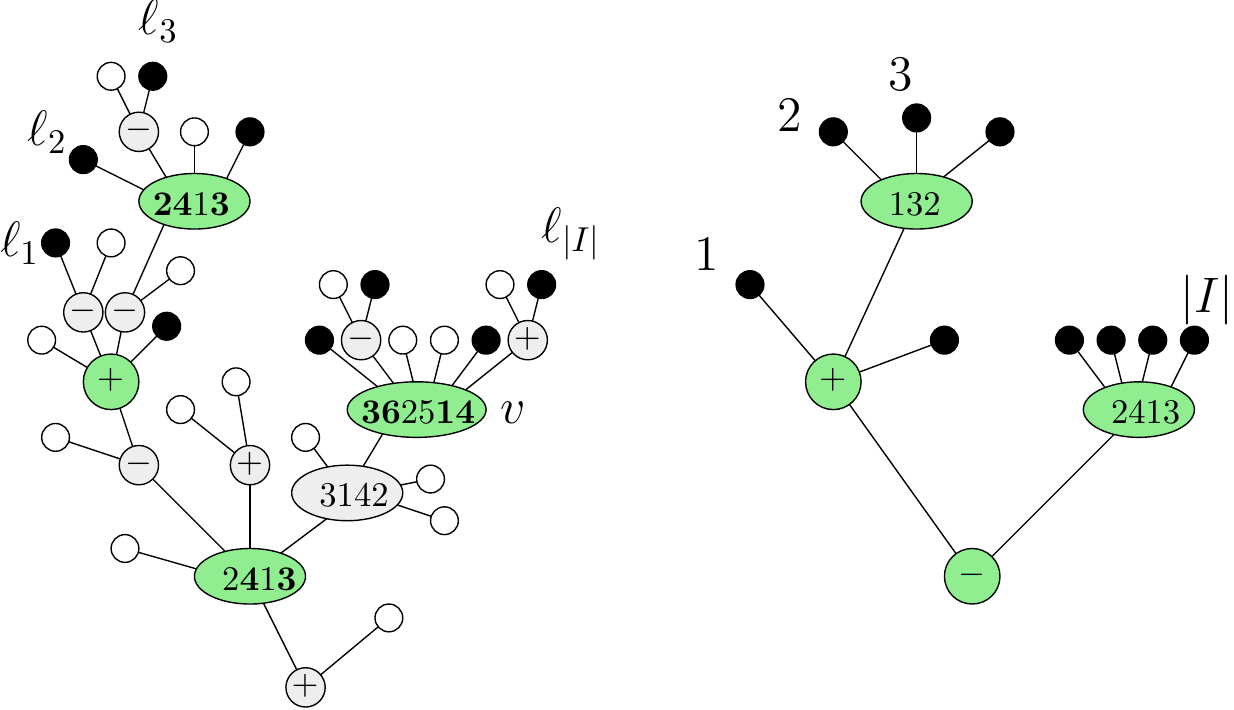}
    \end{center}
\caption{On the left: A substitution tree $t$ of size $n=26$ (which happens to be a canonical tree), where leaves are indicated both by $\circ$ and $\bullet$. 
Among these $26$ leaves, $|I|=9$ leaves are marked and indicated by $\bullet$. 
In green are shown the internal nodes of $t$ which are first common ancestors of these $9$ marked leaves.
On the right: The substitution tree induced by the $9$ marked leaves.
Observe that the node $v$ labeled by $362514$ in $t$ is labeled by $2413$ in $t_I$. 
This is because only the first, second, fifth and sixth children of $v$ have descendants that belong to $I$, and $\pat_{\{1,2,5,6\}}(362514)=2413$.
The induced tree is not canonical since $132$ is not simple.}
    \label{fig:ExampleCanonicalTree}
\end{figure}
Note that if $v$ has label $\oplus$ in $t$, it has also label $\oplus$ in $t_I$.
Indeed, $\oplus$ nodes correspond to increasing permutations and all patterns of increasing permutations
are increasing permutations. The same holds with $\ominus$. The converse is however not true:
a node can be linear in $t_I$ but nonlinear in $t$ ( \emph{e.g.} the bottommost green node in \cref{fig:ExampleCanonicalTree}).

\begin{observation}
By definition, for any substitution tree $t$ with $k$ leaves and subset $I$ of $[k]$, $t_I$ is a substitution tree.
However, if $t$ is a canonical tree, $t_I$ is a substitution tree which is not necessarily canonical
(see for example \cref{fig:ExampleCanonicalTree}).
\end{observation}

An important feature of induced trees is the following, which follows from \cref{obs:caract_perm_tau} and is illustrated in \cref{fig:DiagrammeCommutatif}.
\begin{lemma}\label{lem:DiagrammeCommutatif} 
Let $t$ be a substitution tree with $k$ leaves, and $I$ be a subset of $[k]$.
We have 
$$
\pat_I(\perm(t)) = \perm(t_I).
$$
\end{lemma}

As a consequence of this formula, counting the total number of occurrences of a given pattern
in some family of permutations can be reduced to counting the total number of induced trees
equal to a given $t_0$ in the corresponding family of canonical trees.
This is precisely the goal of the next section.

\begin{figure}[htbp]
    \begin{center}
      \includegraphics[width=10cm]{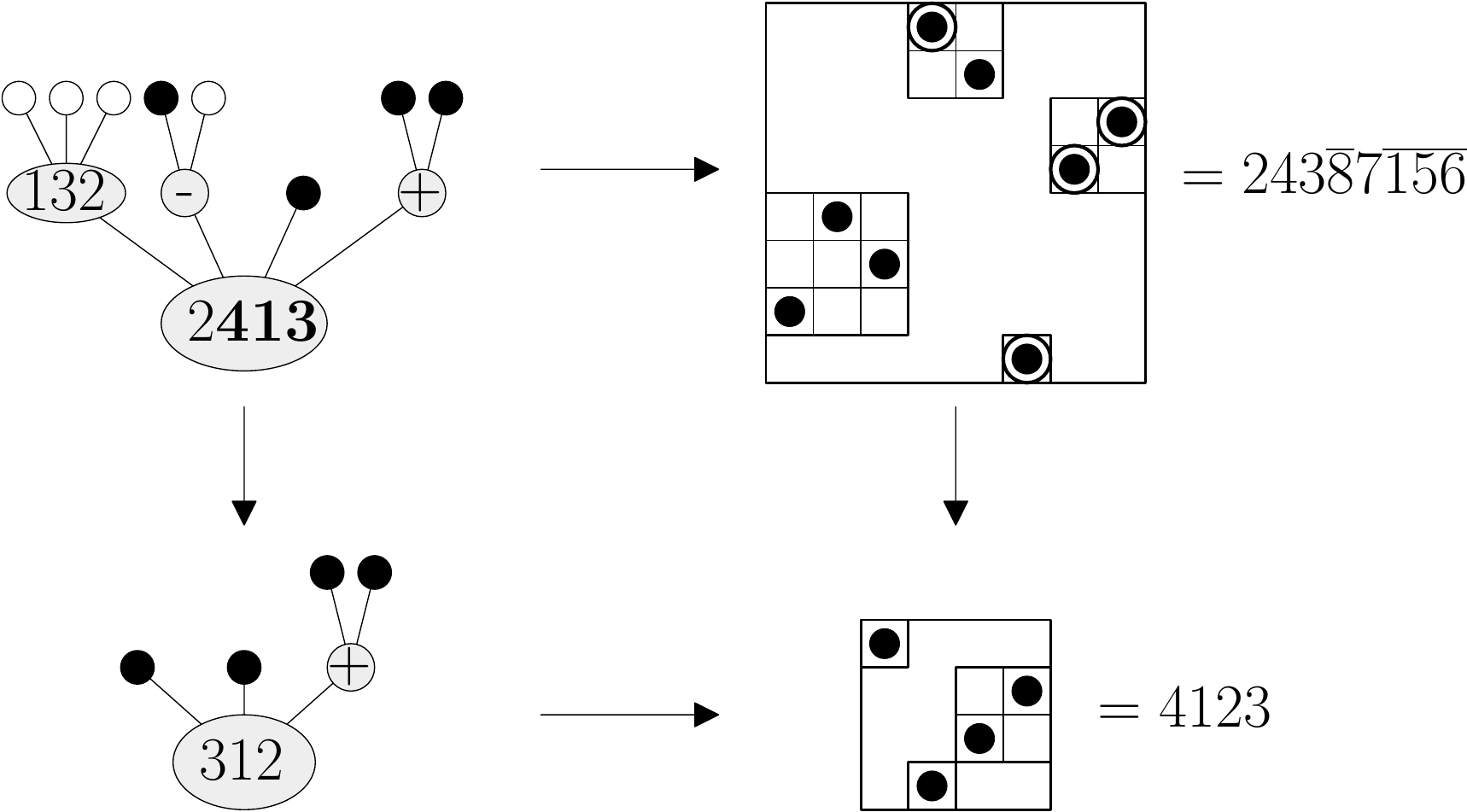}
    \end{center}
    \caption{Illustration of \cref{lem:DiagrammeCommutatif}. 
    On the top: A substitution tree $t$ with $|I|$ marked leaves (in this example $k=8$ and $I=\{4,6,7,8\}$), and the permutation $\perm(t)$ it encodes, with the corresponding $|I|$ marked elements (at positions in $I$). 
    On the bottom: The induced tree $t_I$ and the induced pattern $\pat_I(\perm(t)) = \perm(t_I)$.}
    \label{fig:DiagrammeCommutatif}
\end{figure}

\newpage
\section{Exact enumeration of various families of trees}
\label{Sec:Enumeration}
Let $\mathcal{S}$ be a fixed family of simple permutations. Recall that its generating function is
$$
S(z)=\sum_{\alpha \in \mathcal{S}} z^{|\alpha|}= \sum_{n\geq 4} s_nz^{n},
$$
where $s_n$ is the number of permutations of size $n$ in $\mathcal{S}$. %
An {\em $\mathcal{S}$-canonical tree} is any canonical tree whose simple nodes carry labels in $\mathcal S$. 
We denote by $\mathcal{T}$ the combinatorial class of $\mathcal{S}$-canonical trees,
the size of $|t|$ a tree $t$ being its number of leaves.
Recall that $\langle \mathcal S\rangle$ is by definition the set of permutations whose canonical tree is in $\mathcal{T}$.
Since canonical trees encode permutations in a unique way,
$\perm$ defines a size-preserving bijection between  $\mathcal{T}$ and $\langle \mathcal S\rangle$.
Both have therefore the same generating function which we denote by
$$
T(z)=\sum_{t\in \mathcal{T}}z^{|t|}= \sum_{\sigma\in\langle \mathcal{S}\rangle}z^{|\sigma|}.
$$

In \cref{Subsec:CombSyst} below, we explain how to compute $T(z)$ starting from the datum $S(z)$.
We then study families of $\mathcal{S}$-canonical trees with one marked leaf,
with constraints on the root and/or on the marked leaf.
These are building blocks for \cref{Subsec:DecompTrees},
where we consider the family of $\mathcal{S}$-canonical trees with $k$ marked leaves,
inducing a given tree $t_0$.

\subsection{Generating functions of $\mathcal{S}$-canonical trees (possibly with marked leaves)}
\label{Subsec:CombSyst}
In order to compute $T(z)$ in terms of $S(z)$,
we need to introduce the auxiliary family $\mathcal{T}_{\nonp}$ (resp. $\mathcal{T}_{\nonm}$) of $\mathcal{S}$-canonical trees with a root (always denoted $\varnothing$) that is \textbf{not} labeled $\oplus$ (resp. $\ominus$), 
and its generating function $T_{\nonp}$ (resp. $T_{\nonm}$):
$$
T_\nonp(z)=\sum_{\substack{t\in \mathcal{T};\\ \varnothing\text{ is not labeled }\oplus}}z^{|t|}\,.
$$
Note that replacing all labels $\ominus$ by $\oplus$ and $\oplus$ by $\ominus$ 
defines an involution on $\mathcal S$-canonical trees.
This implies in particular $T_{\nonp}=T_{\nonm}$ and will be used to get other similar identities below. 

\begin{proposition}
  Together with the condition $T_{\nonp}(0)=0$,
  the generating function  $T_{\nonp}$ is determined by the following implicit equation
  \begin{equation}
    T_{\nonp}= z + \frac{T_{\nonp}^2}{1-T_{\nonp}} + S \left(\frac{T_{\nonp}}{1-T_{\nonp}}\right).
    \label{eq:Tnonp}
  \end{equation}
  The main series $T$ is then simply given in terms of $\, T_{\nonp}$ by
  \begin{equation}
    T= \frac{T_{\nonp}}{1-T_{\nonp}}.
    \label{eq:T_Tnp}
  \end{equation}
\label{Prop:systeme1}
\end{proposition}

\begin{proof}
A tree of $\mathcal{T}$ is either a leaf, or a root labeled $\oplus$ and a sequence of at least two trees in $\mathcal{T}_{\nonp}$, or a root labeled  $\ominus$ and a sequence of at least two trees in $\mathcal{T}_{\nonm}$, or a root labeled by $\alpha \in \mathcal{S}$ and a sequence of $|\alpha|$ unconstrained trees. Therefore
$$
T=z+ \frac{T_{\nonp}^2}{1-T_{\nonp}}+\frac{T_{\nonm}^2}{1-T_{\nonm}} + S(T)=z+ 2\,\frac{T_{\nonp}^2}{1-T_{\nonp}} + S(T)
$$
Similarly,
\begin{equation}
T_{\nonp}=T_{\nonm}=z+ \frac{T_{\nonp}^2}{1-T_{\nonp}} + S(T).
\label{eq:Tech10}
\end{equation}

By combining these two equations we get
$T=T_{\nonp} + \frac{T_{\nonp}^2}{1-T_{\nonp}}= \frac{T_{\nonp}}{1-T_{\nonp}}$, that is \cref{eq:T_Tnp}.
Substituting it back in \cref{eq:Tech10} gives \cref{eq:Tnonp}.

Observe, that under the assumption $T_{\nonp}(0)=0$, \cref{eq:Tnonp} allows one to compute
  inductively the coefficients of $T_{\nonp}$.
  Hence $T_{\nonp}$ is uniquely determined by $T_{\nonp}(0)=0$ and \cref{eq:Tnonp},
  as claimed.
\end{proof}

We now consider trees with %
a marked leaf. As before, subscripts indicate a constraint on the root.
The generating function of trees with a marked leaf counted by their number of {\bf unmarked} leaves
is obtained by differentiating the generating function of trees without marked leaf: $T'$, $T'_{\nonp}$,  $T'_{\nonm}$. Indeed,
$$
T'_\nonp(z)=\sum_{\substack{t\in \mathcal{T};\\ \varnothing\text{ is not labeled }\oplus}}|t|\ z^{|t|-1}
=\sum_{\substack{\hat{t} \text{ obtained by marking a leaf }\\ \text{ from a tree of }\mathcal{T} \text{ of root not labeled }\oplus \\ }}z^{\#\text{unmarked leaves of }\\ \hat{t}}\ .
$$
Accordingly, we denote by $\mathcal{T}'$, $\mathcal{T}'_{\nonp}$ and $\mathcal{T}'_{\nonm}$ the families of trees counted by $T'$, $T'_{\nonp}$ and $T'_{\nonm}$. 

Consistently, we use superscripts when we consider families of trees with a marked leaf that satisfies an additional constraint, and similarly for their generating function.
We say that a leaf is $\oplus$-replaceable (resp. $\ominus$-replaceable)
if it may  be replaced by a tree whose root is labeled $\oplus$ (resp. $\ominus$)
without violating the definition of canonical trees (see the third item in \cref{defintro:CanonicalTree}).
In other words, its parent (if it exists) should be labeled by $\ominus$ or by a simple permutation
(resp. by $\oplus$ or by a simple permutation).
We then denote $\mathcal{T}_{\nonp}^+$ (resp. $\mathcal{T}_{\nonp}^-$)
the families of trees in $\mathcal{T}_{\nonp}$ with a $\oplus$-replaceable marked leaf
(resp. a  $\ominus$-replaceable marked leaf).
Similar definitions hold for $\mathcal{T}^+$, $\mathcal{T}^-$, $\mathcal{T}_{\nonm}^+$ and $\mathcal{T}_{\nonm}^+$.

As for $T'$, we take the convention that $T_{\nonp}^+$ and all generating functions with superscript
count trees  according to the number of {\bf unmarked} leaves.
By definition, $T_{\nonp}^-$ has constant coefficient $1$ (corresponding to the tree consisting of a single leaf). 
We however take the convention that $T_{\nonp}^+$ has constant coefficient $0$: 
in other words, the single leaf is excluded from the family $\mathcal{T}_{\nonp}^+$
(intuitively, a single leaf cannot be replaced by a tree with root labeled $\oplus$,
since the trees in $\mathcal{T}_{\nonp}$ should not have a root labeled $\oplus$).

\begin{proposition}
  The generating functions $T^+$, $T_{\nonm}^{+}$ and $T_{\nonp}^+$ are given by the following formulas: 
  \begin{align}
    \label{eq:tp} T^+&=\frac{1}{1- \W S'(T) - \W - S'(T)}; \\
\label{eq:tpm} T_{\nonm}^{+} &=\frac{1}{1+\W} \, T^+ ;  \\
\label{eq:tpp} T_{\nonp}^{+} & = (\W S'(T) + \W + S'(T)) T_{\nonm}^+
  \end{align}
  where $T$ and $T_{\nonp}$ are given by \cref{Prop:systeme1} 
  and $\W=(\tfrac{1}{1-T_{\nonp}})^2-1$.
 
Quantities with a minus superscript are obtained by symmetry: $T^-=T^+$, $T_{\nonm}^{-}=T_{\nonp}^{+}$ and  $T_{\nonp}^{-}=T_{\nonm}^{+}$.
\label{Prop:Sol_TPlus}
\end{proposition}

\begin{proof}
Consider a tree $t$ in $\mathcal{T}_{\nonp}^{+}$. As explained above, $|t| \neq 1$ and we distinguish cases according to the label of the root of $t$, 
which may be either $\ominus$ or a simple permutation. 

\begin{figure}[htbp]
 \includegraphics[width=12cm]{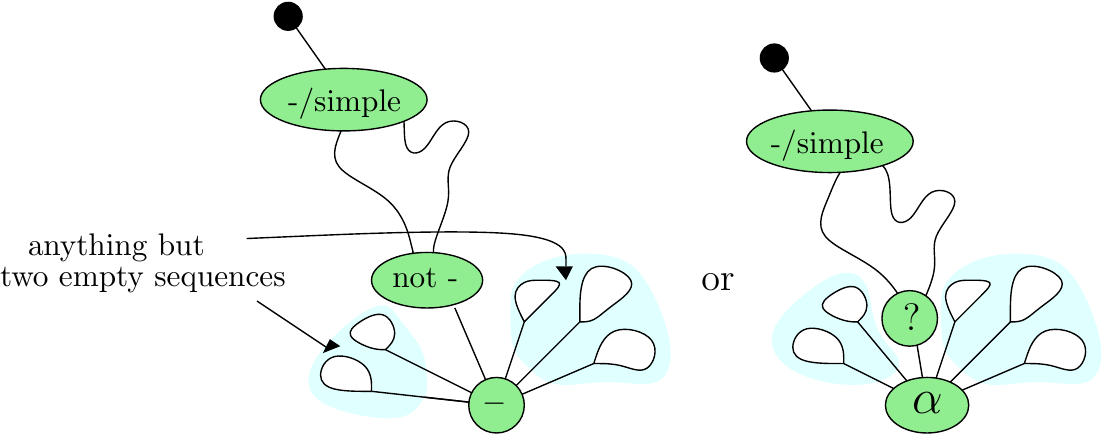}
 \caption{Proof of \cref{Prop:Sol_TPlus}.}
 \label{fig:proof_first_system}
\end{figure}

\begin{enumerate}
\item The root of $t$ is labeled $\ominus$ (see left of \cref{fig:proof_first_system}). Then $t$ can be decomposed as a tree in $\mathcal{T}_{\nonm}^{+}$ (which may be a single leaf) and a nonempty pair of sequences of unmarked trees in $\mathcal{T}_{\nonm}$.
\item The root of $t$ is labeled by a simple permutation $\alpha\in\mathcal{S}$ of size $d$ (see right of \cref{fig:proof_first_system}). Then $t$ can be decomposed as a $d$-uple of unconstrained trees, with one of them having a $\oplus$-replaceable marked leaf.
\end{enumerate} 
Therefore we have
$$
T_{\nonp}^{+} = \W T_{\nonm}^{+} + S'(T) \, T^+,
$$
where $\W=(\tfrac{1}{1-T_{\nonp}})^2-1$ counts nonempty pairs of sequences of unmarked trees in $\mathcal{T}_{\nonm}$ (since $T_{\nonm}=T_{\nonp}$).

Similarly, we have
\begin{align*}
T_{\nonm}^{+} &= 1 + \W T_{\nonp}^{+} + S'(T) \, T^+ ; \\
T^+ &= 1 + \W T_{\nonp}^{+} + \W T_{\nonm}^{+} + S'(T) \, T^+.
\end{align*}

The above three equations form a system with three indeterminates: $T_{\nonp}^{+}$, $T_{\nonm}^{+}$ and $T^+$
($W$ and $T$ are known thanks to \cref{Prop:systeme1}).
Solving this system gives \cref{eq:tp,eq:tpm,eq:tpp}.

The symmetry argument giving $T_{\nonp}^{-}$, $T_{\nonm}^{-}$ and $T^-$
consist as before in exchanging $\ominus$ and $\oplus$ labels in $\mathcal S$-canonical trees.
\end{proof}

\subsection{Generating function counting trees with marked leaves inducing a given tree} \label{Subsec:DecompTrees}
To enumerate trees with marked leaves inducing a given tree, we introduce another kind of generating functions. 
Recall that for any permutations $\alpha$ and $\theta$, $\mathrm{occ}(\theta,\alpha)$ is the number of occurrences of $\theta$ in $\alpha$. 
For a permutation $\theta$, we set
\begin{equation}\label{eq:defOcc}
\Occ_\theta(z) = \sum_{\alpha\in \mathcal{S}} \mathrm{occ}(\theta,\alpha) z^{|\alpha|-|\theta|}.
\end{equation}

\begin{observation}\label{obs:RcvOcc>RS}
For $d\geq 1$ and any fixed $\alpha$, 
$\sum_{\theta \in \Sn_d}\mathrm{occ}(\theta,\alpha) = \binom {|\alpha|}{d}$.
Therefore $\sum_{\theta \in \Sn_d} \Occ_{\theta}$ is related to the $d$-th derivative of $S$ by %
$\sum_{\theta \in \Sn_d} \Occ_{\theta} = \tfrac {S^{(d)}} {d!}$.
This implies that the radius of convergence of each $\Occ_{\theta}$ is at least $R_S$, the radius of convergence of $S$.
\end{observation}
\bigskip

Fix a substitution tree $\patterntree$ with $k$ leaves. 
Let us call $\TTT_{\patterntree}$ the family of $\mathcal{S}$-canonical trees $t$ with $k$ marked leaves $\ll=(\ell_1,\dots,\ell_k)$ such that these leaves induce $\patterntree$:
$$
\TTT_{\patterntree} = \set{(t,\ll) \text{ such that } t \in \mathcal{T} \text{ and } t_{\ll}=\patterntree}.
$$
We define the size of an object $(t,\ll)$ as the number of leaves in $t$ ({\bf both marked and unmarked}).
The corresponding generating series is denoted $T_{\patterntree}(z)$
\medskip

Let $(t,\ll) \in \TTT_{\patterntree}$. 
As noted after the definition of induced trees, a nonlinear node of $t_{\ll}$ has to come from a nonlinear node of $t$,
whereas a linear node of $t_{\ll}$ may come from a linear or a nonlinear node of $t$.
In order to ease the enumeration, we partition $\TTT_{\patterntree}$ according to the set of nodes of $t_{\ll} = t_0$
coming from nonlinear nodes of $t$ (that is, simple nodes of $t$ since $t$ is canonical).

More formally, let $\Internal{t}$ 
be the set of internal nodes
of a tree $t$.
With each $(t,\ll)\in \TTT_{\patterntree}$, we associate the set $\FCA\subseteq\Internal{t}$ of the first common ancestors of $\ll$ in $t$. 
From the definition of induced tree, a node $v$ in $\patterntree$
corresponds to a unique node in $\FCA$, that we denote $\varphi(v)$.
For $V_s \subseteq \Internal{\patterntree}$, let
\[\TTT_{\patterntree,V_s} = \left\{ (t,\ll) \in \TTT_{\patterntree} : \{v\in \Internal{\patterntree} : \varphi(v) \text{ is simple}\} = V_s\right\}.\] 
Clearly $\TTT_{\patterntree,V_s}$ is nonempty if and only if $V_s$ contains every nonlinear node of $\patterntree$.
An example of a marked tree $(t,\ll)$ with the corresponding pair $(\patterntree, V_s)$ is shown on \cref{Fig:ExampleofT_tS}.
In pictures, we will always circle nodes $v$ in $V_s$ and the corresponding nodes $\varphi(v)$ in $t$.
\begin{figure}[thbp]
	\begin{center}
		\includegraphics[width=12cm]{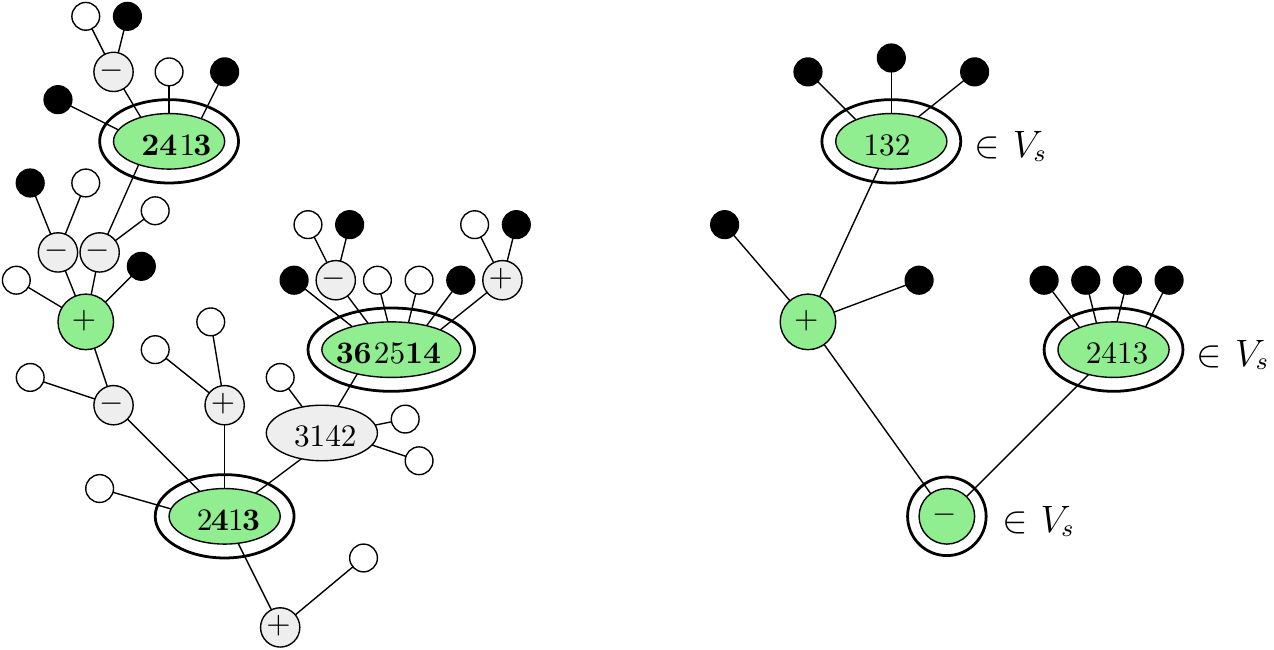}
		\caption{On the left: A canonical tree $t$ with $k=9$ marked leaves (in black).
        The set $\FCA$ of first common ancestors of the marked leaves is the set of green nodes.
        On the right: The corresponding decorated tree $(\patterntree,V_s)$.
        The tree $\patterntree$ is the tree induced by the marked leaves in $t$.
		\label{Fig:ExampleofT_tS}}
	\end{center}
\end{figure}
\begin{definition}
A \emph{decorated tree} is a pair $(\patterntree,V_s)$ where 
\begin{itemize}
\item $\patterntree$ is a substitution tree;
\item $V_s$ is a subset of $\Internal\patterntree$ that contains all nonlinear nodes.
\end{itemize}
\end{definition}

Therefore we have the following decomposition:
$$
\TTT_{\patterntree}=\bigcup_{V_s\text{ s.t. }(\patterntree,V_s)\text{ is a decorated tree} } \TTT_{\patterntree,V_s}.
$$

Let $(\patterntree,V_s)$ be a decorated tree.
We consider the generating function $T_{\patterntree,V_s}$ of $\TTT_{\patterntree,V_s}$,
the size of $(t,\ll)$ being its number of leaves (both marked and unmarked):
$$
T_{\patterntree,V_s}(z)=\sum_{(t,\ll) \in \TTT_{\patterntree,V_s}} z^{|t|}.
$$
To compute $T_{\patterntree,V_s}$, we introduce some notation. 
For every internal node $v$ of $\patterntree$, let
\begin{itemize}
	\item $\theta_v$ be the permutation labeling $v$,
	\item $d'_v$ be its number of children which are leaves or in $V_s$,
	\item $d^+_v$ be its number of children which are not in $V_s$ and are labeled by $\oplus$,
	\item $d^-_v$ be its number of children which are not in $V_s$ and are labeled by $\ominus$,
	\item $d_v=d'_v+d^+_v+d^-_v$ be its total number of children.
\end{itemize}
We also set the type of root to be $\prime$ if the root of $t_0$ is in $V_s$, and $+$ (resp. $-$) if the root is not in $V_s$ and labeled $\oplus$ (resp. $\ominus$).
\begin{proposition}[Enumeration of trees with marked leaves inducing a given decorated tree]
	\label{prop:Dec_TtS}
Let $(\patterntree,V_s)$ be a decorated tree and $k$ be its number of leaves. Then
\begin{equation}\label{eq:T_tS}
		T_{\patterntree,V_s}
		= z^k \ T^{\text{type of root}}
		\prod_{v\in \Internal{\patterntree}} A_v,
	\end{equation}
	where
	\begin{equation}A_v =  \begin{cases} \label{eq:A_v}
		\Occ_{\theta_v}(T) \ (T')^{d'_v} (T^+)^{d^+_v} (T^-)^{d^-_v}  & \text{ if } v\in V_s \, , \\		 
		\left(\tfrac 1 {1-T_{\nonp}}\right)^{d_v+1 }
		(T_{\nonp}')^{d'_v} (T_{\nonp}^+)^{d^+_v} (T_{\nonp}^-)^{d^-_v}
		  & \text{ if } v\notin  V_s \text{ and } \theta_v = \oplus \, , \\
		\left( 
		\tfrac 1 {1-T_{\nonm}}\right)^{d_v+1 }
		(T_{\nonm}')^{d'_v} (T_{\nonm}^+)^{d^+_v} (T_{\nonm}^-)^{d^-_v}
		  & \text{ if } v\notin  V_s \text{ and } \theta_v = \ominus \, .\\
	\end{cases}  \end{equation}
\end{proposition}

\begin{proof}%
The proof is based on a decomposition of marked $\mathcal{S}$-canonical trees $(t,\ll)$ of $\TTT_{\patterntree,V_s}$ followed by a study of the series $A_v$ depending on the type of the node $\varphi(v)$ in $t$. 
\medskip 

\noindent{\bf First step: Decomposing a tree in $\TTT_{\patterntree,V_s}$.}\\
We fix a decorated tree $(\patterntree,V_s)$ with $k$ leaves and a marked $\mathcal{S}$-canonical tree $(t,\ll) \in \TTT_{\patterntree,V_s}$. 
We want to decompose $t$ into subtrees, one for each internal node of $\patterntree$ plus one attached to the root of $t$.
Recall that $\varphi:\Internal{\patterntree}\to \FCA\subseteq\Internal{t}$ is the correspondence between the internal nodes of $\patterntree$ and the set of first common ancestors of leaves $\ll$ in $t$. 

For every internal node $v$ of $\patterntree$, let $t_v$ be the subtree of $t$ defined as follows.

\begin{itemize}
\item The root of $t_v$ is $\varphi(v)$.
\item The nodes of $t_v$ are descendants of $\varphi(v)$.
\item A descendant of $\varphi(v)$ in $t$ belongs to $t_v$ if and only if its first proper ancestor in $\FCA$ is $\varphi(v)$ (proper meaning different from the node itself).
\end{itemize}
Moreover we define $t_B$ as the subtree of $t$ rooted at the root of $t$
and containing the nodes of $t$ having no proper ancestor in $\FCA$;
$B$ stands for ``bottom'' and is used here as a symbol, not as a variable.
(If $\varphi$ maps the root of $t_0$ to the root of $t$, then $t_B$ is reduced to a leaf.) 

A schematic representation of the trees $t_v$ and $t_B$
is given in \cref{Fig:decomp_of_T_tS}.

\begin{figure}[thbp]
	\begin{center}
		\includegraphics[width=12cm]{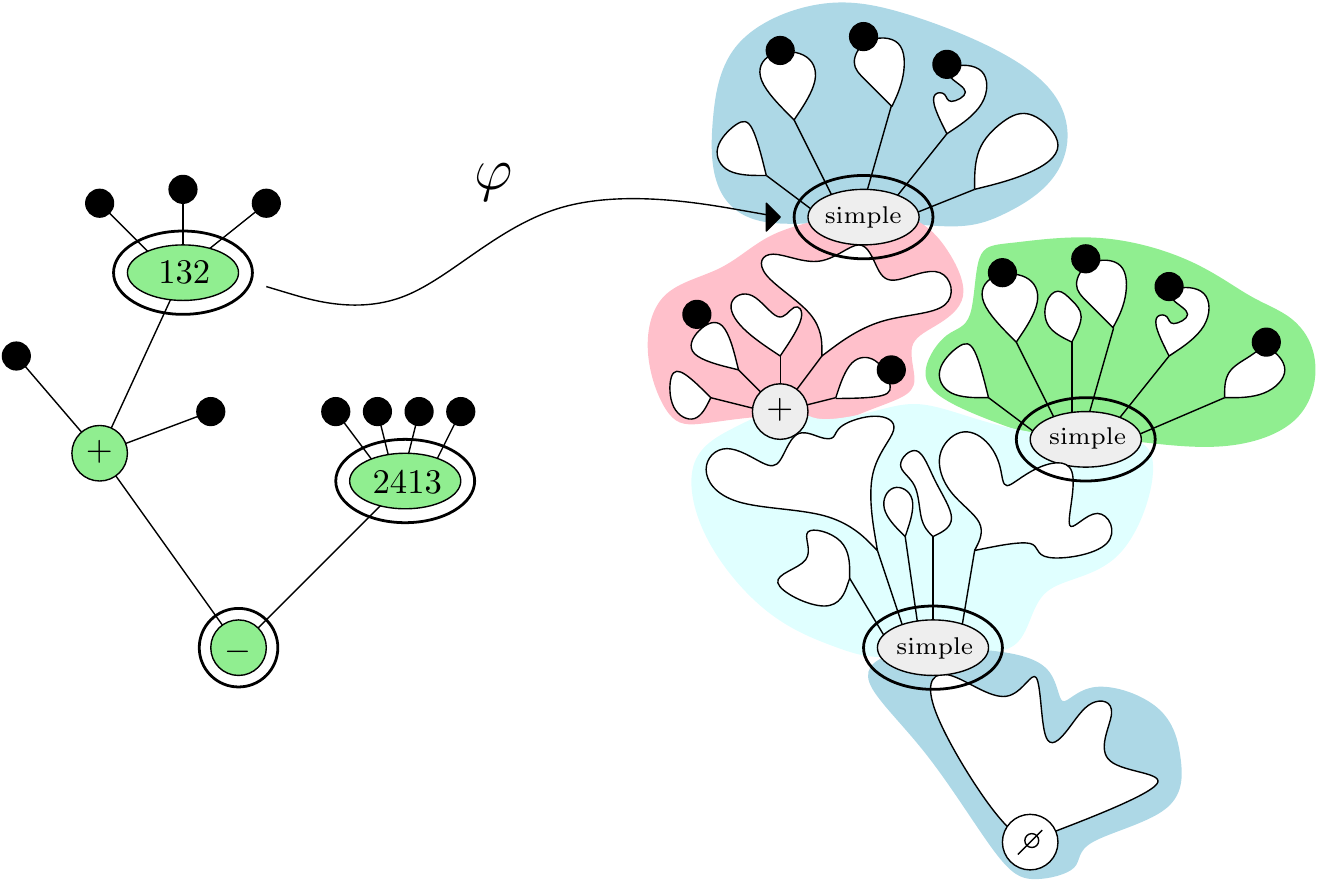}
		\caption{On the left: A decorated tree $(\patterntree,V_s)$ with $9$ leaves and $4$ internal nodes.
		On the right: A schematic representation of a tree $t \in \TTT_{\patterntree,V_s}$.
		In this case, $t$ is decomposed into $4$ subtrees $t_v$ and a subtree $t_B$.
		Note that a linear node in $\patterntree$ corresponds 
        to either a simple or a linear node in $t$, depending on whether or not it belongs to $V_s$.
		\label{Fig:decomp_of_T_tS}}
	\end{center}
\end{figure}

By definition, a node $u$ of $t$ that is not in $\FCA$ belongs to exactly one $t_v$.
On the contrary, if $u$ is in $\FCA$, then $u$ is the root of $t_{\varphi^{-1}(u)}$ and is a leaf of another $t_v$,
where $v$ is the parent of $\varphi^{-1}(u)$.
(If $\varphi^{-1}(u)$ is the root of $\patterntree$, then there is no such $v$, and $u$ is a leaf of $t_B$.)

By construction of $t_v$ and $t_B$, their leaves are either leaves of $t$ or belong to $\FCA$.
We mark the leaves that belong to $\FCA$ or that are marked leaves of $t$.
In this way, the trees $t_v$ and $t_B$ that we have constructed are {\em marked trees}.

The following properties are straightforward to check.
\begin{enumerate}
	\item The tree $t_v$ is an $\mathcal{S}$-canonical tree with $d_v$ marked leaves.
	\item The root of $t_v$ is nonlinear  if and only if $v\in V_s$. 
	\item The root of $t_v$ is $\oplus$ if and only if $v\notin V_s$ and is labeled $\oplus$.
	\item The root of $t_v$ is $\ominus$ if and only if $v\notin V_s$ and is labeled $\ominus$. 
	\item The $d_v$ marked leaves of $t_v$ belong to $d_v$ subtrees coming from $d_v$ \emph{distinct} children of the root of $t_v$.
The pattern induced by the position of those $d_v$ children on the permutation labeling the root of $t_v$ is $\theta_v$. 
(For example, in \cref{Fig:ExampleofT_tS}, four marked leaves are branched on the node labeled $362514$ at positions $1,2,5,6$. This implies that the corresponding node in $t_0$ is labeled with $\theta_v=2413$.)
\item Let $w$ be the $i$-th child of $v$ in $\patterntree$. If $w\in \Internal{\patterntree} \setminus V_s$, and its label is a $\oplus$ (resp. a $\ominus$), then the $i$-th marked leaf of $t_v$ must be $\oplus$-replaceable (resp. $\ominus$-replaceable).
\end{enumerate}
The combinatorial class of trees satisfying properties i) to vi) will be denoted $\mathcal A_v$.

In addition, we observe that $t_B$ is an $\mathcal{S}$-canonical tree with one marked leaf;
moreover, if the root of $\patterntree$ is not in $V_s$ and is labeled by $\oplus$ (resp. $\ominus$),
then the marked leaf of $t_B$ must be $\oplus$-replaceable (resp. $\ominus$-replaceable).

This yields a map $$
\begin{array}{r r c l}
\TT : & \TTT_{\patterntree,V_s}        & \to & \TTT^{\,\text{type of root}} \times \prod_{v\in \Internal{\patterntree}} \mathcal A_v \\
& (t,\ll) & \mapsto & (t_B,(t_v)_{v\in \Internal{\patterntree}}).
\end{array}
$$

We claim that this map is a bijection, and that the inverse map is obtained as follows. 
Let us be given $t_B$ and a collection of trees $t_v$, one for each internal node of $\patterntree$.
We first take $t_B$ and glue $t_{\text{root of }\patterntree}$ on it,
the root of $t_{\text{root of }\patterntree}$ replacing the marked leaf of $t_B$.
We then proceed inductively: if $t_v$ has already been glued and $w$ is the $i$-th child of $v$,
we glue $t_w$ on $t_v$, by replacing the $i$-th marked leaf of $t_v$ with the root of $t_w$. 
This yields a tree $t$ with $k$ marked leaves, denoted $\ll$.
This tree is $\mathcal{S}$-canonical because of items i), iii), iv) and vi) of the definition of $\mathcal A_v$:
we only glue trees with root $\oplus$ (resp. $\ominus$) on $\oplus$-replaceable leaves (resp. $\ominus$-replaceable leaves).
By construction, the $k$ marked leaves $\ll$ induce a tree having the same structure as $t_0$
and item v) of the definition of $\mathcal A_v$ ensures that the labels in the induced tree and in $t_0$ do match.
Because of item ii), the tree $(t,\ll)$ is indeed in $\TTT_{t_0,V_s}$.
We have therefore constructed a map from $ \TTT^{\,\text{type of root}} \times \prod_{v\in \Internal{\patterntree}} \mathcal A_v$
to $\TTT_{t_0,V_s}$. By construction, this map indeed inverts $\TT$, and $\TT$ is a bijection.

Let $A_v$ be the generating function of the combinatorial class $\mathcal A_v$, counted by the number of \textbf{unmarked leaves}.
If $A_v$ verifies \eqref{eq:A_v}, then \eqref{eq:T_tS} follows from the fact that $\TT$ is a bijection.
Note indeed that the factor $z^k$ in \eqref{eq:T_tS} comes from the fact that we count marked leaves in the series in the left-hand side
and but not in the series in the right hand side (the bijection $\TT$ leaves the number of unmarked leaves invariant).

We are left to show that the generating function $A_v$ verifies \eqref{eq:A_v}.
\medskip

\noindent{\bf Second step (i): Computing $A_v$ when $v\in V_s$.}\\
Recall that the $d_v$ marked leaves of any tree $t \in \mathcal A_v$ belong to $d_v$
subtrees coming from $d_v$ \emph{distinct} children of the root of $t$.
Since $v \in V_s$, the elements $t$ of $\mathcal A_v$ can be uniquely decomposed as follows 
(this decomposition is illustrated on \cref{fig:DecompoSommet_Racine231dansS}).
\begin{figure}[thbp]
\begin{center}
\includegraphics[width=12cm]{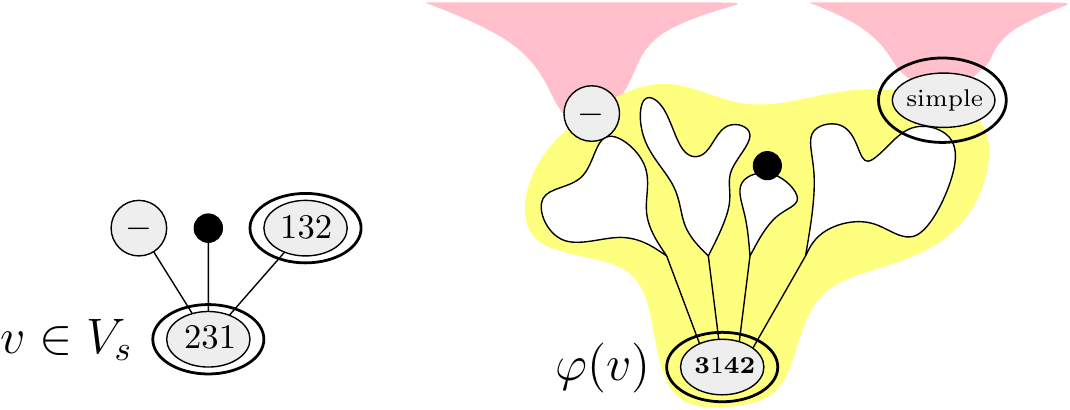}
\caption{Computing $A_v$ when $v\in V_s$. On the left: We only represent the node $v$ in $t_0$ and its children;
the other parts of $t_0$ being irrelevant to compute $A_v$.
On the right: We zoom in on the part corresponding to $t_v$ in the schematic representation of an element
in $\mathcal T_{(\patterntree,V_s)}$.
Here the series $A_v$ is \hbox{$A_v=\Occ_{231}(T)\, (T')^{2}\, (T^-)$.}}
\label{fig:DecompoSommet_Racine231dansS}
\end{center}
\end{figure}
\begin{enumerate}
	\item The root of $t$ should be labeled by a simple permutation $\alpha$ in $\mathcal S$;
      among the $|\alpha|$ children of the root, $d_v$ are marked (corresponding to the subtrees containing a marked leaf)
      and the pattern of $\alpha$ corresponding to the positions of these marked children should be $\theta_v$.
      (In \cref{fig:DecompoSommet_Racine231dansS}, $\alpha=3142$, the marked leaves are the first, third and fourth subtrees,
      and the pattern of $3142$ corresponding to positions $\{1,3,4\}$ is indeed $\theta_v=231$)
	\item We glue $|\alpha|-d_v$
      unmarked $\mathcal{S}$-canonical trees with arbitrary roots
      on the unmarked children of $\alpha$.
      (In \cref{fig:DecompoSommet_Racine231dansS}, we have only one such tree,
      which is glued on the second child of the root);
	\item We glue $d_v$ $\mathcal{S}$-canonical trees with one leaf marked and an arbitrary root on the marked children of $\alpha$.
      In addition, 
	\begin{itemize}
		\item for $d'_v$ of these trees, there is no constraint on the marked leaf.
          (In \cref{fig:DecompoSommet_Racine231dansS}, the trees glued on the third and fourth children of the root
          are unconstrained trees with a marked leaf.)
		\item For $d^+_v$ (resp. $d^-_v$) of these trees, the marked leaf must be $\oplus$-replaceable (resp. $\ominus$-replaceable).
          (In \cref{fig:DecompoSommet_Racine231dansS}, we must glue a tree with a $\ominus$-replaceable marked leaf
          on the first child of the root.)
	\end{itemize}	
\end{enumerate}
The generating functions of the first two steps can be computed as
\[\sum_{\alpha \in \mathcal S} \mathrm{occ}(\theta_v,\alpha) T(z)^{|\alpha|-d_v}=\Occ_{\theta_v}(T(z)),\]
where $\Occ_{\theta_v}$ is defined in \cref{eq:defOcc} p.~\pageref{eq:defOcc}.
Indeed, once the label $\alpha$ of the root is chosen, $\mathrm{occ}(\theta_v,\alpha)$ counts the number
of ways to mark children of the root in step i), and $T(z)^{|\alpha|-d_v}$ comes from step ii).
Step iii) yields an additional factor $(T')^{d'_v} (T^+)^{d^+_v} (T^-)^{d^-_v}$.
This proves the formula \eqref{eq:A_v} in the case where $v$ is in $V_s$.

\medskip

\noindent{\bf Second step (ii): Computing $A_v$ when $v\notin V_s$.}\\
When $v$ is not in $V_s$ and labeled by $\oplus$,
the elements of the class $\mathcal A_v$ can be uniquely decomposed as follows
(this decomposition is illustrated on \cref{fig:DecompoSommet_Racine+PasdansS}).

\begin{figure}[thbp]
\begin{center}
\includegraphics[width=12cm]{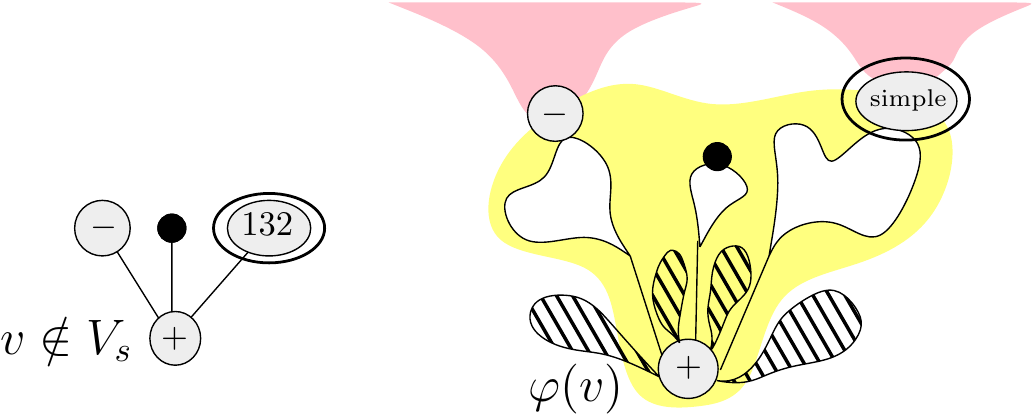}
\caption{Computing $A_v$ when $v$ is not in $V_s$.
On the left: We again represent only the node $v$ in $t_0$ and its children;
the other parts of $t_0$ being irrelevant to compute $A_v$.
On the right: We zoom in on the part corresponding to $t_v$ in the schematic representation of an element
in $\mathcal T_{(\patterntree,V_s)}$.
Here the series $A_v$ is \hbox{$A_v=\left(\tfrac{1}{1-T_{\nonp}}\right)^{4}\,
 (T_{\nonp}')^{2} \, (T_{\nonp}^-)$}.
}
\label{fig:DecompoSommet_Racine+PasdansS}
\end{center}
\end{figure}

\begin{enumerate}
  \item The root is labeled by $\oplus$.
	\item We attach to the root $d_v$ $\mathcal{S}$-canonical trees whose root is not labeled by $\oplus$,
      each with one marked leaf. In addition, 
	\begin{itemize}
		\item for $d'_v$ of these trees, there is no constraint on the marked leaf.
          (In \cref{fig:DecompoSommet_Racine+PasdansS}, the two right-most nonhatched trees attached to the root
          are trees with an unconstraint marked leaf.)
		\item for $d^+_v$ (resp. $d^-_v$) of these trees, the marked leaf must be $\oplus$-replaceable (resp. $\ominus$-replaceable).
          (In \cref{fig:DecompoSommet_Racine+PasdansS}, the left-most nonhatched tree attached to the root
          should have a $\ominus$ replaceable marked leaf.)
	\end{itemize}	
	\item Between and around these $d_v$ trees,
      we attach $d_v + 1$ possibly empty sequences of unmarked $\mathcal{S}$-canonical trees whose root is not labeled by $\oplus$.
       (In \cref{fig:DecompoSommet_Racine+PasdansS}, each of these sequences is represented by a hatched blob.)
\end{enumerate}

Item i) does not involve any choice.
Choices in item ii) are counted by $(T_{\nonp}')^{d'_v} (T^+_{\nonp})^{d^+_v} (T^-_{\nonp})^{d^-_v}$,
while item iii) yields a factor $\left(\tfrac 1 {1-T_{\nonp}}\right)^{d_v+1 }$.
This proves the formula \eqref{eq:A_v} in the case where $v$ is not in $V_s$ and labeled by $\oplus$.
\smallskip

The case when $v$ is not in $V_s$ and is labeled by $\ominus$ follows by symmetry.
This ends the proof of the combinatorial identity \eqref{eq:A_v} and therefore of \cref{prop:Dec_TtS}.
\end{proof}
\section{Asymptotic analysis: The standard case $S'(R_S)>2/(1+R_S)^2-1$}

\label{Sec:Standard}

Let $\mathcal{S}$ be a set of simple permutations. 
The goal of this section is to precisely state, and then prove, \cref{Th:MainIntro} (p.\pageref{Th:MainIntro}): 
the convergence to the biased Brownian separable permuton of uniform random permutations in $\langle \mathcal{S}\rangle$
when $\mathcal{S}$ satisfies Condition \eqref{eq:h1intro}. 

\subsection{Definition of the biased Brownian separable permuton and statement of the theorem}
\label{ssec:BrownianPermuton}
The (unbiased) Brownian separable permuton was defined in \cite{BrownianPermutation}. 
Because the biased Brownian separable permuton is a one-parameter deformation of it, 
it is useful to first recall some facts about the substitution trees encoding separable permutations and the (unbiased) Brownian separable permuton. 

As noted in \cref{sec:OperatorsPermutations}, the canonical trees (called decomposition trees in \cite{BrownianPermutation}) 
of separable permutations are those whose internal nodes are all labeled $\oplus$ or $\ominus$. 
If we consider more generally substitution trees, the following implication still holds:
if $\tau$ is a substitution tree whose nodes are labeled $\oplus$ or $\ominus$,
then $\perm(\tau)$ is a separable permutation.

 Recall from \cref{sec:SubsTrees} that 
 an \textit{expanded tree} is a substitution tree where nonlinear nodes are labeled by simple permutations,
while linear nodes are required to be binary.
In the case of separable permutations, we do not have simple nodes,
so that {\em expanded tree} are binary trees labeled with $\oplus$ and $\ominus$.
These are also referred to as {\em separation trees} in the literature.
\cref{fig:ExempleSeparationTree} shows a separable permutations together with two separation trees associated with it.

\begin{figure}[htbp]
    \begin{center}
      \includegraphics[width=12cm]{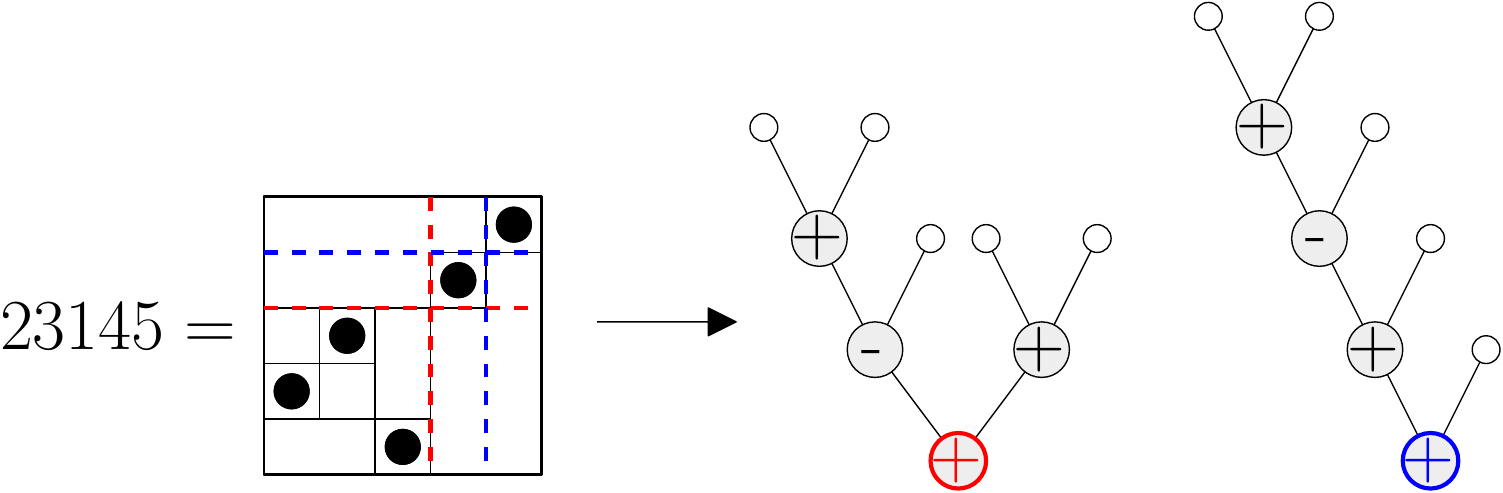}
    \end{center}
    \caption{Two examples of separation trees associated with the separable permutation $\pi=23145$. In this case $r_+(\pi)=3$, $r_-(\pi)=1$. }
    \label{fig:ExempleSeparationTree}
\end{figure}

For any separable permutation $\pi$, we denote by $N_\pi$ its number of separation trees. 
If $\pi$ is not separable, we set $N_\pi = 0$. 
It is shown in \cite[Prop. 9.1]{BrownianPermutation} that the Brownian separable permuton $\bm\mu$ satisfies the following property:
for any $k\geq 2$ and any $\pi \in \Sn_k$,
	\[\esper[\occ(\pi,\bm\mu)] = \frac{N_\pi}{2^{k-1}\Cat_{k-1}}\,,\]
where, as before, we denote by $\Cat_{k} := \frac{1}{k+1}\binom{2k}{k}$ the $k$-th Catalan number,
which counts complete binary trees with $k$ leaves.
In other words, the random permutation of size $k$ extracted from $\bm\mu$
is distributed like the permutation encoded by a uniform complete binary tree with $k$ leaves, 
and hence $k-1$ internal nodes, 
whose signs are chosen uniformly and independently in $\{\oplus,\ominus\}$.
In light of \cref{Prop:CaracterisationLoiPermuton} and \cref{eq:E(occ)=P(perm)} (p.~\pageref{eq:E(occ)=P(perm)}), 
this characterizes the law of the Brownian separable permuton $\bm\mu$ among random permutons. 

The biased Brownian separable permuton of parameter $p\in(0,1)$ has a similar characterization, except that the signs are now chosen with a bias. 
For a separable $\pi$, let $r_+(\pi)$ (resp. $r_-(\pi)$) be the number of internal nodes labeled $\oplus$ (resp. $\ominus$) in a separation tree of $\pi$. 
Even if this is not relevant for the present paper,
let us observe that $r_+(\pi)$ (resp. $r_-(\pi)$) is simply the number of ascents (resp. descents) of $\pi$\footnote{
To see this, observe that each internal node $v$ of a separation tree is the first common ancestor of exactly one pair of consecutive leaves
(the right-most leaf of its left subtree and the left-most leaf of its right subtree).
This two consecutive leaves, corresponding to consecutive elements of the permutation, form an ascent (resp. a descent)
if and only if $v$ is labeled by $\oplus$ (resp. $\ominus$).}.
In particular, $r_+(\pi)$ and $r_-(\pi)$ do not depend on the choice of a separation tree
(this is also a particular case of \cref{cor:OnExpandedTrees}).

\begin{definition}
	The biased Brownian separable permuton of parameter $p\in (0,1)$ is the random permuton $\bm\mu^{(p)}$ characterized by the following relations: for all $k\geq 2$ and all $\pi \in \Sn_k$,
    \begin{equation}
      \esper[\occ(\pi,\bm\mu^{(p)})] = \frac{N_\pi}{\Cat_{k-1}} \ p^{r_+(\pi)} \, (1-p)^{r_-(\pi)} \,. 
      \label{eq:def_PermutonBiaise}
    \end{equation}
	\label{Def:PermutonBiaise}
(Note that the right-hand side is zero if $\pi$ is not separable.)
\end{definition}

Several remarks are in order. 
\begin{itemize}
  \item For $p=1/2$, we get the unbiased Brownian separable permuton.
  \item This characterization of $\bm \mu^{(p)}$ is equivalent to the following: for every $k\geq 1$, 
  \[\Perm(\Mk,\bm \mu^{(p)}) \stackrel d= \perm(\bm b_k^{(p)}),\]
  where $\bm b_k^{(p)}$ is a uniform binary planar tree with $k$ leaves, where each internal node is labeled $\oplus$ (resp. $\ominus$) with probability $p$ (resp. $1-p$), independently from each other.
  \item The existence of $\bm\mu^{(p)}$ is not immediate from this definition, but according to \cref{Prop:existence_permuton}, it suffices to show that $\perm(\bm b_k^{(p)})$ forms a consistent family of random permutations. This is indeed the case, and follows from the fact that a uniform induced subtree of $\bm b_n^{(p)}$ of size $k$ is distributed like $\bm b_k^{(p)}$ (this is,  \emph{e.g.}, a consequence of Rémy's algorithm to generate uniform random binary trees \cite{Remy}).
  \item The definition of $\bm\mu^{(p)}$, and the above argument justifying its existence,
    are not constructive.
    For an explicit construction of $\bm\mu^{(p)}$ starting from a Brownian excursion, see \cite{MickaelConstruction}.
  \item Knowing a priori that such a permuton exists is not necessary for the proof
    of our main theorem. Indeed, we will prove that
    the quantity $\esper[\occ(\pi,\bm \si_n)]$ converges to the right-hand side of \cref{eq:def_PermutonBiaise} (for all patterns $\pi$),
    where $\bm \si_n$ is a uniform permutation in $\langle \mathcal{S}\rangle_n$, and the parameter $p$ depends on $S$.
    From \cref{thm:randompermutonthm}, this implies the existence of a random permuton $\bm\mu^{(p)}$
    satisfying \eqref{eq:def_PermutonBiaise} (only for the relevant value of $p$; not for all $p$) 
    and the convergence in distribution of $(\mu_{\bm \si_n})_n$ to $\bm\mu^{(p)}$.
\end{itemize}

Now, we have all the necessary definitions to make explicit the parameter $p$ of the statement of \cref{Th:MainIntro}, that we restate in a full version. 
(Recall also the definition of $\Occ_\theta(z)$ from \cref{eq:defOcc}, p.~\pageref{eq:defOcc}.)
\begin{theorem}\label{Th:Main}
Let $\mathcal{S}$ be a set of simple permutations such that
\begin{equation}\label{eq:h1} 
  R_S > 0 \quad \text{and} \quad \lim_{r\rightarrow R_S \atop r < R_S} S'(r) > \frac{2}{(1+R_S)^2} -1.
  \tag{H1}
\end{equation}
For every $n\geq 1$, let $\Si_n$ be a uniform permutation in $\langle \mathcal{S}\rangle_n$, and let $\mu_{\Si_n}$ be the random permuton
associated with $\Si_n$.
The sequence $(\mu_{\Si_n})_n$ tends in distribution in the weak
convergence topology to the biased Brownian separable permuton $\bm\mu^{(p)}$ of parameter $p$, where
\begin{align}
	p &= \frac{ (1+\kappa)^3\Occ_{12} (\kappa)  +1}
	{(1+\kappa)^3(\Occ_{12} (\kappa)+\Occ_{21} (\kappa)) +2}%
\label{eq:p_plus_p_moins}
\end{align}
and $\kappa$ is the unique point in $(0,R_S)$ such that $S'(\kappa) = \tfrac 2 {(1+\kappa)^2} - 1$.
\end{theorem}
Since $S$ is a power series with nonnegative coefficients,
 $t\mapsto S'(t) - \tfrac 2 {(1+t)^2} + 1$ is increasing and continuous on $[0,R_S)$ (as the sum of two increasing and continuous functions).
 It therefore takes all values from $-1$ to some positive number (possibly $+\infty$) exactly once.
This entails the existence and uniqueness of $\kappa$.

\begin{example}\label{ex:calculp1}
In many cases $\Occ_{12} = \Occ_{21}$, and then $p = 1/2$ and $\bm\mu^{(p)}$ is the unbiased Brownian separable permuton.
This is the case with separable permutations ($\SSS = \emptyset$), with $\SSS = \{2413\}$ or $\SSS = \{3142\}$,
and with any set of simple permutations stable by taking reverse or complement,
like the one considered in the introduction
$\SSS=\{2413,3142,24153,42513\}$.
\end{example}
\begin{example}\label{ex:calculp2}
When $\SSS$ is the family of increasing oscillations (see for instance~\cite{PinPerm}), we can compute
\[S(z) = \frac{2z^4}{1-z};\quad \Occ_{12}(z) = \frac{2z^2(3 - 3z + z^2)}{(1 - z)^3}; \quad \Occ_{21}(z)=\frac{2z^2(3-2z)}{(1 - z)^2}.
\] 
We get through numerical approximation $\kappa \approx 0.2709$ and deduce $p\approx0.5353$.
\end{example}
\begin{example}\label{ex:calculp3}
  Taking $\SSS$ to be the family of simple permutations in $\Av(321)$, we are interested in the class $\CCC=\langle \SSS \rangle$ which is the substitution-closure of $\Av(321)$. In this case, \cite{AlbertVatter} gives 
\[S(z) = \frac{1-z-2z^2 -2z^3 - \sqrt{1-2z-3z^2}}{2+2z}.
\] 
We get through numerical approximation $\kappa \approx 0.2486$. It seems hard to compute the generating series $\Occ_{12}$, but we can locate its value at $\kappa$ by exhaustively computing the number of inversions of each permutation in $\SSS$ up to a certain order $N$, and controlling the rest of the series using the fact that 
a permutation of size $n$ in $\Av(321)$ cannot have more than $n^2/4$ inversions\footnote{Permutations avoiding $321$ consist of two increasing subsequences. 
The number of inversions of $\sigma\in \Av(321)$ of size $n$ is therefore at most $\max_{0\leq k \leq n} k(n-k) \leq \tfrac{n^2}{4}$. 
The claim follows.}.
Performing this with $N=12$ yields $p\in[0.577, 0.622]$.
\end{example}

The remainder of Section \ref{Sec:Standard} is devoted to the proof of \cref{Th:Main},
using generating functions from \cref{Sec:Enumeration} and methods of analytic combinatorics.
More precisely, using \cref{thm:randompermutonthm} we are interested in the limit of $\esper[\occ(\pi,\Si_n)]$,
which we will express in terms of probability that a tree with marked leaves induces a given subtree.
This probability itself will be expressed as the ratio of the coefficients of
the generating function of trees with marked leaves inducing a given subtree
and of the generating function of trees without marked leaves.
We begin with the study of the asymptotics of these generating functions.
\smallskip

Notation: throughout the article, the class $\CCC$, or equivalently its set of simple permutations $\SSS$,
are considered as fixed, and so is the pattern $\pi$ or the tree $t_0$ of which we are studying the proportion
of occurrences (and therefore their size $k$).
Constants in asymptotic expansions, including the ones in $o$, $\mathcal O$ and $\Theta$ symbols,
may therefore depend on these objects.

\subsection{Asymptotics of the generating function of trees with no or one marked leaf}
\label{ssec:Asymp_TreesGF_Std}

From \cref{eq:Tnonp} p.\pageref{eq:Tnonp}, we have
\begin{equation}\label{eq:Tnonp_Lambda}
T_{\nonp}=z+ \Lambda(T_{\nonp})
\end{equation}
where
\begin{equation}\label{eq:DefLambda}
\Lambda(u)=\frac{u^2}{1-u}+ S\left(\frac{u}{1-u}\right).
\end{equation}

We denote by $R_\Lambda$ the radius of convergence of $\Lambda$. Note that $R_\Lambda  = \tfrac {R_S} {1+R_S} \leq 1$.
We will also use repeatedly the inverse equation $R_S=\tfrac {R_\Lambda}{1-R_\Lambda}$.
In the following, to lighten the notation, 
we write $\Lambda'(R_\Lambda) := \lim_{r\rightarrow R_\Lambda \atop r < R_\Lambda} \Lambda'(r)$. Note that $\Lambda'(R_\Lambda)$ may be $\infty$. 

\begin{observation}
\label{obs:conditionH1}
Differentiating \cref{eq:DefLambda}, we get
\begin{equation}
  \Lambda'(u)= \frac{1}{(1-u)^2} \left(1 + S'\Big( \frac{u}{1-u}\Big) \right)-1.
  \label{eq:LambdaPrime}
\end{equation}
In particular, it follows that the condition \eqref{eq:h1} is equivalent to $R_\Lambda >0$ and $\Lambda'(R_\Lambda)>1$.
\end{observation}

\begin{observation}
\label{obs:Lambda_at_0}
Since $S$ is analytic at $0$ with nonnegative coefficients, the same holds for $\Lambda$. 
Moreover, the series expansion of $\Lambda$ is 
\[
\Lambda(u) = u^2 + \sum_{i\geq 3} \lambda_i u^i, \text{ with } \lambda_i \geq 1 \text{ for all }i \geq 3.
\]
In particular it is aperiodic, in the sense given in \cref{ssec:aperiodicity}.
\end{observation}

\begin{proposition}[Asymptotics of the generating function of $\mathcal S$-canonical trees with no marked leaf]\label{lem:MMCversionSerie}    	\label{Prop:asymp_T}
	Assume that \eqref{eq:h1} holds, and recall that $\kappa$ is defined by $S'(\kappa) = \tfrac 2 {(1+\kappa)^2} - 1$. 
	There is a unique $\tau\in (0,R_\Lambda)$ such that $\Lambda'(\tau)=1$, and we have $\tau = \frac {\kappa}{1+\kappa}$.
	The generating functions $T$ and $T_{\nonp}$ have the same radius of convergence $\rho=\tau-\Lambda(\tau)\in (0,\tau)$ and have a unique dominant singularity\footnote{For the reader who is not familiar with complex analysis,
		    all useful definitions and results are given in \cref{sec:complex_analysis}. In particular, "near $\rho$" means "in a $\Delta$-neighborhood of $\rho$",
		    where "$\Delta$-neighborhood" is defined in \cref{Def:DeltaDomaine}.
		    The formal definition of \emph{(unique) dominant singularity} is given in \cref{eq:def_Exp} p.\pageref{eq:def_Exp}.
	} in $\rho$.
	Their asymptotic expansions near $\rho$ are:
	\begin{align}
	T(z) &= \frac{\tau}{1 - \tau} - \beta\, \lambda\,  \sqrt{1-\tfrac{z}{\rho}} + \mathcal{O}(1-\tfrac{z}{\rho}),\label{eq:Exp_T}\\
	T_{\nonp}(z)&= \tau - \beta \sqrt{1-\tfrac{z}{\rho}} + \mathcal{O}(1-\tfrac{z}{\rho}).\label{eq:DevT}
	\end{align}
	where $\displaystyle{\beta = \sqrt{\frac{2 \rho}{\Lambda''(\tau)}}}$ and $\lambda= \displaystyle{\frac{1}{(1-\tau)^2}}$. 
	In particular, $T$ and $T_{\nonp}$ are convergent at $z=\rho$ and
	$$
	T(\rho)= \frac{\tau}{1 - \tau},\qquad T_{\nonp}(\rho)=\tau.
	$$
\end{proposition}

This type of behavior with a square-root dominant singularity is classical for series defined by an implicit equation
(such as $T_{\nonp}$, which is characterized by $T_\nonp(z) = z + \Lambda(T_\nonp(z))$), that belong to the \textit{smooth implicit function schema} \cite[Def. VII.4]{Violet}. This schema is defined by the existence of a solution to some \textit{characteristic equation}, which in our case reduces to hypothesis \eqref{eq:h1}, as explained in \cref{Subsec:outline_proof}.
Our result is a special case of \cite[Th. 1]{MathildeMarniCyril}, a general result for equations of the form $U(z) = z +\Lambda(U(z))$. Implicit equations of this form characterize generating functions of weighted trees counted by their number of leaves, and are also considered in \cite[Prop. 8]{PitmanRizzolo}.

\begin{proof}%
	From \cref{obs:Lambda_at_0}, $\Lambda'$ is strictly increasing in the real interval $(0,R_\Lambda)$.
	Together with the fact that $\Lambda'(0)=0$
    and the assumption $\Lambda'(R_\Lambda)>1$ (see \cref{obs:conditionH1}),
    this proves the existence and uniqueness of $\tau >0$ such that $\Lambda'(\tau)=1$.
	
    Setting $v=\tfrac{u}{1-u}$ in $\Lambda'$, which is given by \eqref{eq:LambdaPrime},
    we have $
	\Lambda'\left(\tfrac v {1+v}\right)
	= (1+v)^2(1+S'(v)) - 1
	$. It follows that $\Lambda'(\frac {\kappa} {1+\kappa}) = 1$.
    By uniqueness of $\tau$, we conclude $\tau = \frac {\kappa} {1+\kappa}$.
	
	We now consider the expansion of $T_{\nonp}$ and deduce afterwards the one of $T$.
	From \cref{eq:Tnonp_Lambda}, we have $T_{\nonp}(z)=z+\Lambda(T_{\nonp}(z))$.
	Then Theorem 1 in \cite{MathildeMarniCyril} gives that $T_\nonp$ is analytic at $0$ and
	has a unique dominant singularity of exponent $\tfrac 1 2$ in $\rho=\tau-\Lambda(\tau)$, with the expansion given in \cref{eq:DevT}: 
	$T_{\nonp}(z)=\tau - \beta \sqrt{1-\tfrac{z}{\rho}} + \mathcal{O}(1-\tfrac{z}{\rho})$.
	The next step is to justify that $\rho \in (0,\tau)$.
    That $\rho <\tau$ follows from $\Lambda(\tau)>0$ (since $\tau >0$).
    Moreover, since $\Lambda$ has nonnegative coefficients and no constant term, we have
    $\Lambda(\tau)<\tau \Lambda'(\tau)=\tau$, so that $\rho>0$.

    Finally, we look at the series $T=\tfrac{T_{\nonp}}{1-T_{\nonp}}$ (see \cref{eq:T_Tnp}).
    Observe that
    $T_{\nonp}(\rho)=\tau < R_\Lambda \leq 1$.
	Consequently, the dominant singularity of $T$ is the same as $T_{\nonp}$, \emph{i.e.} $\rho$, 
    and is still unique -- indeed this singularity is reached before that the denominator vanishes; more formally,
    this is a particular case of subcritical composition (\cref{lem:Comp}).
	The asymptotic expansion of $T$ near $\rho$ is obtained through the following computation:
    \[
	T(z)=\frac{\tau - \beta \sqrt{1-\tfrac{z}{\rho}} + \mathcal{O}(1-\tfrac{z}{\rho})}{1-\tau + \beta \sqrt{1-\tfrac{z}{\rho}} + \mathcal{O}(1-\tfrac{z}{\rho})}
        = \frac{\tau}{1 - \tau} - \frac{\beta}{(1 - \tau)^2} \sqrt{1-\tfrac{z}{\rho}} + \mathcal{O}(1-\tfrac{z}{\rho}). \qedhere
        \]
\end{proof}

\begin{proposition}[Asymptotics of the  generating function of $\mathcal S$-canonical trees with marked leaves]\label{prop:DevMarkedleaves}
	All generating functions $T'$, $T'_{\nonp}$, $T'_{\nonm}$, $T^+$, $T^+_{\nonp}$, $T^+_{\nonm}$, $T^-$, $T^-_{\nonp}$ and $T^-_{\nonm}$
	have a unique dominant singularity in $\rho$. They diverge at the singularity $z=\rho$
	and behave as $K (1-\tfrac{z}{\rho})^{-1/2}$, 
	where the constant $K$ is given in the table below:
	\[\begin{array}{| c | l | l | l |}\hline 
	\text{superscript} \backslash \text{subscript}  &\ \  \emptyset & \text{not} \oplus & \text{not} \ominus \\\hline
	' & \Cbase \, \lambda^2 & \Cbase \, \lambda  &  \Cbase \, \lambda  \\ \hline
	+ & \Cbase \, \lambda & \Cbase & \Cbase \\ \hline
	- & \Cbase \, \lambda & \Cbase & \Cbase \\ \hline
	\end{array}\]
	with
	$\Cbase= \frac{\beta(1-\tau)^2}{2\rho}$
    (recall that $\lambda
	= \frac{1}{(1-\tau)^2}$ was defined in \cref{Prop:asymp_T}).
\end{proposition}
Note that the table is in fact a rank $1$ matrix.
Namely, 
passing from a root different from $\oplus$ (resp. $\ominus$)
to a nonconditioned root always adds a factor $\lambda$,
independently of the condition on the marked leaf.
Similarly, removing the leaf condition always yields the same factor $\lambda$
independently of the conditions on roots.

\begin{proof}%
By singular differentiation (see \cref{thm:singular_diff})
of \cref{eq:Exp_T,eq:DevT}, 
we have (near $\rho$)
\begin{align*}
&T'(z)= \frac{\beta}{2\rho(1-\tau)^2}  (1-\tfrac{z}{\rho})^{-1/2} +\O(1) = \Cbase\lambda^2 (1-\tfrac{z}{\rho})^{-1/2} +\O(1) ,\\
&T_{\nonp}'(z)=\frac{\beta}{2\rho}  (1-\tfrac{z}{\rho})^{-1/2} +\O(1)= \Cbase\lambda (1-\tfrac{z}{\rho})^{-1/2} +\O(1).
\end{align*} 

Since $T_{\nonm}=T_{\nonp}$, we have obtained the constants in the first line of the table. We now turn to the two last lines.
In order to use \cref{eq:tp,eq:tpm,eq:tpp} p.\pageref{eq:tp},
we first compute the expansions of all intermediate quantities appearing in these formulas. %
From \cref{eq:DevT}, we obtain the following expansion near $\rho$:
\begin{equation}\label{eq:DevW}
1+ W(z)= \left(\frac{1}{1-T_\nonp(z)}\right)^2 = \frac{1}{(1-\tau)^2} -\frac{2\beta}{(1-\tau)^3}\sqrt{1-\tfrac{z}{\rho}} + \mathcal{O}(1-\tfrac{z}{\rho}).
\end{equation}

We turn to the expansion of $S'(T)$.
Putting $u=T_{\nonp}$ in $\Lambda'(u)$ (which is given by \eqref{eq:LambdaPrime}) and using \cref{eq:T_Tnp}, we have
\begin{equation}\label{eq:DevS'T}
1+S'(T) = 1+S'(\tfrac {T_{\nonp}}{1-T_\nonp}) = (1-T_\nonp)^2\, (\Lambda'(T_\nonp) +1).
\end{equation}
Recall that $T_{\nonp}(\rho) = \tau < R_\Lambda$. 
Therefore, the composition $\Lambda'\circ T_\nonp$ is subcritical (see \cref{lem:Comp}). 
This implies that $\Lambda'\circ T_\nonp$ has a unique dominant singularity at $\rho$, 
and plugging in the asymptotic expansion \eqref{eq:DevT} of $T_\nonp$ at $\rho$, we obtain
\begin{align}\label{eq:DevLambda'T}
	\Lambda'(T_\nonp(z)) &= \Lambda'(\tau) + \Lambda''(\tau)(T_\nonp(z) - \tau) + O((T_\nonp(z) - \tau) ^2)\nonumber\\
	&= 1 - \tfrac {2\rho} {\beta}\racine +O\lineaire\,,
\end{align}
where we used the equalities $\Lambda'(\tau) = 1$ (by definition of $\tau$) and $\Lambda''(\tau) = \tfrac {2\rho}{\beta^2}$ (by definition of $\beta$). 
Combining \cref{eq:DevT} and \cref{eq:DevLambda'T} into \cref{eq:DevS'T}, we obtain, after simplification:
\begin{equation}\label{eq:DevF'T}
1 + S'(T) = 2(1-\tau)^2 - \left[ \tfrac{2\rho}{\beta} (1-\tau)^2 - 4\beta(1-\tau)\right]\racine + \O\lineaire
\end{equation}
The expansion of  $WS'(T) + W + S'(T)$ then follows from \cref{eq:DevW} and \cref{eq:DevF'T}:
\begin{align}
&WS'(T) + W + S'(T) = (1+W)(1+S'(T)) - 1 \nonumber \\
&= 2\tfrac{(1- \tau)^2}{(1- \tau)^2} -1 - \left[\tfrac{4\beta{(1- \tau)^2}}{(1-\tau)^3} + \tfrac{2\rho{(1- \tau)^2}}{\beta(1- \tau)^2} - \tfrac{4\beta{(1- \tau)}}{(1- \tau)^2} \right]\racine + \O\lineaire \nonumber\\
&= 1 - \tfrac {1}{\Cbase\lambda}\racine + \O\lineaire \ \text{ by definition of $\Cbase$ and } \lambda.
\end{align}

We can now derive the expansions of our generating functions, using \cref{eq:tp,eq:tpm,eq:tpp}. First,
$$
T^+=\frac{1}{1- \W S'(T) - \W - S'(T)}=\Cbase\lambda\left(1-\tfrac{z}{\rho}\right)^{-1/2}+\O(1).
$$
Then,
$$
T^+_{\nonm}=\frac{1}{1+W} \ T^+ = \frac{\Cbase\lambda}{\lambda}\left(1-\tfrac{z}{\rho}\right)^{-1/2}+\O(1) =\Cbase\left(1-\tfrac{z}{\rho}\right)^{-1/2}+\O(1).
$$
Since $\W S'(T) + \W + S'(T)$ takes value $1$ at $\rho$, the series $T^+_{\nonp}$ has the same first-order expansion:
$$
T^+_{\nonp}=(\W S'(T) + \W + S'(T)) T_{\nonm}^+= \Cbase \left(1-\tfrac{z}{\rho}\right)^{-1/2}+\O(1).
$$

By symmetry, we have $T^-=T^+$, $T_{\nonm}^{-}=T_{\nonp}^{+}$ and $T_{\nonp}^{-}=T_{\nonm}^{+}$,
and this completes the proof of the proposition. 
\end{proof}

\subsection{Asymptotics of the generating function of marked trees with a given induced tree}

We recall some notation introduced in \cref{Subsec:DecompTrees}.
Let $t_0$ be a substitution tree with $k\geq 2$ leaves and $e$ edges. 

\label{def:Vstar}
Let $V_*$ (resp. $V_+$, $V_-$) be the set of nonlinear nodes (resp. nodes labeled $\oplus$, $\ominus$) in $\Internal{\patterntree}$. Recall that, for $v \in \Internal{t_0}$, $d_v$ is the degree of $v$ and $\theta_v$ the permutation labeling  $v$,
and that $\TTT_{\patterntree}$ is the set of $\mathcal{S}$-canonical trees $t$ with $k$ marked leaves such that these leaves induce $\patterntree$. 
Denote by $T_{\patterntree}$ the generating function of $\TTT_{\patterntree}$ (where the size is the number of leaves, both marked and unmarked).
\begin{proposition}\label{prop:Asymp_Tt}
	The series $T_{\patterntree}$ has a unique dominant singularity in $\rho$, with the asymptotic expansion $T_{\patterntree} = \Const_{\patterntree} (1-\tfrac z \rho)^{-(e+1)/2} (1+o(1))$, where the constant $\Const_{\patterntree}$ is 
\begin{equation}\label{Eq:DefinitionCt}
\Const_{\patterntree} =\rho^k (\Cbase\lambda^2)^{e+1}\prod_{v\in V_*} \Occ_{\theta_v}(\tfrac{\tau}{1-\tau})\times \prod_{v\in V_+ \cup V_-} (\Occ_{\theta_v}(\tfrac{\tau}{1-\tau}) +  (1-\tau)^{d_v+1}).
\end{equation}
\end{proposition}
\begin{proof}
  By definition, $T_{\patterntree}=\sum_{V_s} T_{\patterntree,V_s}$,
  where the sum runs over sets $V_s$ such that $(\patterntree,V_s)$ is a decorated tree.
	We start from the formula for $T_{\patterntree,V_s}$, which is given by \cref{prop:Dec_TtS}.
	From \cref{prop:DevMarkedleaves}, the nine series $T',\ldots,T^-_{\nonm}$ of trees with one marked leaf all have unique dominant singularities in $\rho$. This is also the case for the functions $\Occ_{\theta_v}(T)$ and $\left(\tfrac 1 {1-T_\nonp}\right)^{d_v}$ by subcritical composition (see \cref{lem:Comp}).
	Indeed, from \cref{obs:RcvOcc>RS} the radius of convergence of $\Occ_{\theta}$ is at least $R_S$ and from \cref{lem:MMCversionSerie}, $T$ and $T_{\nonp}$ are convergent at $\rho$ with $T(\rho) = \tfrac \tau{1-\tau}< \tfrac {R_\Lambda}{1-R_\Lambda} = R_S$ and $T_\nonp (\rho) = \tau < 1$. 
	As a consequence, $T_{\patterntree}$ has a unique dominant singularity in $\rho$ (see \cref{lem:MultExp}).
    
    For exact asymptotics near $\rho$, note that $\Occ_{\theta_v}(T)$ and $\tfrac 1 {1-T_\nonp}$
    converge respectively to $\Occ_{\theta_v}(\tfrac{\tau}{1-\tau})$ and $\tfrac 1 {1-\tau}$ at $\rho$
    (see \cref{lem:MMCversionSerie}),
    while the nine series $T',\ldots,T^-_{\nonm}$ behave as $\cst (1-\tfrac{z}{\rho})^{-1/2}$,
    where the constants are given in \cref{prop:DevMarkedleaves}.
    We thus get using the notation of \cref{prop:Dec_TtS}:
    \begin{align*}
      z^k \ T^{\text{type of root}} &= \rho^k \Cbase\lambda^{1+\One_{\text{\tiny root} \in V_s}} (1-\tfrac z \rho)^{-1/2} (1+o(1));\\
      A_v&= (1-\tfrac z \rho)^{-d_v/2} \, (1+o(1)) \cdot \begin{cases}
        \Occ_{\theta_v}(\tfrac{\tau}{1-\tau}) (\Cbase\lambda^2)^{d'_v} (\Cbase\lambda)^{d_v^+ + d_v^-}  & \text{if }v \in V_s;\vspace{1mm}\\
        \left( \tfrac 1 {1-\tau}\right)^{d_v+1}                                                                              
        (\Cbase \lambda)^{d'_v} \Cbase^{d_v^+ + d_v^-}  & \text{if }v \notin V_s.
      \end{cases}
    \end{align*}
    The asymptotic behavior of $T_{\patterntree,V_s}$ near $\rho$ is then obtained by multiplying the above expressions.
    The formula can be simplified by observing that
    $\sum_{v\in \Internal{\patterntree}} (d'_v + d_v^+ + d_v^-) =\sum _{v\in \Internal{\patterntree}} d_v= e$ 
    and  $\sum_{v\in \Internal{\patterntree}}d'_v +\One_{\text{\scriptsize root} \in V_s} = |V_s|+k$,
    and we obtain:
\begin{align*}
	T_{\patterntree,V_s}&=(1+o(1)) (1-\tfrac z \rho)^{-(e+1)/2} \rho^k \Cbase^{e+1} \lambda^{1+k+|V_s|}
	\prod_{v\in V_s} \lambda^{d_v}\Occ_{\theta_v}(\tfrac{\tau}{1-\tau})
	\prod_{v\notin V_s} \left( \tfrac 1 {1-\tau}\right)^{d_v+1}\\
	&=(1+o(1)) (1-\tfrac z \rho)^{-(e+1)/2}  \rho^k (\Cbase\lambda^2)^{e+1}
	\prod_{v\in V_s} \Occ_{\theta_v}(\tfrac{\tau}{1-\tau})
	\prod_{v\notin V_s} \big[\left( \tfrac 1 {1-\tau}\right)^{d_v+1} \lambda^{-d_v-1}\big].
\end{align*}
To write the second line, we have used that 
\[
 \sum_{v\in V_s} d_v + \sum_{v\notin V_s} (d_v+1) +1+k+|V_s| =  \sum_{v\in \Internal{\patterntree}} d_v + |\Internal{\patterntree}| +k +1 = 2e+2.
\]
Now we have that $T_{\patterntree}$ is the sum of $T_{\patterntree,V_s}$
over sets $V_s$ such that $(\patterntree,V_s)$ is a decorated tree.
By definition, such $V_s$ can be written as $V_* \cup \widetilde{V_s}$ for some $\widetilde{V_s} \subset V_+ \cup V_-$
(the notation $V_*$, introduced right before the proposition, is the set of nonlinear nodes of $\patterntree$).
This change of variables leads to
\begin{multline*}
	T_{\patterntree} =(1+o(1)) (1-\tfrac z \rho)^{-(e+1)/2}  \rho^k (\Cbase\lambda^2)^{e+1} 
	\prod_{v\in V_*} \Occ_{\theta_v}(\tfrac{\tau}{1-\tau})\\
	\times \sum_{\widetilde{V_s} \subset V_+ \cup V_-}\left( \prod_{v\in \widetilde{V_s}} \Occ_{\theta_v}(\tfrac{\tau}{1-\tau})\right) 
	\left( \prod_{{v\in V_+ \cup V_-}\atop {v\notin \widetilde{V_s}}}
    \big[\left( \tfrac 1 {1-\tau}\right)^{d_v+1} \lambda^{-d_v-1} \big]\right).
\end{multline*} 
We first observe that since $\lambda=(1-\tau)^{-2}$, the last factor simplifies as $(1-\tau)^{d_v+1}$.
The proposition then follows by writing the sum of products on the second line as a product of sums.
\end{proof}

\subsection{Probability of tree patterns} \label{section:TreePatterns}
Recall that $\TTT$ is the set of $\mathcal{S}$-canonical trees (ie canonical trees of permutations in $\langle \mathcal{S}\rangle$). 
We take a uniform random tree with $n$ leaves in $\TTT$  and mark $k$ of its leaves, also chosen uniformly at random.
We denote by $\mathbf{t}^{(n)}_k$ the tree induced by the $k$ marked leaves.

\begin{proposition}\label{prop:proba_arbre}
Let $k\geq 2$, and let $\patterntree$ be any substitution tree with $k$ leaves. Then
	\[\proba (\mathbf{t}^{(n)}_k = \patterntree) = k! \frac{2\sqrt{\pi}}{\Gamma(\tfrac{e(t_0)+1} 2)}\frac{(1-\tau)^2}{\beta} \Const_{\patterntree} \, n^{e(\patterntree)/2+1-k} \, (1+o(1)),  \]
where $\Const_{\patterntree}$ is given by \cref{Eq:DefinitionCt}, and $e(\patterntree)$ is the number of edges of $\patterntree$.
\end{proposition}
\begin{proof}
  Directly from the definition, we have:
  \[ \proba(\mathbf{t}^{(n)}_k=\patterntree) = \frac{[z^n] T_{\patterntree}(z)}{\binom{n}{k} [z^n] T(z)}.\]
  The Transfer Theorem (\cref{thm:transfert}) gives us the asymptotic
  behavior of $[z^n] T(z)$ and $[z^n] T_{\patterntree}(z)$ from the asymptotic expansions in \cref{lem:MMCversionSerie} (\cref{eq:Exp_T})
  and \cref{prop:Asymp_Tt}.
  Deriving the result from there is a routine exercise.
\end{proof}

\subsection{Back to permutations}\mbox{ }

Let $\pi$ be a permutation of size $k$.
 Recall from \cref{sec:SubsTrees} that 
 an \textit{expanded tree} is a substitution tree where nonlinear nodes are labeled by simple permutations,
while linear nodes are required to be binary.
As in \cref{cor:OnExpandedTrees}, we denote
$\widetilde{N_\pi}$ the number of expanded trees of $\pi$.
We know (see \cref{cor:OnExpandedTrees}) that they have all the same number of linear nodes labeled $\oplus$ (resp. $\ominus$),
this number being denoted by $r_+$ (resp. $r_-$) and they all contain the same $r_*$ simple nodes, whose labels will be denoted as $\theta_1, \ldots,\theta_{r_*}$.

%
%
%
%
%
%
%
%
%
%
%
%
%
%

We introduce the \emph{default of binarity} of the permutation $\pi$:
\begin{equation}\label{eq:defi_db}
\db(\pi) %
= \sum_{i=1}^{r_*} (|\theta_i|-2).%
\end{equation}
Observe that $\db(\pi)=0$ if and only if $\pi$ is separable.

Finally, to state the next proposition, we also need to introduce the quantities 
\[\nu_+ = \Occ_{12}\left(\frac{\tau}{1-\tau}\right), \quad
\nu_- = \Occ_{21}\left(\frac{\tau}{1-\tau}\right)\, \text{ and} \quad
p = \frac {\nu_+ +(1-\tau)^3} {\nu_+ + \nu_- +2(1-\tau)^3},\quad
\]
Note that this is the same  $p$ as in \cref{Th:Main}.
\begin{proposition} 
	\label{prop:proba_patterns_general}
	Let $\pi \in \Sn_k$ with $k\geq 2$ and let $\Si_n$ be a  uniform random permutation in $\langle \SSS \rangle_n$. 
	With notation as above, we have 
	\begin{align*}
	\esper [ \occ (\pi,\Si_n)] =& \Const_\pi \, n^{-\db(\pi)/2} \, (1+o(1)), \\
	\text{where } \, \Const_\pi = & \widetilde{N_\pi}
	\frac{ k! \sqrt \pi}{2^{2k-2-\db(\pi)}\Gamma(k - \tfrac{\db(\pi)+1} 2)}
	p^{r_+}(1-p)^{r_-}
	\prod_{i=1}^{r_*}\left[\rho^{-1}\left(\tfrac{\beta}{(1-\tau)^2}\right)^{|\theta_i|}\Occ_{\theta_i}(\tfrac \tau {1-\tau})\right] .
	\end{align*}
\end{proposition}
\begin{proof}
  We denote by $\bm I$ a uniform random $k$-element subset of $[n]$
  and by $\bm t^{(n)}$ a uniform random $\mathcal{S}$-canonical tree with $n$ leaves. It holds that $\Si_n \stackrel{d}{=} \perm(\bm t^{(n)})$.
	As a consequence of \cref{eq:E(occ)=P(pat),lem:DiagrammeCommutatif}, we have
	\[\esper[\occ(\pi,\Si_n)] = \proba(\pat_{\bm I}(\Si_n) = \pi)
	=\proba(\perm({\mathbf{t}}^{(n)}_{\bm I}) = \pi)
	=  \sum_{\patterntree : \perm(\patterntree) = \pi} \proba ({\mathbf{t}}^{(n)}_{\bm I} = \patterntree). 
	\]
	After plugging in the estimate of \cref{prop:proba_arbre}, we get 
	\begin{align}
	\label{eq:Esper_is_sum_on_trees}
		\esper[\occ(\pi,\Si_n)] =  (1+o(1))\sum_{\patterntree:\perm(\patterntree) = \pi} k! \frac{2\sqrt{\pi}}{\Gamma(\tfrac{e(t_0)+1} 2)}\frac{(1-\tau)^2}{\beta}  \Const_{\patterntree} n^{-\db(\patterntree)/2}
	\end{align}
	where $\db(\patterntree) = 2k-2 -e(\patterntree)$ is the default of binarity of the tree $\patterntree$. 

We claim that if $\patterntree$ is a substitution tree of $\pi$, then $\db(\patterntree)\geq \db(\pi)$
with equality if and only if $\patterntree$ is an expanded tree.
Indeed,
\[\db(\patterntree) = e(\patterntree) + 2( k-1-e(\patterntree))                                    
     = \sum_{v \in \Internal{\patterntree}} (d_v -2).\]
Moreover from  \cref{lem:fromExpandedToAll} any substitution tree can be obtained from an expanded tree of $\pi$
by merging some internal nodes along edges connecting them and such merges always increase (strictly) the considered sum,
which proves the claim.

It follows that, in the sum of \cref{eq:Esper_is_sum_on_trees}, only expanded trees appear asymptotically. 
Moreover, $e(\patterntree)$ and the constant $\Const_{\patterntree}$ does not depend on the choice of an expanded tree $\patterntree$ of $\pi$. 
As a result, we get $\esper [ \occ (\pi,\Si_n)] =(1+o(1)) \Const_\pi \, n^{-\db(\pi)/2}$ where 
	\begin{multline}\label{eq:calcul_C_pi}
		\Const_\pi = \widetilde{N_\pi} \frac{k!2\sqrt{\pi}}{\Gamma(\tfrac {e+1} 2)}
		\frac{(1-\tau)^2}{\beta} \rho^k (\Cbase\lambda^2)^{e+1}
			\left(\prod_{i=1}^{r_*} \Occ_{\theta_i}(\tfrac{\tau}{1-\tau})\right)
			 \Big(\nu_+ + (1-\tau)^3 \Big)^{r_+}
			 \Big(\nu_- + (1-\tau)^3 \Big)^{r_-}\\
		=\widetilde{N_\pi} \frac{ k! 2 \sqrt \pi}{\Gamma(\tfrac{e+1} 2)}
			p^{r_+}(1-p)^{r_-}
			\left(\prod_{i=1}^{r_*} {\Occ_{\theta_i}(\tfrac{\tau}{1-\tau})} \right)
			\times
			\rho^k(\Cbase\lambda^2)^{e+1}
			\frac{(1-\tau)^2}{\beta}\Big(\nu_+ + \nu_- + 2(1-\tau)^3\Big)^{r_+ +r_-}
	\end{multline}
Differentiating \cref{eq:LambdaPrime} p.\pageref{eq:LambdaPrime}
and using $\Lambda'(\tau) = 1$
yield the identity $\Lambda''(\tau) = \tfrac 4 {1-\tau} + \tfrac 1 {(1-\tau)^4} S''(\tfrac {\tau}{1-\tau})$.
Moreover, since $\Occ_{12} +\Occ_{21} = \tfrac {S''}2$ (see \cref{obs:RcvOcc>RS}), this gives us 
	\begin{equation*}
		\nu_+ + \nu_- + 2(1-\tau)^3 = \frac{(1-\tau)^4}{2} \Lambda''(\tau) = \frac {(1-\tau)^4 \rho} {\beta^2}
	\end{equation*}
	Finally, after collecting everything together, we get
	\begin{align*}
		& \rho^k(\Cbase\lambda^2)^{e+1}
		\frac{(1-\tau)^2}{\beta}\Big(\nu_+ + \nu_- + 2(1-\tau)^3\Big)^{r_+ +r_-} \\
		& =\rho^k\left(\frac \beta {2\rho(1-\tau)^2}\right)^{e+1}
			\frac{(1-\tau)^2}{\beta}
			\left(\frac {\rho (1-\tau)^4} {\beta^2}\right)^{r_+ +r_-}
		= \frac{1}{\rho^{r_*} 2^{e+1}}
			\left(\frac{\beta}{(1-\tau)^2}\right)^{2r_* +\db(\pi)}, 
	\end{align*}
	where the last equality above has been obtained using that, for any expanded tree
    $\patterntree$ of $\pi$, we have
	$r_+ + r_- + r_* = |\Internal{\patterntree}| = e-k+1$
	and $\db(\pi) = 2k-2-e$.
	This allows us to simplify \cref{eq:calcul_C_pi} and yields the desired value of $\Const_\pi$.
\end{proof}
We can now conclude the proof of \cref{Th:Main}.
Let $\Si_n$ be a  uniform random permutation in $\langle \SSS \rangle_n$.
Our goal is to show that $\mu_{\Si_n}$ converges to the biased Brownian separable permuton of parameter $p$.
Let $\pi$ be any permutation of size $k\geq 2$. 
As a consequence of \cref{thm:randompermutonthm} (with \cref{obs:StartAt2}) and \cref{Def:PermutonBiaise},
we just have to show that 
  \[ \esper [\occ(\pi,\Si_n)] \stackrel{n\to \infty}{\to} 
  \frac{N_\pi}{\Cat_{k-1}} \, p^{r_+(\pi)} (1-p)^{r_-(\pi)}.\]
Assume first that $\pi$ is not separable. 
In this case, we have $N_\pi = 0$. It also holds that $\db(\pi)>0$, and \cref{prop:proba_patterns_general} implies that 
$\esper [\occ(\pi,\Si_n)] \to 0$. 

Assume on the contrary that $\pi$ is separable. In this case, $\db(\pi) = 0$, $\widetilde{N_\pi} = N_\pi$ and $r_*=0$. 
Therefore, from \cref{prop:proba_patterns_general} we get that 
  \[ \esper [\occ(\pi,\Si_n)] \stackrel{n\to \infty}{\to} 
  N_\pi \, p^{r_+(\pi)} (1-p)^{r_-(\pi)} \, \frac{k! \sqrt{\pi}}{2^{2k-2}\Gamma(k-\tfrac12)}
  = \frac{N_\pi}{\Cat_{k-1}} \, p^{r_+(\pi)} (1-p)^{r_-(\pi)},\]
where we have used the identity $ \Gamma(k-\tfrac12) = \frac{2^{3-2k}\sqrt{\pi}\, \Gamma(2k-2)}{\Gamma(k-1)} = \frac{2^{3-2k}\sqrt{\pi}(2k-3)!}{(k-2)!}$ 
coming from the duplication formula of the Gamma function.
This concludes the proof.\qed

\subsection{Occurrences of nonseparable patterns}\label{Sec:OccNonSeparables}
Since $\esper[\occ(\pi,\Si_n)]$ tends to $0$
whenever $\pi$ is a nonseparable pattern and the random variable  takes only nonnegative values,
$\occ(\pi,\Si_n)$ tends to $0$ in probability.

Here, we discuss more precisely the asymptotic behavior of $\occ(\pi,\Si_n)$ in this case.
The first result gives the order of magnitude of its moments;
we then present a consequence for the random variable itself.

\begin{proposition} \label{prop:momentsnonseparables}
	For $\pi \in \langle \SSS \rangle$ and $m\geq 1$, $\esper[(\occ(\pi,\Si_n))^m] = \Theta (n^{-\db(\pi)/2})$. 
\end{proposition}
\begin{remark}
	This result does consider separable patterns $\pi$, but in this case it is a direct consequence of our main theorem. Indeed, according to \cref{thm:randompermutonthm}, \cref{Th:Main} entails convergence in distribution of $(\occ(\pi,\Si_n))_n$ to $\occ(\pi,\Mu^{(p)})$, jointly for all $\pi\in\Sn$, and hence of all moments and mixed moments (since those random variables are bounded by $1$). Namely, we have $\esper[\occ(\pi,\Si_n)^m]\xrightarrow[n\to\infty]{} \esper[\occ(\pi,\Mu^{(p)})^m]$. 
    This limiting value is positive if and only if $\pi$ is separable,
    and can be computed exactly by adapting the method exposed in \cite[Section 9.1]{BrownianPermutation}.
\end{remark}
\begin{proof}%
By definition, $\occ(\pi,\Si_n) = \tbinom n k ^{-1} \sum_{I \subset [n], |I| = k} \One_{\pat_{I}(\Si_n) = \pi}$,
where we use $k$ for the size of the pattern $\pi$, as usual. Consequently,
	\begin{align*}
		 \esper [\occ(\pi,\Si_n)^m] = \tbinom n k ^{-m} \ \esper \bigg[ 
			\sum_{I_1,\ldots, I_m \subset [n] \atop \forall i,|I_i| = k} 
			\One_{\forall i,\pat_{I_i}(\Si_n) = \pi}
			\bigg]. 
	\end{align*}	
We split the sum according to the different possible values of $K = \bigcup_i I_i$ and $j = |K|$. 
Denoting $B^K_{k,m}$ the set of possible ordered covers of $K$ by $m$ sets of size $k$, this gives 
	\begin{align*}
		 \esper [\occ(\pi,\Si_n)^m] = \tbinom n k ^{-m} \ \esper \bigg[ 
		\sum_{j=k}^{mk} \sum_{K\subset [n] \atop |K|=j}
		\sum_{(I_1,\ldots, I_m) \in B^K_{k,m}}
		\One_{\forall i,\pat_{I_i}(\Si_n) = \pi}
		\bigg].
	\end{align*}
	Let us now remark that the unique increasing bijection between $K$ and $[j]$ induces a bijection between $B^K_{k,m}$ and $B^{[j]}_{k,m}$.
	Let $(J_i)_{1\leq i\leq m}$ denote the image of $(I_i)_{1\leq i \leq m}$ by this bijection. Clearly, 
	\begin{align} \label{eq:pi_pattern_of_rho}
	\pat_{I_i}(\Si_n) = \pi \iff \pat_{J_i}(\pat_K(\Si_n)) = \pi.
	\end{align}
	The sum can now be decomposed according to the different values of $\rho =\pat_K(\Si_n)$ yielding 
	\begin{align*}
		 \esper [\occ(\pi,\Si_n)^m] &= \tbinom n k ^{-m} \ \esper \bigg[ 
			\sum_{j=k}^{mk} \sum_{K\subset [n] \atop |K|=j}
			\sum_{(J_1,\ldots, J_m) \in B^{[j]}_{k,m}} \sum_{\rho \in \Sn_j}
			\One_{\pat_{K}(\Si_n) = \rho} \One_{\forall i,\pat_{J_i}(\rho)=\pi}
		\bigg]\\
		&= \sum_{j=k}^{mk} 
		\sum_{(J_1,\ldots, J_m) \in B^{[j]}_{k,m}}
		\sum_{\rho \in \Sn_j \atop \forall i,\pat_{J_i}(\rho)=\pi}
		\tbinom n k ^{-m} \tbinom n j \, \esper[\occ(\rho,\Si_n)].
	\end{align*}
    Since the summation index sets do not depend on $n$,
    it is enough to consider each summand separately to get the asymptotics.
    From \cref{prop:proba_patterns_general},
    the summand $\tbinom n k ^{-m} \tbinom n j \, \esper[\occ(\rho,\Si_n)]$ is of order $n^{j-km -\db(\rho)/2}$.
	
	Whenever \cref{eq:pi_pattern_of_rho} holds, $\pi$ is a pattern of $\rho =\pat_K(\Si_n)$. 
	As a consequence, an expanded tree of $\rho$ must have a substitution tree of $\pi$ as an induced tree. 
	Since the default of binarity may only decrease when taking induced trees, 
	this implies that $\db(\rho) \geq \db(\pi)$.
    Since additionally $j \le km$,
    we deduce that $j-km -\db(\rho)/2 \leq -\db(\pi)/2$ which gives $\esper[\occ(\pi,\Si_n)^m] = \O(n^{-\db(\pi)/2})$.
	
	To prove that $\esper[\occ(\pi,\Si_n)^m] = \Theta(n^{-\db(\pi)/2})$,
    it is then enough to find one summand, which grows as $n^{-\db(\pi)/2}$ for large $n$.
	This is achieved considering the summand indexed by
    \[j = km;\  J_i = \{ m\, q + i:\, 0 \le q \le k-1\}; \ \rho = \pi[1\cdots m,\ldots,1\cdots m]. \]
    Indeed in this case, $\db(\rho)=\db(\pi)$, so that $j-km -\db(\rho)/2 = -\db(\pi)/2$, 
	which concludes the proof of the proposition.
\end{proof}
\begin{corollary}
  \label{corol:momentsnonseparables}
	For $\pi \in \langle \SSS \rangle$ and  $\eps > 0$ small enough, $\proba(\occ (\pi,\Si_n)> \eps) = \Theta (n^{-\db(\pi)/2})$,
    where the constant in the $\Theta$ symbol depends on $\eps$.
\end{corollary}
\begin{proof}
	The upper bound is an immediate consequence of Markov's inequality. 
    For the lower bound, let $X$ be a random variable in $[0,1]$, we have
    \[\esper [X^2] \le \esper \big [\eps \One_{(X<\eps)}X+\One_{(X\ge \eps)}\big] \le \eps \esper[X] + \proba(X \ge \eps).\]
	The corollary follows by taking $X=\occ(\pi,\Si_n)$ and $\eps$ small enough.
\end{proof}

\section{Asymptotic analysis: The degenerate case $S'(R_S)<2/(1+R_S)^2-1$}
\label{sec:Asymp2}

In this section, we are interested in the case
where the generating function $S$ of simple permutations in $\mathcal S$
satisfies the following condition.
\begin{definition}
  [Hypothesis $(H2)$]
  The generating function $S$ of a family $\mathcal S$ of simple permutations is said to satisfy hypothesis $(H2)$ if $S$ meets the following conditions at its radius of convergence \hbox{$R_S>0$}:   
  \begin{enumerate}
  	\item $S'$ is convergent at $R_S$ and 
      \begin{equation}
        S'(R_S) < \frac{2}{(1+R_S)^2} -1;
        \label{eq:Hyp_Sp}
      \end{equation}
    \item\label{item:exp} $S$ has a      
      dominant singularity of exponent $\bm{\delta>1}$ in $R_S$.
  \end{enumerate}
\end{definition}

\cref{item:exp} means that, around the singularity $R_S$, one has
\[S(z)=g_S(z) + (C_S+o(1)) (R_S-z)^\delta,\]
for some analytic function $g_S$ and constant $C_S \neq 0$.
We refer to \cref{sec:singularity} for a precise definition.
Clearly, under $(H2)$, it holds that $R_S < \infty$. 
Note also that the assumption $\delta >1$ is redundant with the convergence of $S'$ at $R_S$.

\subsection{Asymptotic behavior of the main series}
As in \cref{ssec:Asymp_TreesGF_Std},
the first step is to derive the asymptotic behavior of all generating functions
for marked trees around their common dominant singularity.
In this section, we will not compute constants explicitly, but only
focus on the singularity exponent.
Indeed, keeping track only of singularity exponents is here sufficient 
to determine the limiting permuton.

The function $\Lambda$ is defined in \eqref{eq:DefLambda} by:
$$
    \Lambda(u)= \frac{u^2}{1-u} +S\left( \frac{u}{1-u} \right).
$$
\begin{lemma}
  Assume that $S$ satisfies hypothesis $(H2)$.
  Then $\Lambda$ has a unique dominant singularity of exponent $\delta$
  in $R_\Lambda:=\tfrac{R_S}{1+R_S}<1$.
  Moreover, $\Lambda'$ is convergent at $R_\Lambda$ and
      $\Lambda'(R_\Lambda)<1$.
  \label{lem:H2_De_S_A_Lambda}
\end{lemma}
\begin{proof}
  The first assertion follows from \cref{lem:Comp} (Supercritical case), using also that $R_S < \infty$. 
  The convergence of $\Lambda'$ at $R_\Lambda$ follows from that of $S'$ at $R_S$. 
  Finally, the inequality \hbox{$\Lambda'(R_\Lambda)<1$}
  is a straightforward computation from \eqref{eq:Hyp_Sp}
  (recall that $\Lambda'$ is given in \eqref{eq:LambdaPrime}).
\end{proof}

Recall that from \cref{eq:Tnonp} (p. \pageref{eq:Tnonp})
$T_{\nonp}$ is implicitly defined by the equation
\begin{equation}
  T_{\nonp}(z)=z + \Lambda(T_{\nonp}(z)).
  \label{eq:Tnonp2}
\end{equation}
As explained in \cref{Subsec:outline_proof}, the condition $\Lambda'(R_\Lambda)<1$ implies that the singularity
of $T_{\nonp}(z)$ is not a \emph{branch point}, but is inferred from the singularity of $\Lambda$.
%
%
\begin{lemma}
  Assume that $S$ satisfies hypothesis $(H2)$.
  Then there is  a unique $\rho>0$ such that $T_{\nonp}(\rho)=R_\Lambda$.
  Moreover, $\rho$ is the radius of convergence of $T_{\nonp}$
  and $T_{\nonp}$ has a unique dominant singularity of exponent $\delta$ in $\rho$.
  \label{lem:Analyse_Tnonp2}
\end{lemma}
The proof is rather technical and postponed to \cref{sec:proof_Inv2}.

Now since from \cref{eq:T_Tnp}: $T=\tfrac{T_{\nonp}}{1-T_{\nonp}}$, at   
the singularity $\rho$ of $T_{\nonp}$, the denominator of $T(\rho)$ is
\hbox{$1-T_{\nonp}(\rho)=1-R_\Lambda>0$}, hence the singularity
of $T$ is inherited from the one of $T_{\nonp}$.
More precisely, from \cref{lem:Comp} (Subcritical case), we have the following corollary.
\begin{corollary}
  \label{cor:sing_T}
  Assume that $S$ satisfies hypothesis $(H2)$.
  Then $T$ has a unique dominant singularity of exponent $\delta$ in $\rho$, with $T(\rho) = \tfrac{R_\Lambda}{1-R_\Lambda} = R_S$.
\end{corollary}

We now turn to the behavior of generating function of trees with one marked leaf.
\begin{lemma}
  Assume that $S$ satisfies hypothesis $(H2)$.
  Then each of the nine generating functions $T'$, $T^+$, \ldots, $T_{\nonm}^-$
   has a  unique dominant singularity of exponent $\delta-1$ in $\rho$.
  \label{lem:sing_Tp}
\end{lemma}
\begin{proof}
For $T'$ (resp. $T'_{\nonp}=T'_{\nonm}$) this follows immediately from \cref{cor:sing_T} (resp. \cref{lem:Analyse_Tnonp2}) 
by singular differentiation (\cref{lem:singular_diff}).

  Recall that $T^+$ is explicitly given in \cref{Prop:Sol_TPlus} as
\[T^+=\frac{1}{1- \W S'(T) - \W - S'(T)}\]
  where $\W=(\tfrac{1}{1-T_{\nonp}})^2-1$.
  We examine $\W$ and $S'(T)$ to determine the exponent of the singularity of $T^+$ at its radius of convergence (which we will prove to be $\rho$). 

Since $T_{\nonp}(\rho)=R_\Lambda <1$, again by subcritical composition, 
  \cref{lem:Analyse_Tnonp2} gives that $\W$ has a unique dominant singularity of exponent $\delta$ in $\rho$. 
  Moreover, $\W(\rho) = (\tfrac{1}{1-R_\Lambda})^2-1$. 

  As for $S'(T)$ we need to analyse $S'$ and $T$ and determine whether the composition is critical, supercritical or subcritical
  (see \cref{lem:Comp}).
  \begin{itemize}
   \item     By singular differentiation (\cref{lem:singular_diff}), it follows from $(H2)$ that 
  $S'$ has a dominant singularity of exponent $\delta -1$ in $R_S$. 
  In addition, $S'(R_S) < \frac{2}{(1+R_S)^2} -1$ by $(H2)$. 
\item Recall that from \cref{cor:sing_T}  $T$ has a unique dominant singularity of exponent $\delta$ in $\rho$, with $T(\rho) = \tfrac{R_\Lambda}{1-R_\Lambda} = R_S$. The composition $S' \circ T$ is therefore critical.
\item Moreover, $T$ is aperiodic since it counts a superset of the separable permutations.
  \end{itemize}
  From \cref{lem:Comp} (Critical case-A),  we obtain that $S'(T)$ has a unique dominant singularity 
  of exponent $\delta-1$ in $\rho$, and therefore so does $\W S'(T) + \W + S'(T)$ (using \cref{lem:MultExp}). 
  
  In $\rho$, the value of the series $\W S'(T) + \W + S'(T) = (\W +1) (S'(T)+1) -1 $ is less than 
  $(\tfrac{1}{1-R_\Lambda})^2 \tfrac{2}{(1+R_S)^2} -1 = 1$. 
  Therefore, by subcritical composition, $T^+=\frac{1}{1- \W S'(T) - \W - S'(T)}$ 
  has a unique dominant singularity of exponent $\delta-1$ in $\rho$. 
  Since $\frac{1}{1+\W}$ and $(\W S'(T) + \W + S'(T))$ have unique dominant singularities in $\rho$ of respective exponents $\delta$ and $\delta - 1$,
  the other cases follow by \cref{lem:MultExp}, using the formulas given in \cref{Prop:Sol_TPlus}.
\end{proof}

\subsection{Probability of given patterns}\label{sec:proba_pattern2}
Recall that the function $\Occ_{\theta}$ was defined in \cref{eq:defOcc} by $\Occ_\theta(z) = \sum_{\alpha\in \mathcal{S}} \mathrm{occ}(\theta,\alpha) z^{|\alpha|-|\theta|}$.

Unlike in the previous section, the functions $\Occ_{\theta}(z)$ 
will appear in the asymptotic behaviors, thus we need some additional assumptions on them.
First, as noticed in \cref{obs:RcvOcc>RS}, we have 
\begin{equation}\label{eq:relOccS}
\sum_{\theta \in \mathfrak S_k} \Occ_{\theta}(z)= \frac{1}{k!} S^{(k)}(z),
\end{equation}
which, under $(H2)$, has a dominant singularity of exponent $\delta-k$ in $R_S$ (see \cref{lem:singular_diff}, about singular differentiation).
The following hypothesis is thus reasonable.
\begin{definition}[Hypothesis $(CS)$]\label{Def:HypoCS}
Let $S$ have a dominant singularity of exponent $\delta>1$ in $R_S$.
  The family of simple permutations $\mathcal{S}$ satisfies the hypothesis $(CS)$ if,
  for each pattern $\theta$ of size $k$, the corresponding series $\Occ_{\theta}(z)$
  has a dominant singularity of exponent at least $\delta-k$ in $R_S$.
\end{definition}

Let $\patterntree$ be a substitution tree with $k$ leaves.  
Recall that $\Internal{\patterntree}$ is the set of internal nodes of $\patterntree$, 
and that for any node $v \in \Internal{\patterntree}$, $d_v$ denotes its number of children.
Recall also from \cref{Subsec:DecompTrees} (p. \pageref{Subsec:DecompTrees})
that $\TTT_{\patterntree}$ is the family of canonical trees with $k$ marked leaves 
which induce a tree equal to $\patterntree$, 
and that $T_{\patterntree}$ is its generating function.

Combining \cref{prop:Dec_TtS} and the above results, we obtain the following. 
\begin{proposition}
  For any substitution tree $\patterntree$ of size $k\geq2$, 
 assuming $(H2)$ and $(CS)$,
 the series $T_{\patterntree}$ has a unique dominant singularity of exponent at least 
 $\widetilde{e}_{t_0}$ in $\rho$, where
 \begin{itemize}
   \item $\widetilde{e}_{t_0} = \sum_{v: d_v > \delta} (\delta - d_v)$, if there is at least one node $v$ such that $d_v > \delta$;
   \item $\widetilde{e}_{t_0} = \delta - \max_{v \in \Internal{\patterntree}} d_v$ otherwise\footnote{Note for future reference that 
   the two expressions for $\widetilde{e}_{t_0}$ are equal when there is a unique $v$ such that $d_v > \delta$.\label{footnote_e_tilde}}.
 \end{itemize}
 \label{prop:sing_Tt2}
\end{proposition}
\begin{proof}
  First recall that $\TTT_{\patterntree}$ can be decomposed as a union $\bigcup_{V_s} \TTT_{\patterntree,V_s}$,
  where $V_s$ are subsets of $\Internal{\patterntree}$ 
  which contain all nonlinear nodes (see \cref{Subsec:CombSyst}).
  It is therefore enough to prove that each series $T_{\patterntree,V_s}$
  has a unique dominant singularity in $\rho$ with at least the desired exponent.

  Recall from \cref{prop:Dec_TtS} the following formula for $T_{\patterntree,V_s}(z)$:
$$
T_{\patterntree,V_s}= z^k T^{\text{type of root}}\prod_{v\in \Internal{\patterntree}} A_v,
$$
where each $A_v$ is given by \eqref{eq:A_v} and depends on the type of $v$.
Using hypothesis $(CS)$, \cref{cor:sing_T} and \cref{lem:Comp} (Critical case-A),
  we know that $\Occ_{\theta_v}(T)$ has a dominant singularity of exponent at least $\delta-d_v$
  in $\rho$, 
  and, from \cref{lem:sing_Tp},  that 
  it  is the term with the lowest exponent and
  the only possibly divergent term arising in $A_v$. 
  It then follows from \cref{lem:MultExp} that $A_v$ has a dominant singularity of exponent at least $\delta-d_v$        
  in $\rho$.                           

  The series $T_{\patterntree,V_s}(z)$ is then the product of the $A_v$'s and of some convergent series. 
  The result of the proposition is obtained from \cref{lem:MultExp}: 
  $T_{\patterntree,V_s}(z)$ has a dominant singularity of exponent at least $\sum_{v: d_v > \delta} (\delta - d_v)$ in $\rho$,
  if this sum is nonempty (\emph{i.e.}, when one of the series is divergent) and 
  of the smallest singularity exponent among them, that is
  $\delta - \max_{v \in \Internal{\patterntree}} d_v$  otherwise, \emph{i.e.}, if all factors are convergent.
\end{proof}

As in \cref{section:TreePatterns} (p. \pageref{section:TreePatterns}), we take a uniform random tree with $n$ leaves in $\TTT$ with $k$ marked leaves  (chosen also uniformly at random).
Denote as before by $\mathbf{t}_n^k$ the tree induced by the $k$ marked leaves.

 \begin{corollary}
   Assume $(H2)$ and $(CS)$.
   For any substitution tree $\patterntree$ with $k \geq 2$ leaves,
   the probability $\proba({\mathbf{t}^{(n)}_k}=\patterntree)$ tends to $0$,
   unless $\patterntree$ has only one internal node.
   \label{cor:SubCritical_OnlyOneIntVertex}
 \end{corollary}
 \begin{proof}
   As in \cref{section:TreePatterns}, we use the formula
   \[\proba({\mathbf{t}^{(n)}_k}=\patterntree) = \frac{[z^n] T_{\patterntree}(z)}{\binom{n}{k} [z^n] T(z) }\]

   From \cref{thm:transfert} and \cref{prop:sing_Tt2} (and using the notation $\widetilde{e}_{t_0}$ herein defined),
   we get that
   \[ [z^n] T_{\patterntree}(z) = (\tilde{C}_{\patterntree} +o(1)) \rho^{-n} n^{-1-\widetilde{e}_{t_0} } \]
   for some constant $\tilde{C}_{\patterntree}$, possibly equal to $0$.
   On the other hand, \cref{thm:transfert} and \cref{cor:sing_T} imply that 
   \[ [z^n] T(z) = C \rho^{-n} n^{-\delta-1} (1+o(1)),\]
   for some constant $C\ne 0$.
   Putting everything together, we obtain 
   \[ \proba(\mathbf{t}^{(n)}_k=\patterntree)  =\frac{k! \, (\tilde{C}_{\patterntree}+o(1))}{C} n^{e_{\patterntree}},\]
   where $e_{\patterntree}= \delta-k-\widetilde{e}_{t_0}$.

   For any subset $A$ of the internal nodes of a tree $t$ with $k$ leaves,
   we claim  that the following inequality holds: $\sum_{v \in A} d_v \le |A|+k-1$.
   It is clear when $k=1$. 
   For $k>1$, we decompose $t$ as a root $\varnothing$ with $d\geq2$ subtrees $t_1, \dots, t_d$. 
   The chosen set $A$ of nodes of $t$ determines a set $A_i$ of nodes in each tree $t_i$ that has $k_i$ leaves. 
   Assume that $\sum_{v \in A_i} d_v  \leq |A_i|+k_i-1$ for all $i$.
   Then, we have 
\begin{multline*}
\sum_{v \in A} d_v = d \cdot \One_{\varnothing \in A} + \sum_{i=1}^d \sum_{v \in A_i} d_v 
 \leq d \cdot \One_{\varnothing \in A} + \sum_{i=1}^d |A_i|+k_i-1 \\
 = d \cdot \One_{\varnothing \in A} + |A| - \One_{\varnothing \in A} +k -d = |A|+k-1 + (d-1)(\One_{\varnothing \in A}-1),
\end{multline*}
and with the observation that $(d-1)(\One_{\varnothing \in A}-1) \leq 0$, this proves our claim.
   Set $A=\{v \in \Internal{\patterntree}: d_v >\delta\}$; if $|A| \ge 1$, one has
   \[ e_{\patterntree} = \delta-k-\delta|A| + \sum_{v \in A} d_v \le \delta -k -\delta |A| + |A|+k-1 = (1-|A|) (\delta -1),\]
   which is negative for $|A| \ge 2$ (indeed, $\delta >1$ by $(H2)$).
   When $|A|=0$ or $|A|=1$, we have  $\widetilde{e}_{t_0}=\delta - \max_{v \in \Internal{\patterntree}} d_v$
   (see also \cref{footnote_e_tilde})
   and, therefore, $e_{\patterntree}= \max_{v \in \Internal{\patterntree}} d_v -k$ is negative unless 
   $\patterntree$ has exactly one internal node (which is then of degree $k$).
    \end{proof}
 
It is now straightforward to translate this result in terms
of the probability to find a given pattern in a random permutation in the set $\CCC:=\langle \SSS \rangle$. 
As recalled in \cref{sec:singularity}, the hypothesis $(CS)$ is equivalent to the following: for every $k \ge 1$
  and every permutation $\theta$ of size $k$,
  there exists an analytic function $g_\theta$ and a constant $C_\theta$ (possibly equal to $0$)
  such that, on an $\Delta$-neighborhood of $R_S$, it holds that
  \begin{equation}
    \Occ_\theta(z)=g_\theta(z) + (C_\theta+o(1)) (R_S - z)^{\delta-k}.
    \label{eq:CS}
  \end{equation}
  The quantities $C_\theta$ are involved in the statement of the following theorem.
\begin{theorem}\label{Th:MainH2}
  Let $\Si_n$ be a uniform random permutation in $\langle \mathcal{S}\rangle_n$.
  We assume hypotheses $(H2)$ and $(CS)$. Then, for any $k\geq 2$ and for any $\pi\in \Sn_k$,
  \begin{equation}
    \lim_{n\to \infty} \esper [\occ(\pi,\Si_n)] = \frac{C_\pi}{\sum_{\theta \in \Sn_k} C_\theta}. 
    \label{eq:limit_Pattern_Supercritical}
  \end{equation}
  Consequently, there exists a random permuton $\Mu_\CCC$ with
  \begin{equation}
    \esper [\occ(\pi,\Mu_\CCC)]= \frac{C_\pi}{\sum_{\theta \in S_k} C_\theta} 
    \label{eq:def_MuCsc}
  \end{equation}
such that $\left(\mu_{\Si_n}\right)_n$ tends to $\Mu_\CCC$ in distribution.
\end{theorem}
\begin{proof}
  The starting point is the same as in the standard case (\cref{prop:proba_patterns_general}).
  As before we denote by $\bm I$ a random uniform $k$-element subset of $[n]$, and ${\mathbf{t}^{(n)}_k}$ is the tree of size $k$ induced by $k$ uniform leaves in a uniform canonical tree of size $n$ in $\TTT$.
  From \cref{lem:DiagrammeCommutatif}, we have
	\[
		\esper(\occ(\pi,\Si_n)) = \proba(\pat_{\bm I}(\Si_n) = \pi)
		=\proba(\perm({\mathbf{t}^{(n)}_k}) = \pi) 
		= \sum_{\patterntree:\perm(\patterntree) = \pi} \proba ({\mathbf{t}^{(n)}_k} = \patterntree),
	\]
    where the sum runs over all substitution trees encoding $\pi$.
    Denote by $\patterntree^\pi$ the substitution tree with only one internal node
    labeled by $\pi$.
    When $n$ tends to infinity, using \cref{cor:SubCritical_OnlyOneIntVertex},
    we know that every term in the above sum vanishes,
    but the term corresponding
    to $\patterntree^\pi$:
    \begin{equation}
      \lim_{n \to \infty} \esper(\occ(\pi,\Si_n)) = 
    \lim_{n \to \infty} \proba ({\mathbf{t}^{(n)}_k} = \patterntree^\pi).
    \label{eq:OccPerm_OccTree_Super}
  \end{equation}
  Now we can compute directly from \cref{prop:Dec_TtS} that
  \[T_{\patterntree^\pi} = \One_{\pi \text{ linear}}z^kT^+\left( \tfrac 1 {1-T_\nonp} \right)^{k+1} (T')^{k} + 
  z^k\Occ_{\pi}(T)(T')^{k+1}.
  \]
  The first term has a dominant singularity of exponent at least $\delta-1$,
  while the dominant singularity exponent of the second term is at least $\delta-k$.
We therefore focus on the second term and get by an easy computation
that, around $z=\rho$, we have
 $$ \Occ_{\pi}(T(z))= g(z) + \left(C_\pi T'(\rho)^{\delta-k} +o(1) \right) (\rho-z)^{\delta-k},$$
  and thus $$T_{\patterntree^\pi}(z) = h(z) + \left (\rho^k T'(\rho)^{\delta+1} C_\pi +o(1) \right ) (\rho-z)^{\delta-k},$$
  where $g$ and $h$ are analytic functions.
  Then we can apply the Transfer Theorem (\cref{thm:transfert}) to $T$ and $T_{\patterntree^\pi}$
  (as in the proof of \cref{cor:SubCritical_OnlyOneIntVertex})
    and obtain that $\lim_{n \to \infty} \proba ({\mathbf{t}^{(n)}_k} = \patterntree^\pi)$
    is proportional to $C_\pi$ for $\pi \in \Sn_{k}$.
    Since the left-hand side of \cref{eq:OccPerm_OccTree_Super} sums to one
    (when summed over $\pi \in \Sn_k$, for a fixed $k$),
    this proves \cref{eq:limit_Pattern_Supercritical}.
    
    The rest of the statement follows immediately, using \cref{thm:randompermutonthm} (with \cref{obs:StartAt2}).
\end{proof}

\subsection{Hypothesis $(CS)$ and convergence of uniform random simple permutations}
We may wish to replace the hypothesis $(CS)$ with a less technical hypothesis,
such as the
convergence of a random \emph{simple} permutation in our set $\SSS$
to the random permuton $\Mu_\SSS$ uniquely determined by \cref{eq:def_MuCsc}.
We show here that, though not equivalent, these hypotheses are strongly related.
Remark that we do not assume $(H2)$ here, so that the existence of  $\bm \mu_\SSS$ cannot be inferred from \cref{Th:MainH2}
and needs a separate proof.

%
%
%
\begin{proposition}\label{Prop:H2degenerate}
Suppose that $S$ has a dominant singularity of exponent $\delta>1$ and 
assume condition $(CS)$. Then there exists a permuton $\Mu_{\SSS}$ such that
\begin{equation}
\esper [\occ(\pi,\Mu_\SSS)]= \frac{C_\pi}{\sum_{\theta \in S_k} C_\theta},
\label{eq:def_MuCsc2}
\end{equation}
where the $C_\pi$ are given by \cref{eq:CS} (which holds under hypothesis $(CS)$).

Let $\bm\alpha_n$ be a uniform random permutation of size $n$ in $\SSS$. 
If $\left(\mu_{\bm\alpha_n}\right)$ converges in distribution, then its limit is $\Mu_{\SSS}$. 
Conversely, if we assume  that $S$ and all series
      $\Occ_\theta$ have a unique dominant singularity,
      then $\left(\mu_{\bm\alpha_n}\right)$ converges in distribution
      (and the limit must be $\Mu_\SSS$, using the first part of the proposition).
\end{proposition}
Before giving the proof, let us do the following observation.
If both $(H2)$ and $(CS)$ are satisfied, then we can apply both \cref{Th:MainH2,Prop:H2degenerate}.
By comparing \cref{eq:def_MuCsc,eq:def_MuCsc2}, we have $\Mu_\CCC=\Mu_\SSS$ in distribution
(recall that the distribution of a random permuton is determined by its expected pattern densities, 
see \cref{Prop:CaracterisationLoiPermuton}).
In particular, assuming that a uniform random {\em simple} permutation in the class converges in distribution to some random permuton,
then this random permuton is $\Mu_\SSS=\Mu_\CCC$, that is the limit of a uniform random permutation in the class.
This justifies a claim in the introduction.
\begin{proof}
  We start with the existence of $\Mu_{\SSS}$. For every $k\geq 1$, let $\bm \rho_k$ be a random permutation in $\Sn_k$ such that $\proba(\bm \rho_k = \pi) = C_\pi /({\sum_{\theta \in S_k} C_\theta})$. By \cref{Prop:existence_permuton}, we only need  to show that $(\bm \rho_k)_k$ forms a consistent family. Let $1\leq k \leq n$, then for $\pi \in \Sn_k$,
  \begin{align}
  	\proba(\pat_{{\bm I}_{n,k}}(\bm \rho_n) = \pi) = \frac 1 {\binom n k \sum_{\theta \in S_k} C_\theta} \sum_{\sigma\in \Sn_n} \,\mathrm{occ}(\pi,\sigma) C_\sigma.
  	\label{eq:tech10}
  \end{align}
  On the other hand, 
  the following combinatorial identity can be derived from the definition of the $(\Occ_\theta)_{\theta \in \Sn}$:
  \begin{equation}
    \frac{1}{(n-k)!} \Occ_\pi^{(n-k)}(z) = \sum_{\sigma\in \Sn_n} \,\mathrm{occ} (\pi,\sigma) \Occ_\sigma(z).
    \label{eq:Tech12}
  \end{equation}
  Indeed, the left-hand side is the series of simple permutations in $\mathcal S$ whose entries are partitoned into a 
   set of $k$ marked entries forming a pattern $\pi$
  and a set of $n-k$ marked entries.
  The right-hand side counts the same object,
  according to the pattern $\sigma$ formed by all the $n$ marked entries.
  To distinguish the marked entries of the first set from the ones of the second set,
  we need to specify a subpattern $\pi$ inside the pattern $\sigma$,
  which explains the factor $\mathrm{occ}(\pi,\sigma)$.

  We now differentiate both sides of \cref{eq:Tech12} $m$ times so that $\delta - n - m <0$, and replace all series with their asymptotic estimates obtained  thanks to hypothesis $(CS)$, \cref{eq:CS} and singular differentiation (\Cref{thm:singular_diff})
  \footnote{For $x\in \mathbb C$ and $r\in \mathbb N$, we denote by $(x)_r$ the falling factorial $x(x-1)\cdots(x-r+1)$}.
  \begin{multline*}
   g_\pi^{(m)}(z) +(-1)^{m+n - k}(\delta-k)_{m+n-k}(C_\pi+o(1))(R_S - z)^{\delta - m - n} \\
  = (n-k)! \sum_{\sigma\in \Sn_n} \,\mathrm{occ} (\pi,\sigma)
  \left[ g_\sigma^{(m)}(z) + (-1)^m(\delta - n)_{m} (C_\sigma + o(1))(R_S - z)^{\delta - m - n} \right]
  \end{multline*}
  As only the singular parts diverge, taking the limit in $z\to R_S$ allows to identify the constants, yielding $\sum_{\sigma \in \Sn_n}\mathrm{occ} (\pi,\sigma) C_\sigma = (-1)^{n-k}\binom {\delta - k}{n-k} C_\pi$. Plugging this back in \cref{eq:tech10} yields 
  \begin{equation*}
  \proba(\pat_{{\bm I}_{n,k}}(\bm \rho_n) = \pi) \propto C_\pi, \pi \in \Sn_k.
  \end{equation*}
  As probabilities sum to $1$, we get $\proba(\pat_{{\bm I}_{n,k}}(\bm \rho_n) = \pi) = \proba(\bm \rho_k = \pi)$, proving the consistency of $(\bm \rho_k)_k$.
  \medskip
  
  As noticed in the proof of \cref{Th:MainH2}, the convergence of $\left(\mu_{\bm\alpha_n}\right)$ to $\Mu_\SSS$ is equivalent to the following:
  for any fixed $k \ge 2$ and any $\pi \in \Sn_k$, the limit
  $\lim_{n\to\infty} \esper\big[ \occ(\pi,\bm\alpha_n)  \big]$ exists and is proportional to $C_\pi$ for $\pi \in \Sn_k$. Furthermore, by consistency, we only need to show it for large $k$.
  Directly from the definitions, we have
  \begin{equation}
   \esper\big[ \occ(\pi,\bm\alpha_n)  \big] =\,
  \frac{k!\,[z^{n-k}] \Occ_\pi(z)}{[z^{n-k}]S^{(k)}(z)}. 
  \label{eq:Tech11}
\end{equation}
  
  For the proof of the direct implication, we assume that $\mu_{\bm\alpha_n}$ converges in distribution.
  From \cref{thm:randompermutonthm}, this means that $\esper\big[ \occ(\pi,\bm\alpha_n)  \big]$ has a limit $\Delta_\pi$ for every $\pi$ of size $k \ge 2$.
  Then when $n$ goes to infinity,
  \[[z^{n}] \Occ_\pi(z) = \frac {\Delta_\pi + o(1)}{k!} [z^{n}] S^{(k)}(z).\]
  As a consequence, for any fixed $\pi$ and $\eps>0$, there exists polynomials $g_-$,$g_+$ such that for any real $z$ in $[0,R_S)$, 
  \begin{equation}\label{eq:inegal_cvCS}
  \left(\frac{\Delta_\pi}{k!} - \eps\right)S^{(k)}(z) + g_-(z) \leq \Occ_{\pi}(z) \leq \left(\frac{\Delta_\pi}{k!} + \eps\right)S^{(k)}(z) + g_+(z).
  \end{equation}
  Hypothesis $(CS)$ implies that in $R_S$ we have $\Occ_{\pi}(z) = g_\pi(z) +(C_\pi+o(1))(R_S - z)^{\delta - k}$ for some analytic function $g_\pi$. Also $S^{(k)}$ has a dominant singularity of exponent $\delta-k$ in $R_S$ so $S^{(k)}(z) = g_{S^{(k)}}(z) + (C_{S^{(k)}}+o(1))(R_S - z)^{\delta-k}$ for some analytic function $g_{S^{(k)}}$ and constant $C_{S^{(k)}}>0$. Plugging these asymptotic estimates into \eqref{eq:inegal_cvCS} yields
  \begin{multline*}\left(\frac{\Delta_\pi}{k!} - \eps\right)
  \left[(C_{S^{(k)}} +o(1)) (R_S - z)^{\delta - k} +g_{S^{(k)}} \right]
  +g_-\\
  \leq (C_\pi +o(1))\left[(R_S - z)^{\delta  - k} +g_\pi\right] \\
  \leq \left(\frac{\Delta_\pi}{k!} + \eps\right)
  \left[(C_{S^{(k)}} +o(1)) (R_S - z)^{\delta - k} +g_{S^{(k)}} \right]
  +g_+.
  \end{multline*}
  Let $k$ be such that $\delta - k < 0$, so that the singular parts are the only diverging quantities when $z\to R_S$. After taking the limit we get $\left|C_\pi - \frac{C_{S^{(k)}}}{k!}\Delta_\pi \right| \leq \eps$ for every $\eps$ and hence equality. We have proven that $(\Delta_\pi)_{\pi\in \Sn_k}$ is proportional to $(C_\pi)_{\pi\in \Sn_k}$ for large $k$, as desired.

  For the converse, we start from \cref{eq:Tech11}.
  \cref{thm:transfert} (which we can apply because of the hypotheses on $S$
  and $\Occ_\theta$) gives the following asymptotic behavior when $n\to\infty$:
  \[ \esper\big[ \occ(\pi,\bm \alpha_n)  \big] 
  = k!\, n^{-k} 
  \,\frac{(C_\pi+o(1)) R_S^{-n+k} n^{-\delta+k-1}}{C_S R_S^{-n} n^{-\delta-1}}.\]
  For fixed $k$, the limit of the right-hand side is proportional to $C_\pi$,
  which concludes the proof of the proposition.
\end{proof}

\section{Asymptotic analysis: The critical case $S'(R_S)=2/(1+R_S)^2-1$}
\label{sec:critical}

The goal of this section is to describe the limiting permuton
of a uniform permutation in a substitution-closed class $\CCC$,
whose set of simple permutations satisfies the following hypothesis.
\begin{definition}
  [Hypothesis (H3)]
  A family $\mathcal S$ of simple permutations is said to satisfy hypothesis $(H3)$ 
  if the generating function $S$ meets the following conditions at its radius of convergence $R_S>0$:
  \begin{itemize}
    \item $S$ has a      
      dominant singularity of exponent $\bm{\delta>1}$ in $R_S$;
    \item $S'$ is convergent at $R_S$ and 
\[ S'(R_S) = \frac{2}{(1+R_S)^2} -1.\]
  \end{itemize}
\end{definition}

This hypothesis implies the following behavior of $\Lambda$ near its singularity.
\begin{lemma}
  Assume that $S$ satisfies hypothesis $(H3)$ and set, as before,
\[
    \Lambda(u)= \frac{u^2}{1-u} +S\left( \frac{u}{1-u} \right).
\]
  Then $\Lambda$ has a unique dominant singularity of exponent $\delta$
  in $R_\Lambda:=\tfrac{R_S}{1+R_S}<1$.
  Moreover, $\Lambda'$ is convergent at $R_\Lambda$ and
      $\Lambda'(R_\Lambda)=1$.
  \label{lem:H3_De_S_A_Lambda}
\end{lemma}
\begin{proof}
  The statement on the singularity exponent follows from \cref{lem:Comp} (Supercritical case).
  The statement on $\Lambda'$ is a simple computation from (H3) and \cref{eq:LambdaPrime}.
\end{proof}

In the following, we denote $\delta_*=\min(\delta,2)$.
The behavior of $T_{\nonp}$ is given in the following lemma,
whose proof is postponed to \cref{sec:proof_Inv3}.
Note that the exponent is different from the one observed
in \cref{sec:Asymp2} (the singularity comes here from 
a mixture of a branch point and of the singularity of $\Lambda$, as explained in \cref{Subsec:outline_proof}).
\begin{lemma}
  \label{lem:Inversion3}
  Assume that $S$ satisfies hypothesis $(H3)$.
  Then there is  a unique $\rho>0$ such that $T_{\nonp}(\rho)=R_\Lambda$.
  Moreover, $\rho$ is the radius of convergence of $T_{\nonp}$
  and $T_{\nonp}$ has a unique dominant singularity of exponent $1/\delta_*$ in $\rho$.
\end{lemma}

From here, the strategy is similar as the previous sections,
we therefore skip unnecessary details.
As in the previous section, we deduce immediately the asymptotic behavior
of all generating functions of marked trees:
\begin{corollary}
  Assume that $S$ satisfies hypothesis $(H3)$ and define $\rho$ as above.
  Then $T$ has a unique dominant singularity of exponent $1/\delta_*$ in $\rho$, with $T(\rho) = R_S$.
  Moreover each of the generating functions $T'$, $T_{\nonp}'$, \ldots, $T^-_{\nonm}$
  has a unique dominant singularity of exponent $1/\delta_*-1$ in $\rho$.
\end{corollary}

To go further, we have to assume hypothesis $(CS)$,  \emph{i.e.} that, for any $\theta$ of size $k$,
the series $\Occ_\theta(z)$ has a unique dominant singularity of exponent at least $\delta-k$ in $R_S$.
The composition $\Occ_\theta \circ T$ is then critical,
and from \cref{lem:Comp} (Critical case-B), $\Occ_\theta(T(z))$ has a unique singularity of exponent
at least $\tfrac{1}{\delta_*}\min(\delta-k,1)$ in $\rho$.

By a straightforward computation from \cref{prop:Dec_TtS}, we get the following singularity for the series $T_{\patterntree,V_s}$.
\begin{lemma} \label{lem:Sing_Tt3}
	Fix a tree $\patterntree$ and a subset $V_s$ of its internal nodes.
	Then the series 	$T_{\patterntree,V_s}$ has a singularity of exponent at least
	\[ \widetilde{e}_{(\patterntree,V_s)}:=
	(|E|+1)  (\tfrac{1}{\delta_*}-1) + \sum_{v \in V_s} \tfrac{1}{\delta_*} \min(\delta-d_v,0).\]
\end{lemma}
The behavior of the uniform permutation in $\langle \SSS \rangle_n$ now depends on whether $\delta$ is smaller or larger than $2$.
\subsection{The case $\delta\in (1,2)$.}\label{ssec:deltaPlusPetit}
\begin{theorem}\label{Th:Main4}
   	Let $\mathcal S$ be a family of simple permutations verifying hypothesis $(H3)$ and $(CS)$, with $\delta \in(1,2)$.
    We consider the permuton $\Mu_{\SSS}$ as in \Cref{Prop:H2degenerate} and denote for $\pi \in \Sn_k$
   	\[\Delta_\pi = \esper[\occ(\pi,\bm\mu_{\SSS})] = \frac{C_\pi}{\sum_{\theta \in \Sn_k} C_\theta}.\]
    Let also 
   	\[\nu_{\delta,k}(\patterntree) = \frac{k!}{(\delta -1) \cdots ((k-1)\delta - 1)}\prod_{v\in \Internal\patterntree} \frac{(d_v - 1 - \delta) \cdots (2-\delta)(\delta - 1)}{d_v!}\]
 be 
   	the probability distribution of the induced subtree with $k$ leaves 
    in the $\delta$-stable tree (see \cite[Thm 3.3.3]{DuquesneLeGall}).

    If $\Si_n$ is a uniform permutation in $\langle \mathcal{S}\rangle_n$, then
   	\begin{equation}\label{eq:limoccStableCase} \lim_{n\to\infty}{\esper[\occ(\pi,\Si_n)]} = \sum_{\patterntree:\perm(\patterntree) = \pi} \nu_{\delta,k}(\patterntree)\prod_{v\in \Internal\patterntree} \Delta_{\theta_v}.\end{equation}
   	As a consequence, $\mu_{\Si_n}$ converges in distribution to a random permuton, whose average pattern densities are determined by \cref{eq:limoccStableCase} and depend only on $\delta$ and $\bm\mu_{\SSS}$. We call it the $\delta$-stable permuton driven by $\bm\mu_{\SSS}$.
\end{theorem}
A construction of the $\delta$-stable permuton driven by $\bm \nu$ for every $\delta \in (1,2)$ and random permuton $\bm \nu$ is given in \cref{Lem:ConstructionElementaireStable}.
\begin{remark}
	In this case, all possible patterns, in particular nonseparable ones,
    appear with positive probability in the limit
    (as long as they appear with positive probability in a uniform simple permutation in the class).
    More precisely, the proof will show the following:
    $k$ random leaves in a uniform canonical tree induce substitution trees with arbitrary large node degrees, and  the first common ancestors of those leaves are all simple permutations with probability tending to 1.
\end{remark}
\begin{proof}
	We start from the estimate of \cref{lem:Sing_Tt3}. In the present case $\delta_*=\delta$ and $\delta-d_v$ is always negative.
	If we fix $(\patterntree,V_s)$ and apply \cref{thm:transfert} to $T$ and $T_{\patterntree,V_s}$, we get
	\[\frac{[z^n]T_{\patterntree,V_s}}{\binom n k [z^n]T}=(\mathrm{cst}+o(1))n^{e_{(\patterntree,V_s)}},\]
	with
        \begin{equation}
        e_{(\patterntree,V_s)}
        = \tfrac{1}{\delta} - k - \widetilde{e}_{(\patterntree,V_s)}
        =\tfrac{1}{\delta} - k -(|E|+1)  \left(\tfrac{1}{\delta}-1\right) +\sum_{v \in V_s}(\tfrac{d_v}{\delta} - 1) 
	= -\sum_{v \in \Internal{\patterntree} \setminus V_s} (\tfrac{d_v}{\delta} - 1),
	\label{eq:exposant}
        \end{equation}
	where the last identity uses $|E|-|\Internal{\patterntree}| -k +1=0$.
	Therefore $e_{(\patterntree,V_s)}<0$
	unless $V_s=\Internal{\patterntree}$,  \emph{i.e.}
	all internal nodes of $\patterntree$ are in $V_s$.	As a consequence, if $\bm t^{(n)}$ is a uniform canonical tree of size $n$ and $\bm I$ is a uniform subset of the leaves of length $k$, 
	\[\proba(\bm t^{(n)}_{\bm I} = \patterntree) = \sum_{V_s \subseteq\, \Internal {t_0}} \frac{[z^n]T_{\patterntree,V_s}}{\binom n k [z^n]T}  = \frac{[z^n]T_{\patterntree,\Internal{t_0}}}{\binom n k [z^n]T} +o(1).\]
	
Now we focus on the case $V_s = \Internal {t_0}$. The series $T_{\patterntree,\Internal{\patterntree}}$ has a dominant singularity of exponent 
$\widetilde{e}_{(\patterntree,\Internal {t_0})} = 1/\delta-k$ (see \cref{eq:exposant}). 
We are left with identifying the constant in its singular expansion.
    In what follows, we will always denote by $C_A$ 
    the constant in front of the singular part of the expansion of the analytic function $A$,
    \emph{i.e.} $A(z) = g_A(z) + C_A(R_A - z)^{\delta_A}$, with $R_A$ the radius of convergence of $A$ and $g_A$ analytic at $R_A$. 
    In particular, given that $\delta \in (1,2)$ and $k \ge 2$, we have the following expansions (recall $T(\rho)=R_S$):
    \begin{align*}
      &T(z) = R_S +(C_T+o(1))(\rho-z)^{1/\delta},
      \ \text{which implies}\ \begin{cases}
        R_S-T(z)=-(C_T+o(1))(\rho-z)^{1/\delta};\\
        T'(z)=-(\tfrac{C_T}{\delta}+o(1))(\rho-z)^{1/\delta-1};
      \end{cases}\\
       &\Occ_{\theta}(w)= (C_{\theta}+o(1)) (R_S-w)^{\delta-k}.
    \end{align*}
    Using \Cref{prop:Dec_TtS}, we can find the singular expansion of $T_{\patterntree,\Internal{\patterntree}}$:
	\begin{align*}
        T_{\patterntree,\Internal{\patterntree}}(z)
	&= z^k (T')^{|E|+1} \prod_{v} \Occ_{\theta_v}(T)\\
        &= \left[\rho^k \left(\frac{-C_T}{\delta}\right)^{|E|+1} \prod_{v\in \Internal{\patterntree}} C_{\theta_v} (-C_T)^{\delta - d_v} +o(1) \right]\,(\rho-z)^{1/\delta - k}.
	\end{align*}
    From the Transfer Theorem (\cref{thm:transfert}) applied to $T$ and $T_{\patterntree}$ we deduce
	\begin{align} 
	\proba(\bm t^{(n)}_{\bm I} = t_0)
	&= \frac{ \Gamma(-1/\delta) \rho^{-n +(1/\delta - k)}n^{-(1/\delta - k+1)} \left[\rho^k \left(\frac{-C_T}{\delta}\right)^{|E|+1} \prod_{v\in \Internal{\patterntree}} C_{\theta_v} (-C_T)^{\delta - d_v} +o(1) \right] }
	{\binom n k \Gamma(k-1/\delta) \rho^{-n +1/\delta}n^{-(1/\delta +1)} [C_T +o(1)] }\nonumber \\
	&= \frac{-k! \Gamma(-1/\delta)}{\Gamma(k-1/\delta)} \left[ \frac{(-C_T)^{|E|}}{\delta^{|E|+1}}\prod_{v\in \Internal{\patterntree}} C_{\theta_v} (-C_T)^{\delta - d_v} +o(1) \right] \nonumber \\
	&= \frac{-k!\Gamma(-1/\delta)}{\Gamma(k-1/\delta)\delta^k}
	\prod_{v\in \Internal{\patterntree}} \frac {1}\delta C_{\theta_v} (-C_T)^{\delta} + o(1).\label{eq:intermediaireThMain3}
	\end{align}
	The recursive property of the Gamma function gives 
    $\frac{ -k! \Gamma(-1/\delta)}{\Gamma(k-1/\delta)\delta^k} = \frac {k!}{(\delta -1) \cdots ((k-1)\delta - 1)}$. 
	Furthermore, by definition of $\Delta_{\theta_v}$, relation \eqref{eq:relOccS} and singular differentiation of $S$, we get
	\[C_{\theta_v} = \Delta_{\theta_v} \frac  {C_{S^{(d_v)}}}{d_v!} = \Delta_{\theta_v}  \frac{C_S}{d_v!} (-1)^{d_v}\delta(\delta - 1) \ldots (\delta - d_v + 1) = \Delta_{\theta_v} \frac{C_S}{d_v!} \delta(\delta-1)(2-\delta)\cdots(d_v -1 - \delta).\]
	This allows us to rewrite \eqref{eq:intermediaireThMain3} as
    \begin{multline}
      \lim_{n \to \infty}\proba(\bm t^{(n)}_{\bm I} = t_0)\\
      =\tfrac {k!}{(\delta -1) \cdots ((k-1)\delta - 1)}
      \prod_{v\in \Internal{\patterntree}} \frac{(\delta-1)(2-\delta)\cdots(d_v -1 - \delta)}{d_v!} \Delta_{\theta_v}C_{S} (-C_T)^{\delta}. 
      \label{eq:Tech13}
    \end{multline} 
	Now the proof of \cref{lem:Inversion3} (see \cref{eq:Tech8}) yields $-C_{T_\nonp} = C_{\Lambda}^{-1/\delta}$.
    We have also the following relations between the various constants
    (see \cref{eq:CT,eq:CLa} in the appendix):
    \begin{equation}
      C_\Lambda=\frac{C_S}{(1-R_\Lambda)^{2\delta}}\ \text{  and  } \ C_T = \frac{C_{T_\nonp}  }{(1-R_\Lambda)^2}.
      \label{eq:CLa_CT}
    \end{equation}
 From there we deduce $-C_T = C_S^{-1/\delta}$.
 Therefore \cref{eq:Tech13} rewrites
    \begin{equation}
      \lim_{n \to \infty}\proba(\bm t^{(n)}_{\bm I} = t_0) 
      =\tfrac {k!}{(\delta -1) \cdots ((k-1)\delta - 1)}
       \prod_{v\in \Internal{\patterntree}} \frac{(\delta-1)(2-\delta)\cdots(d_v -1 - \delta)}{d_v!} \Delta_{\theta_v}.
      \label{eq:Tech14}
 \end{equation} 
 Summing over trees $t_0$ with $\perm(t_0)=\pi$ gives the theorem.
\end{proof}

\bigskip

\subsection{The case $\delta>2$.}\label{ssec:deltaPlusGrand}
    \begin{theorem}\label{Th:Main3}
    	Let $\mathcal S$ be a family of simple permutations verifying hypotheses $(H3)$ and $(CS)$, with $\delta >2$. If $\Si_n$ is a uniform permutation in $\langle \mathcal{S}\rangle_n$, then $\Si_n$ converges in distribution to the biased Brownian separable permuton of parameter $p$, where
    	\begin{align*}
    	p &= \frac{ (1+R_S)^3\Occ_{12} (R_S)  +1}
    	{(1+R_S)^3(\Occ_{12} (R_S)+\Occ_{21} (R_S)) +2}.
    	\end{align*}
    \end{theorem}
    \begin{remark}
      While the limiting permuton in this case is independent of $\delta>2$ and is the same as in the standard case,
      the fine details of this convergence might be different.
      In particular, if $\pi$ is a nonseparable pattern,
      the order of magnitude of $\esper[\occ(\pi,\bm\si_n)]$ depends on $\delta$
      and is in general bigger than in the standard case
      -- compare \cref{eq:exp_Case3b} to \cref{prop:momentsnonseparables}.
    \end{remark}
	\begin{proof}
		Let $(\patterntree,V_s)$ be a decorated tree with $k$ leaves. Once again, applying \cref{thm:transfert} to $T$ and $T_{\patterntree,V_s}$ leads to
		\[\frac{[z^n]T_{\patterntree,V_s}}{\binom n k [z^n]T}=(\mathrm{cst}+o(1))n^{e_{(\patterntree,V_s)}},\]
		But in this case, 
		\begin{multline}
		e_{(\patterntree,V_s)}= \tfrac{1}{2} -k - \widetilde{e}_{(\patterntree,V_s)} =
		\tfrac{1}{2} +  - k + \tfrac{1}{2}(|E|+1) 
		+ \tfrac{1}{2}\sum_{v \in V_s; d_v > \delta} (d_v-\delta) \\
		\le \tfrac{1}{2}|E|+1  -k +  \tfrac{1}{2}\sum_{v \in \Internal{\patterntree}} (d_v-2)
		= |E|-|\Internal{\patterntree}|-k+1 =0.
		\label{eq:exp_Case3b}
		\end{multline}
		The above inequality is justified as follows:
$$
\sum_{v \in V_s; d_v > \delta} (d_v-\delta)\leq \sum_{v \in V_s; d_v > \delta} (d_v-2)\leq \sum_{v \in \Internal{\patterntree}} (d_v-2).
$$
The first inequality is an equality if and only if $d_v\leq \delta$ for all $v \in V_s$ (recall that $\delta >2$). 
In the second part, the equality case occurs when for all $v$ in $\Internal{\patterntree}$, either $d_v >\delta$ or $d_v=2$.
This implies that if $\patterntree$ is not binary, \cref{eq:exp_Case3b} is a strict inequality (regardless of $V_s$) and
		\begin{equation}
		\proba(\bm t^{(n)}_{\bm I}=t_0) =o(1).
		\label{eq:Cas7.2-nonbinaire}
		\end{equation}
		
		We can show that, up to replacing $\tau$ with $R_\Lambda$ and $\kappa$ with $R_S$, the estimates of the singular parts of $T$,$T_\nonp$,$T'$, $T^+$, \ldots, $T^-_{\nonm}$ in \cref{lem:MMCversionSerie,prop:DevMarkedleaves} still hold. Indeed, the proofs can be transposed verbatim up to replacing some calls to the standard case of
		\cref{lem:Comp} with the critical case.
		
		\cref{prop:Asymp_Tt} does not however hold in its generality anymore, because $\Occ_{\theta}$ is not necessarily convergent at $\tfrac{\tau}{1-\tau}$ for every $\theta$.  It is nonetheless convergent when $|\theta|=2$, since the singularity of $\Occ_{\theta}$ is in $\delta - 2$. This is enough to show that \cref{prop:Asymp_Tt} still holds for binary trees. Moreover nonbinary trees still disappear in the limit according to \cref{eq:Cas7.2-nonbinaire}. This allows us to conclude as in \cref{Sec:Standard}.
	\end{proof}

\appendix %
\section{Complex analysis toolbox} \label{sec:complex_analysis}

\subsection{Aperiodicity and Daffodil Lemma}
\label{ssec:aperiodicity}
To study the asymptotic behavior of combinatorial generating functions,
it is important to locate dominant singularities.
The following lemma is useful to this purpose.

Recall that a function $A$ analytic at $0$ is {\em aperiodic} if there do not
exist two integers $r\geq 0$ and $d\geq 2$ and a function $B$ analytic
at $0$ such that $A(z)=z^rB(z^d)$.

\begin{lemma}[Daffodil Lemma] \label{lem:daffodil}
Let $A$ be a generating function (with nonnegative coefficients) analytic in $|z|<R_A$.
If $A$ is aperiodic, then $|A(z)|<A(|z|) \le A(R_A)$ for $|z| \le R_A$ and $z \neq |z|$.
(The case $|z|=R_A$ can only be considered if $A(R_A)<\infty$.)
\end{lemma}
This lemma can be found in \cite[Lemma IV.1, p. 266]{Violet}.
Note that this reference does not consider the case of $z$ on the circle of convergence,  \emph{i.e.} $|z|=R_A$
(although this case is used later in the book,  \emph{e.g.} in the proof of Theorem VI.6, p.~405);
the proof of the lemma in this case is similar to $|z|<R_A$.
The complete statement of Daffodil Lemma in \cite{Violet}
also deals with cases where the function $A$ is periodic,
but we do not need these cases in our work.

\subsection{Transfer theorem}

We start by defining the notion of $\Delta$-domain. We use $\Arg(z)$ for the principal determination of the argument of $z$ in $\mathbb{C}\setminus \mathbb{R}^-$ taking its values in $(-\pi,\pi)$.
\begin{definition}[$\Delta$-domain and $\Delta$-neighborhood]\label{Def:DeltaDomaine}
A domain $\Delta$ is a {\em $\Delta$-domain at $1$} if there exist two real numbers  $R>1$ and $\pi/2<\phi <\pi$ such that
$$
\Delta=\{z \in \mathbb{C} \mid |z|<R,\, z\neq 1,
|\Arg(1-z)|<\phi\}.
$$ 
By extension, for a complex number $\rho \neq
0$, a domain is a {\em $\Delta$-domain at $\rho$} if it the image by
the mapping $z\rightarrow \rho z$ of a $\Delta$-domain at $1$.  A
{\em $\Delta$-neighborhood} of $\rho$ is the intersection of a
neighborhood of $\rho$ and a $\Delta$-domain at $\rho$.
\end{definition}
We will make use of the following family of $\Delta$-neighborhoods: for $\rho \neq 0 \in \mathbb C$, $0<r<|\rho|$, $\varphi >\pi/2$, set $\Delta(\varphi,r,\rho) = \{z\in \mathbb C, |\rho-z|<r, |\Arg(\rho - z)|<\varphi\}$.

When a function $A$ is analytic on a $\Delta$-domain at some $\rho$,
the asymptotic behavior of its coefficients is closely related to
the behavior of the function near the singularity $\rho$.
The following theorem is a corollary of \cite[Theorem VI.3 p.~390]{Violet}.

\begin{theorem}[Transfer Theorem]\label{thm:transfert}
  Let $A$ be a function analytic on a $\Delta$-domain $\Delta$ at $R_A$, $\delta$ be an arbitrary real number in $\mathbb{R} \setminus \mathbb{Z}_{\geq 0}$ 
  and $C_A$ a constant possibly equal to $0$.
  
  Suppose $A(z) = (C_A + o(1))(1-\tfrac{z}{R_A})^{\delta}$ when $z$ tends to $R_A$ in $\Delta$.
  Then the  coefficient of $z^n$ in $A$ satisfies 
$$[z^n]A(z) = (C_A +o(1)) \frac{1}{R_A^{n}} \ \frac{n^{-(\delta+1)}}{\Gamma(-\delta)}.$$
\end{theorem}

\subsection{Singular differentiation}
The next result is also useful to us.
\begin{theorem}[Singular differentiation]\label{thm:singular_diff}
  Let $A$ be an analytic function in a $\Delta$-neighborhood of $R_A$ with the following singular expansion near its singularity $R_A$ 
  $$A(z)=\sum_{j=0}^J C_j(R_A-z)^{\delta_j}+\O((R_A-z)^{\delta}),$$
  where $\delta_j, \delta \in \mathbb{C}$.

  Then, for each $k>0$, the $k$-th derivative $A^{(k)}$ is analytic in some $\Delta$-domain at $R_A$ and
  $$A^{(k)}(z)=(-1)^k \sum_{j=0}^J C_j \, 
  \delta_j (\delta_j-1) \cdots (\delta_j-k+1)\, (R_A-z)^{\delta_j-k}+\O((R_A-z)^{\delta-k}).$$
\end{theorem}
We refer the reader to  \cite[Theorem VI.8 p. 419]{Violet} for a proof of this theorem (this reference considers functions defined on a $\Delta$-domain, but the proof still works with functions defined on a $\Delta$-neighborhood).

\subsection{Exponents of dominant singularity}\label{sec:singularity}
In this section, we introduce some compact terminology and easy lemmas to keep track of the exponent $\delta$ of the singularities
and of the shape of the domain of analycity without computing the functions explicitly.

Recall that the radius of convergence $R_A$ of an analytic function $A$ is the modulus of the singularities closest to the origin, called \textit{dominant singularities}. Recall also that for series with positive real coefficients, by Pringsheim's theorem \cite[Th. IV.6 p. 240]{Violet}, $R_A$ is necessarily a dominant singularity. This justifies the following definition:

Let $\delta$ be a real, which is not an integer. 
We say that 
a series $A$ with radius of convergence $R_A$ \emph{has a dominant singularity of exponent $\delta$ in $R_A$}
(resp. \emph{of exponent at least $\delta$})
if $A$ has an analytic continuation on a $\Delta$-neighborhood $\Delta_A$
of $R_A$ and, on $\Delta_A$, we have
  \begin{equation}
A(z)= g_A(z) + (C_A +o(1)) \, (R_A - z)^\delta,
\label{eq:def_Exp}
\end{equation}
where $g_A(z)$ is an analytic function on a neighbourhood of $R_A$ (called the {\em analytic part}), and $C_A$ a nonzero constant 
(resp. any constant); $(C_A +o(1)) \, (R_A - z)^\delta$ is sometimes referred to as the {\em singular part}.

If furthermore, $A$ has no other singularity on the disk of convergence,
we say that it has a {\em unique} dominant singularity of exponent $\delta$ 
(resp. at least $\delta$) in $R_A$.
Since we assumed that $A$ has an analytic continuation on a $\Delta$-neighborhood $\Delta_A$
of $R_A$, by a standard compactness argument,
this is equivalent to say that $A$ can be extended to a $\Delta$-domain in $R_A$.

We make the following observation. 
According to the value of $\delta$, 
we may move (part of) $g_A(z)$ in the error term
and write \cref{eq:def_Exp} in a simpler form, still on a $\Delta$-neighborhood of $R_A$. 
\begin{itemize}
 \item For $\delta<0$, $g_A(z) = o((R_A - z)^\delta)$ so $A(z)= (C_A +o(1)) \, (R_A - z)^\delta$.
 \item For $0<\delta<1$, considering the constant term is the Taylor series expansion of $g_A(z)$ we find that
   $A(z)= A(R_A) + (C_A +o(1)) \, (R_A - z)^\delta$.
 \item Similarly, for $\delta>1$, we obtain 
   \[A(z)= A(R_A) + A'(R_A) (z-R_A) + \dots + (C_A +o(1)) \, (R_A - z)^\delta,\]
   in which the third dominant term (after the constant and the linear term)
   depends on how $\delta$ compares with $2$.
   But in each case, we have
\begin{equation}
    A(z)= A(R_A) + A'(R_A) (z-R_A) +\O( (R_A - z)^{\delta_*}),
    \label{eq:ErrorDeltaStar}
\end{equation}
where $\delta_*=\min(\delta,2)$.
\end{itemize}
  
We now record a few easy lemmas to manipulate these notions.
First consider the stability by product.
\begin{lemma}
  Let $F$ and $G$ be series with nonnegative coefficients
  and the same radius of convergence  $R=R_F=R_G \in (0,\infty)$.
  Assume they have each a dominant singularity of exponent $\delta_F$
  and $\delta_G$ respectively in $R$.
  Then $F \cdot G$ has a dominant singularity in $R$ of exponent $\delta$ defined by
  \begin{itemize}
    \item $\delta=\delta_F + \delta_G$ if both $\delta_F$ and $\delta_G$ are negative;
    \item $\delta=\min(\delta_F,\delta_G)$ otherwise.
  \end{itemize}
  \label{lem:MultExp}
  Moreover, if both $F$ and $G$ have unique dominant singularities, so has $F \cdot G$.
\end{lemma}
\begin{proof}
  The proof is easy.
  The analytic function $F \cdot G$ can be extended to the intersection
  of the domain of $F$ and $G$.
  The exponent of the singular expansion around $R$ is obtained
  by multiplying singular expansion of $F$ and $G$: note that, if $\delta_F$ is negative,
  the series $F$ is divergent and the singular part is the dominant part around $R$.
  On the opposite, when $\delta_F$ is positive, the dominant part of the expansion is the value $F(R)$
  of the analytic part at point $R$, which is always positive, since the series has nonnegative coefficients.
  The same holds of course for $G$, which explains the case distinction in the lemma.
\end{proof}

We now consider the composition $F \circ G$.
We should differentiate cases where $G(R_G)>R_F$, $G(R_G)<R_F$ or $G(R_G)=R_F$
(called sometimes supercritical, subcritical and critical cases \cite[Sec.VI.9]{Violet}).

\begin{lemma}[Dominant singularity of $F\circ G$]\label{lem:Comp}
Let $F$ and $G$ be series with nonnegative coefficients with 
radii of convergence $R_F,R_G$ in $(0,\infty)$.\\
\noindent{\bf Supercritical case: } Assume that $G(0)<R_F<G(R_G)$ ($G(R_G)$ may be finite or infinite).\\
Call $\rho < R_G$ the unique positive number with $G(\rho)=R_F$.

We assume that $F$ has a dominant singularity of exponent $\delta_F$ in $R_F$.
Then:
\begin{enumerate}
\item $F \circ G$ has also a dominant singularity of exponent $\delta_F$ in $\rho$.
\item Moreover, if $G$ is aperiodic, then the dominant singularity of $F \circ G$ is unique.
\end{enumerate}

\noindent{\bf Subcritical case: } Assume that $G(R_G)<R_F$.\\ %
We assume that $G$ has a dominant singularity of exponent $\delta_G$ in $R_G$. Then:
\begin{enumerate}
  \item $F \circ G$ has also a dominant singularity of exponent $\delta_G$ in $R_G$.
  \item Moreover, if the dominant singularity of $G$ is unique, 
  then the dominant singularity of $F \circ G$ is unique.
\end{enumerate}

\noindent{\bf Critical case-A: } Assume that $G(R_G)=R_F$.\\ %
We assume that $F$ and $G$ both have a dominant singularity of respective exponents $\delta_F$ and $\delta_G$.
Suppose furthermore \bm{$\delta_G>1$}.
Then:
\begin{enumerate}
  \item $F \circ G$ has also a dominant singularity of exponent $\min(\delta_G,\delta_F)$ in $R_G$.
  \item Moreover, if $G$ is aperiodic,  then the dominant singularity of $F \circ G$ is unique.
\end{enumerate}

\noindent{\bf Critical case-B:}
Assume again that $G(R_G)=R_F$.
We assume that $F$ and $G$ both have a dominant singularity of respective exponents $\delta_F$ and $\delta_G$.
Suppose furthermore \bm{$\delta_G \in (0,1)$}.
Then:
\begin{enumerate}
  \item $F\circ G$ has a dominant singularity of exponent $\min(\delta_F,1)\delta_G$ in $R_G$. 
  \item Moreover, if $G$ is aperiodic, then the singularity is unique.
\end{enumerate}
\end{lemma}
\begin{proof}
\noindent{\bf Supercritical case: }
  It is clear that $F \circ G$ is analytic around any $r \in [0,\rho)$
  and has nonnegative coefficients,
  hence it has radius of convergence at least $\rho$.

  To show that $F \circ G$ is defined in a $\Delta$-neighborhood $\Delta$ of $\rho$,
  we show that $G(\Delta)$ is included in $\Delta_F$.
  This follows easily from the fact that $G$ is analytic in $\rho$ 
  and has a derivative $G'(\rho)$ which is a positive real number.

When $z$ is close to $\rho$, plugging $G(z)$ in the expansion \eqref{eq:def_Exp} of $F$ we obtain
\begin{equation}\label{eq:FGdef_Exp}
F(G(z))=g_F(G(z))+ (C_F+o(1))(R_F-G(z))^{\delta_F}.
\end{equation}
The first term $g_F(G(z))$ is analytic at $\rho$. Since $G(\rho)=R_F$ and $G$ is differentiable at $\rho$ we have
\begin{equation}\label{eq:TaylorG}
R_F-G(z)=(G'(\rho)+o(1))(\rho-z).
\end{equation}
Combining these two expansions yields
\begin{equation}\label{eq:TaylorFG}
F(G(z))=g_F(G(z))+ (C_FG'(\rho)^{\delta_F}+o(1))(\rho-z)^{\delta_F},
\end{equation}
which proves i).

  Item ii) is also easy. In the case where we assume $G$ aperiodic,
  we need \cref{lem:daffodil},
  which ensures that $|G(\zeta)|<R_F$ for $|\zeta| \leq |\rho|$, $\zeta \neq \rho$.
\medskip

\noindent{\bf Subcritical case.} 
Most arguments are similar to the ones of the supercritical case. 
Therefore we only explain the differences in the singular expansion of $F(G(z))$.
Using the singular expansion \eqref{eq:def_Exp} of $G$, we have
\[F(G(z))=F\big[g_G(z)+(C_G+o(1))(R_G-z)^{\delta_G}\big].\]
Since $G(R_G)<R_F<+\infty$, the exponent $\delta_G$ is positive and the term $(R_G-z)^{\delta_G}$
tends to $0$ at $R_G$. Both $G(z)$ and $g_G(z)$ tend to $G(R_G)$ as $z \to R_G$,
so that, by standard calculus arguments, we have
\begin{align}
  F(G(z))&=F(g_G(z))+F'(G(R_G)) (C_G+o(1))(R_G-z)^{\delta_G} +o\big((R_G-z)^{\delta_G}\big)\nonumber\\
  &=F(g_G(z))+\big(C_G\, F'(G(R_G)) +o(1) \big) (R_G-z)^{\delta_G}. \label{eq:Exp_FG_SubCritical}
\end{align}
Since $F$ and $g_G$ are analytic at $G(R_G)$ and $R_G$ respectively,
this expansion is of the desired form.
\medskip

\noindent{\bf Critical case-A.}
As above, we focus on the expansion of $F(G(z))$.
Since $\delta_G >1$, $G$ is differentiable at $\rho=R_G$ and \cref{eq:TaylorFG} still holds.
The difference is that $g_F(G(z))$ is not analytic anymore. Namely, when $z$ is close to $\rho$,
\begin{equation}\label{eq:Expansion_gFG}
g_F(G(z))=g_F(g_G(z))+g_F'\left(g_G(R_G)\right)(C_G+o(1))(R_G-z)^{\delta_G}.
\end{equation}
Then
$$
F(G(z))=g_F(g_G(z))+ \big(g_F'\left(g_G(R_G)\right)C_G+o(1)\big)(R_G-z)^{\delta_G}+ (C_FG'(\rho)^{\delta_F}+o(1))(\rho-z)^{\delta_F}.
$$
Since $g_F(g_G(z))$ is analytic at $\rho$, the  exponent of the dominant singularity of $F\circ G$ is $\min(\delta_F,\delta_G)$. Note that the singular terms cannot cancel each other since when $\delta_F=\delta_G$ the constants have the same sign.
\medskip

\noindent{\bf Critical case-B.}
Again, we focus on the singular expansion of $F(G(z))$. Now, since $\delta_G <1$, $G$ is not differentiable at $\rho=R_G$. Instead of \eqref{eq:TaylorG} we have
$$
R_F-G(z)=-(C_G+o(1))(\rho-z)^{\delta_G}.
$$
Eq. \eqref{eq:TaylorFG} becomes
$$
F(G(z))=g_F(G(z))+ (C_F(-C_G)^{\delta_F}+o(1))(\rho-z)^{\delta_F\delta_G}.
$$
(In this case, $C_G$ must be negative, as can be observed by writing the transfer theorem for the coefficients of $G$ which are non-negative by assumption.) 
As for $g_F(G(z))$, \eqref{eq:Expansion_gFG} still holds. We obtain
$$
F(G(z))=g_F(g_G(z))+ \big(g_F'\left(g_G(R_G)\right)C_G+o(1)\big)(R_G-z)^{\delta_G}+(C_F(-C_G)^{\delta_F}+o(1))(\rho-z)^{\delta_F\delta_G}.
$$
We conclude that the exponent of the dominant singularity is $\min(\delta_F,1)\delta_G$.
\end{proof}

We note that the above proof also yields the constant $C_{F \circ G}$ 
appearing in the singular expansion of $F \circ G$. 
The two following particular cases were used in \cref{sec:critical}, p. \pageref{eq:CLa_CT}
(in particular, we assume in the discussion below that hypothesis $(H3)$ holds).
\begin{itemize}
  \item We take $F(z)=S(z)$ and $G(u)=\frac{u}{1-u}$. The composition is $F\circ G(u)=\Lambda(u)-\tfrac{u^2}{1-u}$,
    and $\tfrac{u^2}{1-u}$ is analytic at $R_\Lambda<1$.
    Since $\frac{u}{1-u}$ diverges at its singularity and $S$ has a finite radius of convergence,
    the composition is supercritical.
    Extracting the constant from \eqref{eq:TaylorFG}, we get
    \begin{equation}
        \label{eq:CLa}
        C_{\Lambda}=C_{F \circ G}= C_F\, G'(R_{F \circ G})^{\delta_F} = \frac{C_S}{(1-R_\Lambda)^{2 \delta}},
    \end{equation}
    since in this case $C_F=C_S$, $\delta_F=\delta$ and $R_{F \circ G}=R_\Lambda$.
  \item We take $F(u)=\frac{u}{1-u}$ and $G(z)=T_{\nonp}(z)$. The composition is 
    $T(z)=\frac{T_{\nonp}(z)}{1-T_{\nonp}(z)}$ (\cref{Prop:systeme1}).
    Since $G(R_G)=T_{\nonp}(\rho)=R_\Lambda<1=R_F$, (here we use Hypothesis (H3) and $\rho$ is defined in \cref{lem:Inversion3}) the composition is subcritical.
    Extracting the constant from \eqref{eq:Exp_FG_SubCritical},
    we get 
    \begin{equation}
      C_T=C_G\, F'(G(R_G))= C_{T_\nonp}\, F'(T_{\nonp}(\rho))= \frac{C_{T_\nonp} }{(1-R_\Lambda)^2}.
      \label{eq:CT}
    \end{equation}
\end{itemize}
\bigskip

Finally, we state the following result, which follows from
\cref{thm:singular_diff}.

\begin{lemma}[Singular differentiation]
If $F$ has a (unique) dominant singularity of exponent (at least) $\delta$ in $\rho$,
then its $k$-th derivative $F^{(k)}$ has a (unique) dominant singularity of exponent (at least) $\delta-k$ in $\rho$.
\label{lem:singular_diff}
\end{lemma}

\subsection{An analytic implicit function theorem}
The following theorem allows to locate the dominant singularity of
series defined by an implicit equation.
\begin{lemma}[Analytic Implicit Functions]\label{lem:implicit}
  Let $F(z,w)$ be a bivariate function analytic at $(z_0, w_0)$, we denote $F_w=\tfrac{\partial F}{\partial w}$.
 If  $F(z_0, w_0) = 0$ and $F_w(z_0, w_0) \neq 0$, then there exists a unique function $\phi(z)$ analytic in a neighbourhood of $z_0$ such that $\phi(z_0) = w_0$ and $F(z, \phi(z)) = 0$.
\end{lemma}
We refer the reader to  \cite[Lemma VII.2, p. 469]{Violet} for a proof of this result.
\subsection{Proof of \cref{lem:Analyse_Tnonp2}}
\label{sec:proof_Inv2}
  Let $R_{\nonp}$ be the radius of convergence of $T_{\nonp}$.
  If $T_{\nonp}(R_{\nonp}) \ge R_\Lambda$, then 
  by intermediate value theorem, we know that there exists $\rho$
  as in the lemma.

  We will prove this by contradiction.
 Assume $T_{\nonp}(R_{\nonp}) < R_\Lambda$.
 We apply \cref{lem:implicit}.
  The bivariate function we consider is $(z,w) \mapsto z-w+\Lambda(w)$.
  It vanishes at $(R_{\nonp},T_{\nonp}(R_{\nonp}))$ and the derivative
  with respect to $w$ at that point is nonzero since
  \[ \Lambda'(T_{\nonp}(R_{\nonp})) < \Lambda'(R_\Lambda)<1.\]
(Recall that this last inequality is equivalent to \eqref{eq:Hyp_Sp} in Hypothesis (H2).)\\
  Therefore, $T_{\nonp}$
  has an analytic continuation on a neighborhood of $R_{\nonp}$.
  Since it has positive coefficients, by Pringsheim's theorem \cite[Th. IV.6 p. 240]{Violet}, this is in contradiction
  with the fact that $R_{\nonp}$ is the radius of convergence of $T_{\nonp}$.

  We have therefore proved that there exists $\rho\leq R_{\nonp}$
  such that $T_{\nonp}(\rho)=R_\Lambda$. Note that it implies the relation $\rho = R_\Lambda - \Lambda(R_\Lambda)$.
\bigskip

  We now consider $T_{\nonp}$ around $z=\rho$.
  Equation \eqref{eq:Tnonp2} defining $T_{\nonp}(z)$ can be rewritten as 
  $T_{\nonp}(z)=G(z,T_{\nonp}(z))$, where
  \[G(z,w)= w+\frac{1}{1-\Lambda'(R_\Lambda)}(z-w+\Lambda(w)).  \]
  Since $\Lambda$ has a dominant singularity of exponent $\delta>1$ in $R_\Lambda$,
  Equation \eqref{eq:ErrorDeltaStar}, together with elementary computations, yield the following:
  for $w$ in a $\Delta$-neighborhood $D_\Lambda$ of $R_\Lambda$,
  \begin{equation}
    G(z,w)= R_\Lambda +\frac{ z-\rho}{1-\Lambda'(R_\Lambda)} + \O( (R_\Lambda-w)^{\delta_*}).
    \label{eq:Tech1}
  \end{equation}
  We now use Picard's method of successive approximants to show the existence and analycity of $T_{\nonp}$
  on a $\Delta$-neighborhood $D_T$ of $\rho$.
  We refer to \cite[Appendix B.5 p. 753]{Violet} for a synthetic description of the method in the case 
  where $\Lambda$ is analytic in $R_\Lambda$; we have to adapt it carefully to our setting.

  Define $\phi_0(z)=R_\Lambda$ and $\phi_{j+1}(z)=G(z,\phi_j(z))$ whenever $\phi_j(z)$ is in $D_\Lambda$.
  We have $\phi_1(z)-\phi_0(z)=\tfrac{z-\rho}{1-\Lambda'(R_\Lambda)}$.
  Also, \cref{thm:singular_diff} of singular differentiation %
  applied to \cref{eq:Tech1}
  implies that
  \[\tfrac{\partial  G(z,w)}{\partial w}= \O( (R_\Lambda-w)^{\delta_*-1}).\]
  Therefore\footnote{There is a slight subtlety here: we would like to apply the classical inequality
  $|f(w)-f(w')| \le \|f'\|_{\infty} |w - w'|$, but this is not possible since the domain $D_\Lambda$ is
  not convex. Note however that a $\Delta$-neighborhood $D$ is always \emph{a quasi-convex set}, in the sense that
  we can always find a path between $w$ and $w'$ whose length is smaller than $K|w-w'|$,
  where $K$ depends on the angle defining $D$ but not on $w$ and $w'$. Therefore the following weaker inequality holds: $|f(w)-f(w')| \le K \|f'\|_{\infty} |w - w'|$, which is good enough for our purpose (the constant $K$ disappears in the $\O$ symbol).}
  for $j \ge 1$, if $\phi_j(z)$ and $\phi_{j+1}(z)$ are defined and lie in $D_\Lambda$, we have 
  \[\phi_{j+1}(z)-\phi_j(z) = \O\big(\eta^{\delta_*-1}|\phi_{j}(z)-\phi_{j-1}(z)|  \big),\]
  where $\eta=\sup_{w \in D_\Lambda} |R_\Lambda-w|$.
  Fix $\eps>0$. Up to reducing the radius of $D_\Lambda$, we can therefore assume that 
  \begin{equation}
    |\phi_{j+1}(z)-\phi_j(z)| \le \eps |\phi_{j}(z)-\phi_{j-1}(z) |.
    \label{eq:Tech2}
  \end{equation}
  Thus, if $\phi_j(z)$ is in $D_\Lambda$ for every $i \le m$, then $\phi_M(z)$ is defined and we have
  \begin{equation}
    |\left(\phi_{M}(z)-R_\Lambda\right) - \tfrac{z-\rho}{1-\Lambda'(R_\Lambda)} | = |\phi_{M}(z)-\phi_1(z)| \le \frac{\eps}{1-\eps} |\phi_1(z)-\phi_0(z)| =  \frac{\eps}{1-\eps} \, \left|\tfrac{z-\rho}{1-\Lambda'(R_\Lambda)}\right|.
    \label{eq:Tech3}
  \end{equation}
  If we take $\eps$ small enough, the argument of $\phi_{M}(z)-R_\Lambda$ is close to the one of $z-\rho$.
  Furthermore if the modulus of $z-\rho$ is small so is the one of $\phi_{M}(z)-R_\Lambda$.
  This ensures the existence of a $\Delta$-neighborhood $D_T$ of $\rho$ (not depending on $M$ and $z$),
  such that for $z\in D_T$ and $M\geq 1$, $\phi_{M}(z)$ is in $D_\Lambda$ as long as it is defined. 
  In particular, $\phi_{M+1}(z)$ is also defined and by immediate induction, 
  all $\phi_j$ are defined and analytic on $D_T$.

  \cref{eq:Tech2} also implies that $\phi_j$ converges locally uniformly on $D_T$.
  The limit is the unique solution $w$ in $D_\Lambda$ of the fixed point equation $w=G(z,w)$
  (the uniqueness of the solution comes from the fact that for every $z\in D_T$, $w \mapsto G(z,w)$
  is a contraction for $w$ in $D_\Lambda$).
  This limit is therefore an analytic continuation of $T_{\nonp}(z)$ to $D_T$.
  Note also that from \cref{eq:Tech3}, the following estimate holds on $D_T$:
  \[T_{\nonp}(z)-R_\Lambda=\tfrac{z-\rho}{1-\Lambda'(R_\Lambda)} + o(|z-\rho|).\]
\medskip

 Using the expansion given in \cref{eq:def_Exp} of $\Lambda$ around $R_\Lambda$,
  we have for $z\in D_T$,
  \[T_{\nonp}(z)=\rho + (z-\rho) +
  g_\Lambda(T_{\nonp}(z)) + (C_\Lambda +o(1))\, (T_{\nonp}(z)-R_\Lambda)^\delta.\]
  As $\id - g_\Lambda$ is analytic at $R_\Lambda$ with a nonzero derivative $1-\Lambda'(R_\Lambda)$, it can be inverted analytically around $R_\Lambda$ by an analytic function $h_\Lambda$ and hence
  \[T_{\nonp}(z) = h_\Lambda \left(z+ (C_\Lambda +o(1))\, (T_{\nonp}(z)-R_\Lambda)^\delta \right)\]
 As $T_{\nonp}(z)-R_\Lambda=\frac{1+o(1)}{1-\Lambda'(R_\Lambda)}(z-\rho)$, it follows from the Taylor expansion of $h_\Lambda$ up to exponent $\lceil \delta \rceil$
 that $T_{\nonp}$ has a singularity of exponent exactly $\delta$ in $\rho$.
 In particular $T_{\nonp}$ has a singularity of exponent $\delta$ in $\rho$ and hence $\rho = R_\nonp$.\bigskip

  We now prove that $T_{\nonp}$ has no singularity $\zeta$ with $|\zeta| \le \rho$,
  except $\zeta=\rho$. By a classical compactness argument (see  \emph{e.g.}
  \cite[end of proof of Theorem 2.19]{Drmota}), this implies
  that $T_{\nonp}$ is analytic on a $\Delta$-domain at $\rho$.

  Take such a singularity.  Since $T_{\nonp}$ has nonnegative
  coefficients, the triangular inequality gives $|T_{\nonp}(\zeta)|
  \le T_{\nonp}(\rho)$ and since $T_{\nonp}(z)$ is aperiodic,
  from \cref{lem:daffodil}
  we have a strict
  inequality unless $\zeta=\rho$ .  Therefore, if $|\zeta| \le \rho$
  and $\zeta\neq \rho$, we have
  $|\Lambda'(T_{\nonp}(\zeta))|<\Lambda'(R_\Lambda)<1$ and we can
  apply \cref{lem:implicit} %
  as above as in the second paragraph of this proof to argue that
  $\zeta$ cannot be a singularity. \qed

\subsection{Proof of \cref{lem:Inversion3}}\label{sec:proof_Inv3}
  As in the proof of \cref{lem:Analyse_Tnonp2},
  the existence of $\rho$ and the fact that 
  the convergence  radius of $T_{\nonp}$ is at least $\rho$ is straightforward.
  The key point is to prove that $T_{\nonp}$ has an analytic continuation
  to a $\Delta$-neighborhood of $\rho$.

  By assumption, $\Lambda$ is analytic on
  a $\Delta$-neighborhood $D_\Lambda=\Delta(\varphi_\Lambda,r_\Lambda,R_\Lambda)$ of $R_\Lambda$, and the following approximation holds:
\[\Lambda(w) = \Lambda(R_\Lambda) - (R_\Lambda-w) 
+ C'_\Lambda (R_\Lambda-w)^{\delta_\star}(1+\eps(w)),\]
where as before, $\delta_*=\min(\delta,2)$;
$C'_\Lambda$ is $C_\Lambda$ or $\tfrac{1}{2}\Lambda''(R_\Lambda)$ depending on whether $\delta$ is smaller or bigger than $2$;
and $\eps(w)$ is an analytic function on $D_\Lambda$ tending to $0$ in $R_\Lambda$.

  Fix $z$ in a $\Delta$ neighborhood $D_T$ of $\rho$,
  whose parameters $r_T$ and $\varphi_T$ will be made precise later.
The equation $w=z + \Lambda(w)$ then rewrites as
\begin{equation}
  \rho-z=C'_\Lambda (R_\Lambda-w)^{\delta_*} (1 +\eps(w)),
  \label{eq:Tech4}
\end{equation}
or, as a fixed point equation $w=G(z,w)$ for
\[G(z,w):= R_\Lambda-\left( \tfrac{1}{C'_\Lambda} (\rho-z)
\cdot \frac{1}{1 +\eps (w)}\right)^{1/\delta_*}. 
\]

We again use Picard's method of successive approximants to find
an analytic solution $w(z)$ for \eqref{eq:Tech4},
which will be the analytic continuation of $T_{\nonp}(z)$ that we are looking for.
For $z\in D_T$, set $\phi_0(z)=R_\Lambda$ and, whenever $\phi_i(z)$ lies in $D_\Lambda \cup \{R_\Lambda\}$,
set $\phi_{i+1}(z)=G(z,\phi_i(z))$.
In particular, 
\[R_\Lambda - \phi_1(z) = \left( \tfrac{1}{C'_\Lambda} (\rho-z) \right)^{1/\delta_*}.\]
Since $1/\delta_*<1$, we have $\Arg(R_\Lambda -\phi_1(z))=\tfrac{1}{\delta_*}\Arg(\rho-z)$.
We choose the parameters defining the $\Delta$-neighborhood $D_T$ of $\rho$
to be $\varphi_T=\varphi_\Lambda$ and $r_T=C'_\Lambda \,(\tfrac{ r_\Lambda}{2})^{\delta_*}$.
In this way, if $z$ is in $D_T$, then
then
$\phi_1(z)$ lives in
$\widetilde{D_\Lambda}=\Delta(\widetilde{\varphi_\Lambda},\tfrac{r_\Lambda}{2},R_\Lambda)$,
for some $\widetilde{\varphi_\Lambda}<\varphi_\Lambda$.

We define an intermediate $\Delta$-neighborhood $D'_\Lambda = \Delta(\frac {\varphi_\Lambda+\widetilde{\varphi_\Lambda}} 2,\tfrac{3 r_\Lambda}{4},R_\Lambda)$. This ensures that we have a constant $0<r_0<1$, depending only on $\varphi_{\Lambda}$ and $\widetilde{\varphi_{\Lambda}}$, such that the circle $\gamma_w$ of center $w$ and radius $r_0\, |R_\Lambda-w|$
is contained in $D_\Lambda$ for every $w \in D'_\Lambda$ and in $D'_\Lambda$ for every $w \in \widetilde{D_\Lambda}$ (cf. \cref{fig:DeltaDomaines}).
\begin{figure}[htbp]
	\centering
	\includegraphics{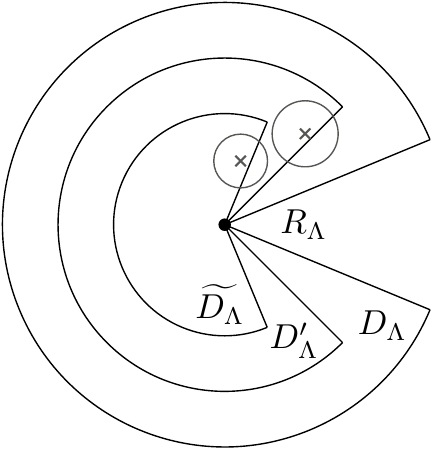}
	\caption{ Illustration of $D_\Lambda, D'_\Lambda, \widetilde{D_\Lambda}$, with two examples of circles $\gamma_w$ represented in gray.
	\label{fig:DeltaDomaines}}
\end{figure}

Consider the partial derivative
\begin{equation}
  \frac{\partial G}{\partial w}(z,w)=\frac{\eps'(w) }{\delta_*} \cdot \left( \frac{1}{C'_\Lambda} (\rho-z)
\right)^{1/\delta_*} \cdot \left(\frac{1}{1 +\eps (w)}\right)^{1/\delta_*+1}.
\label{eq:Tech5}
\end{equation}
We take $w$ in the domain $D'_\Lambda$. The quantity $\eps'(w)$ can now be evaluated through a contour integral on $\gamma_w\subset D_\Lambda$:
\[ \eps'(w) = \frac{1}{2 \pi i} \oint_{\gamma_w} \frac{\eps(u) du}{(u-w)^2}. \]
This yields the inequality 
\[ |\eps'(w)| = \O\left( \frac{\sup_{u \in D_\Lambda} |\eps(u)|}{|R_\Lambda-w|} \right).\]
Plugging this back in \cref{eq:Tech5},
we get, for $w$ in $D'_\Lambda$
\begin{equation}
\left|\frac{\partial G}{\partial w}(z,w)\right|= \O\left(\frac{|z-\rho|^{1/\delta_*} \cdot \sup_{u \in D_\Lambda} |\eps(u)|}{|R_\Lambda-w|} \right).
\label{eq:Tech5.5}
\end{equation}

Now we shall find a domain where we have enough control on $|\tfrac{\partial G}{\partial w}(z,w)|$ as to guarantee the stability of the iterates. A subtlety here is that this control is impossible near $\phi_0(z) = R_\Lambda$. So we need to consider a domain around $\phi_1(z)$, hence that depends on $z$. For every $z\in D_T$, we have $\phi_1(z) \in \widetilde{D_\Lambda}$, so the disk
\[\Gamma_z:=\{w: |w-\phi_1(z)| \le \tfrac{1}{r_0} |\phi_1(z)-R_\Lambda| \}\]
is included in $D'_\Lambda$.
For $w$ in $\Gamma_z$,
we have 
\[|R_\Lambda-w|=\Theta(|\phi_1(z)-R_\Lambda|) = \Theta\big( |\rho-z|^{1/\delta_*}\big),\]
which implies after plugging back into \cref{eq:Tech5.5}
\[\left|\frac{\partial G}{\partial w}(z,w)\right|= \O\left(\sup_{u \in D_\Lambda} |\eps(u)|\right).\]
By possibly reducing the radius $r_\Lambda$ of $D_\Lambda$, we can make $\sup_{u \in D_\Lambda} |\eps(u)|$
as small as wanted: for any $w$ in $\Gamma$,
\begin{equation}
  \left|\frac{\partial G}{\partial w}(z,w)\right| \le \frac{1}{r_0+1}.
  \label{eq:Tech6}
\end{equation}
Similarly,
\[|\phi_2(z)-\phi_1(z)|= \left( \tfrac{1}{C'_\Lambda} |\rho-z| \right)^{1/\delta_*} \cdot
\left| \left( \frac{1}{1+\eps(\phi_1(z))} \right)^{1/\delta_*} -1 \right| \]
can be made smaller than $\tfrac{1}{r_0+1} |\phi_1(z)-R_\Lambda|$ by reducing $r_\Lambda$.
In particular, $\phi_2(z)$ is in $\Gamma_z$.

For $m \ge 2$, assume that $\phi_1(z),\cdots,\phi_m(z)$ lie in $\Gamma_z$.
Then for each $i \le m$, using \eqref{eq:Tech6},
\begin{equation}
  |\phi_{i+1}(z)-\phi_i(z)| \le \left( \tfrac{1}{r_0+1} \right)
|\phi_i(z)-\phi_{i-1}(z)| \le \cdots \le \left( \tfrac{1}{r_0+1} \right)^{i-1} |\phi_2(z)-\phi_1(z)|.
\label{eq:Tech7}
\end{equation}
Since $\phi_m(z)$ lies in $\Gamma_z \subset D_\Lambda$,
the next term $\phi_{m+1}(z)$ is defined
and 
\[|\phi_{m+1}(z)-\phi_1(z)|\le \sum_{i=1}^m |\phi_{i+1}(z)-\phi_i(z)|
\le \left[ \sum_{i=1}^m \left( \tfrac{1}{r_0+1} \right)^{i-1} \right] |\phi_2(z)-\phi_1(z)|
\le \tfrac{1}{r_0} |\phi_1(z)-R_\Lambda|.\]
In particular, $\phi_{m+1}(z)$ also lies in $\Gamma_z$ and an immediate
induction shows that this is indeed the case for all $m \ge 1$.

By \eqref{eq:Tech7}, the series $\sum_{i \ge 0} \phi_{i+1}(z) -\phi_i(z)$
is uniformly bounded
by a geometric series and converges towards an analytic function $\phi$ on $D_T$.
The limit $\phi(z)$ is a solution of
$\phi(z)=z+\Lambda(\phi(z))$ and is the analytic continuation of $T_{\nonp}(z)$ that 
we were looking for.

A small modification of the above argument shows that,
when $r_T$, or equivalently $r_\Lambda$, tends to $0$,
the quotient
\[\frac{|\phi_{m+1}(z)-\phi_1(z)|}{|\phi_1(z)-R_\Lambda|}\]
also tends to $0$. This proves that
\begin{equation}
	\phi(z) - R_\Lambda= (\phi_1(z)-R_\Lambda) (1+o(1))
	=-\left( \tfrac{1}{C'_\Lambda} (\rho-z) \right)^{1/\delta_*}(1+o(1)).
	\label{eq:Tech8}
\end{equation}

The proof that $T_{\nonp}(z)$ has no other singularities than $\rho$
on the circle of convergence is similar to that of \cref{lem:Analyse_Tnonp2}.
\qed

\section{On the simulations given in the introduction}\label{sec:AppendicePermuton}

In this appendix, we explain how the simulations in \cref{Fig:SimusPermuton,Fig:SimusPermutonStable}
have been obtained.
\subsection{Biased Brownian permuton}
  Fix $p$ in $(0,1)$
  and consider  a uniform binary planar tree $\bm b_n^{(p)}$ with $n$ leaves, 
  where each internal node is labeled $\oplus$ (resp. $\ominus$) with probability $p$ (resp. $1-p$), independently from each other.
  As mentioned in \cref{ssec:BrownianPermuton},
  $\tau_n^{(p)}:=\perm(\bm b_n^{(p)})$ forms a consistent family of random permutations.
  Therefore, from \cref{Prop:existence_permuton}, we have the following lemma (using the notation $\Perm(.,.)$ defined in \cref{Sec:ExtractedPermutations})
\begin{lemma}\label{Lem:ConstructionElementaire}
  There exists a random permuton $\bm\mu^{(p)}$, 
  whose induced subpermutation are the $\tau_n^{(p)}$ ( \emph{i.e.}
  for all $n$, $\bm \tau_n^{(p)} \stackrel d = \Perm(\Mn,\Mu^{(p)})$)
  and
 we have the convergence in distribution
$$
\mu_{\bm \tau_n^{(p)}} \stackrel{n\to +\infty}{\to} \bm\mu^{(p)}.
$$
\end{lemma}
By definition, $\bm\mu^{(p)}$ is the biased Brownian separable permuton with parameter $p$
(indeed $\bm \tau_k^{(p)} \stackrel d = \Perm(\Mk,\Mu^{(p)})$ implies
\cref{eq:def_PermutonBiaise} of \cref{Def:PermutonBiaise}).
This lemma is not needed to prove the results of this paper,
but was used in the simulations.
The three pictures in \cref{Fig:SimusPermuton}
p.~\pageref{Fig:SimusPermuton} are obtained
by drawing the diagram of a random permutation distributed as $\bm \tau_n^{(p)}$,
for $p=0.2$ and $n=11\, 629$, $p=0.45$ and $n=12\, 666$, $p=0.5$ and $n=17\, 705$.

\subsection{Stable permutons}

Fix $\delta \in (1,2)$. For every $k$ the following probability distribution on unlabeled plane trees
with $k$ leaves was introduced in \cite[Thm 3.3.3]{DuquesneLeGall} 
and is the distribution of the induced subtree with $k$ leaves in the $\delta$-stable tree:
\[\rho_{\delta,k}(\patterntree) = \frac{k!}{(\delta -1) \cdots ((k-1)\delta - 1)}\prod_{v\in \Internal\patterntree} \mathbf{1}_{d_v\geq 2} \frac{(d_v - 1 - \delta) \cdots (2-\delta)(\delta - 1)}{d_v!}.\]
Now if we fix the distribution of a random permuton $\bm \nu$, for every $n\geq 1$, we build a random substitution tree $\bm t^{(\delta,\bm \nu)}_{n}$ as follows:
the tree is chosen according to $\rho_{\delta,n}$, and conditional on that choice, all internal nodes $v$ are independently labeled by a permutation distributed like $\Perm(\vec{\mathbf { m}}_{d_v}, \bm\nu)$  (the notation $\Perm(.,.)$ is defined in \cref{Sec:ExtractedPermutations}).

Now we can define the permutations $\bm \tau^{(\delta,\bm\nu)}_n = \perm(\bm t^{(\delta,\bm \nu)}_{n})$.
This family of permutations is consistent:
we omit the proof of this fact, which follows from the consistency of the family $(\Perm(\Mk,\bm \nu))_k$
(\cref{Prop:existence_permuton}) and from Marchal's algorithm \cite{Marchal} to generate trees of distribution $\nu_{\delta,k}$.
We deduce the following lemma.

\begin{lemma}\label{Lem:ConstructionElementaireStable}
	For every $\delta\in(1,2)$ and random permuton $\bm \nu$, there exists a random permuton $\bm \mu^{(\delta,\bm \nu)}$, 
    whose induced subpermutations are the $\bm \tau_n^{(\delta,\bm\nu)}$ ( \emph{i.e.}
	for all $n$, $\bm \tau_n^{(\delta,\bm\nu)} \stackrel d = \Perm(\Mn,\Mu^{(\delta,\bm \nu)})$)
	and we have the convergence in distribution
	$$
	\mu_{\bm \tau_n^{(\delta,\bm \nu)}} \stackrel{n\to +\infty}{\to} \bm\mu^{(\delta,\bm \nu)}.
	$$
	We call $\mu^{(\delta,\bm \nu)}$ the $\delta$-stable permuton driven by $\bm \nu$.
\end{lemma}

Once again, this lemma is used in our simulations.
The pictures in  \cref{Fig:SimusPermutonStable} are 
the rescaled diagrams of realizations of $\bm \tau_n^{(\delta,\bm \nu)}$ for $n=20\, 000$ and $\delta \in \{1.1,1.5\}$, 
where we have taken $\bm \nu$ to be the (nonrandom) uniform measure on $[0,1]^2$. 

\subsection{Simulations of permutations in classes}
The uniform random permutations in substitution-closed classes
shown on \cref{Fig:PermInClasses} have been obtained through a Boltzmann sampler.
Obtaining such a sampler is routine from the equations on the generating series
given in \cref{Prop:systeme1} \cite{BoltzmannSampler}.
The only input that we need is a Boltzmann sampler for the set $\SSS$ 
of simple permutations in the class.
In the case of a finite set $\SSS$, this is trivial.
For the set of simple permutations in $\Av(321)$, we built a Boltzmann sampler
for the whole class $\Av(321)$, which has an easy recursive structure,
and we run it until the output is simple
(this happens with probability bigger than $1/4$ when the parameter of the Boltzmann
sampler is close to the radius of convergence of $\CCC=\langle\Av(321)\rangle$).
Code is available on request.

\section*{Acknowledgements}
The authors are grateful to the referees for suggestions which improved the presentation of the paper.

MB and VF are partially supported by the Swiss National Science Foundation, under grants number 200021-172536 and 200020-172515.
LG's research is supported by ANR grants GRAAL (ANR-14-CE25-0014) and PPPP (ANR-16-CE40-0016).

\end{document}